\documentclass[english,12pt]{amsart}

\usepackage[utf8]{inputenc}
\usepackage[english]{babel}

\usepackage{fouriernc}%%%%%%%%%%%%%% mettre les 3 lignes
\usepackage[scaled=0.775]{helvet}
\usepackage{courier}

\usepackage{amstext,amsthm,amsmath,amssymb}
\usepackage{subfig}
\usepackage{graphicx}
\usepackage{stackengine}

\usepackage{enumitem}
\setlist[enumerate,1]{label=\textbf{\arabic*}.}

\usepackage{marginnote}

\usepackage{MnSymbol}
\usepackage{marvosym}
\usepackage{wasysym}

\usepackage{euscript}
\usepackage[normalem]{ulem}

\usepackage{nicefrac}

\usepackage[babel=true]{csquotes}

\usepackage[marginparwidth=2cm]{geometry}
\usepackage[usenames,dvipsnames,svgnames,table]{xcolor}
\definecolor{orange_okabe_ito}{RGB}{226,155,49}
\definecolor{blue_okabe_ito}{RGB}{31,116,173}
\definecolor{green_okabe_ito}{RGB}{45,155,116}

\usepackage[colorlinks=true,urlcolor=DarkBlue,linkcolor=DarkRed,citecolor=DarkGreen,pagebackref=true]{hyperref}

\usepackage{booktabs}

\usepackage{pinlabel}

\usepackage{todo}

\usepackage{tikz}
\usetikzlibrary{shapes,arrows}
\usetikzlibrary{fit,positioning}
\usetikzlibrary{patterns}
\usepackage{xkeyval}
\usepackage{moreverb}
\usepackage{epic}
\usepgfmodule{shapes,plot,decorations}
\usepackage{etex}

\newcommand{\Z}{\mathbb{Z}}

\newcommand{\R}{\mathbb{R}}

\newcommand{\HH}{\mathbb{H}}

\newcommand{\G}{\Gamma}

\newcommand{\LGP}{\Lambda_{P}}

\newcommand{\C}{\mathcal{C}}

\newcommand{\F}{\mathcal{F}}

\newcommand{\V}{\mathcal{V}}

\newcommand{\GG}{\mathcal{G}}

\newcommand{\Ss}{\EuScript{S}}

\renewcommand{\O}{\Omega}

\renewcommand{\S}{\mathbb{S}}

\newcommand*{\Quotient}[2]{\ensuremath{#1/\!\raisebox{-.90ex}{\ensuremath{#2}}}}

\newcommand{\dev}{\textrm{dev}}
\newcommand{\Aut}{\textrm{Aut}}

\newcommand{\exc}{\mathrm{exc}}
\newcommand{\hyp}{\mathrm{hyp}}

\newcommand{\ie}{i.e. }
\usepackage{aurical}
\newcommand{\B}{\textrm{\Fontauri C\normalfont}}

\theoremstyle{plain}
\newtheorem{theorem}{Theorem}[section]

\newtheorem{propo}[theorem]{Proposition}
\newtheorem{prop}[theorem]{Proposition}
\newtheorem{cor}[theorem]{Corollary}
\newtheorem{lemma}[theorem]{Lemma}

\theoremstyle{definition}

\newtheorem{de}[theorem]{Definition}
\newtheorem{rem}[theorem]{Remark}

\newtheorem{theooo}{Theorem}[section]
\title{Deformation spaces of Coxeter truncation polytopes}

\author{Suhyoung Choi}
\address{Department of Mathematical Sciences, KAIST, Daejeon 34141, South Korea}
\email{schoi@math.kaist.ac.kr}

\author{Gye-Seon Lee}
\address{Department of Mathematical Sciences and Research institute of Mathematics, Seoul National University, Seoul 08826, South Korea}
\email{gyeseonlee@snu.ac.kr}

\author{Ludovic Marquis}
\address{Univ Rennes, CNRS, IRMAR - UMR 6625, F-35000 Rennes, France}
\email{ludovic.marquis@univ-rennes1.fr}

\begin{document}

\begin{abstract}
A convex polytope $P$ in the real projective space with reflections in the facets of $P$ is a \emph{Coxeter polytope} if the reflections generate a subgroup $\Gamma$ of the group of projective transformations so that the $\Gamma$-translates of the interior of $P$ are mutually disjoint. It follows from work of Vinberg that if $P$ is a Coxeter polytope, then the interior $\Omega$ of the $\Gamma$-orbit of $P$ is convex and $\Gamma$ acts properly discontinuously on $\Omega$.  
A Coxeter polytope $P$ is \emph{$2$-perfect} if $P \smallsetminus \Omega$ consists of only some vertices of $P$. 

In this paper, we describe the deformation spaces of $2$-perfect Coxeter polytopes $P$ of dimension $d \geqslant 4$ with the same dihedral angles when the underlying polytope of $P$ is a \emph{truncation polytope}, \ie a polytope obtained from a simplex by successively truncating vertices. The deformation spaces of Coxeter truncation polytopes of dimension $d = 2$ and $d = 3$ were studied respectively by Goldman and the third author.
\end{abstract}



\maketitle

\tableofcontents

\section{Introduction}

Let $X$ be a homogeneous space of a Lie group $G$. A $(G,X)$-structure on a manifold or an orbifold $\mathcal{O}$ is an atlas of coordinate charts on $\mathcal{O}$ valued in $X$ such that the changes of coordinates locally lie in $G$. It is a natural question to ask whether or not one can put a $(G,X)$-structure on $\mathcal{O}$ and, if so, how one can parametrize the space of all possible $(G,X)$-structures on $\mathcal{O}$, up to a certain equivalence, called the \emph{deformation space of $(G,X)$-structures on $\mathcal{O}$} (see Thurston \cite{thurston_bible} and Goldman \cite{bill_bible}).

\medskip

This paper studies convex real projective structures on orbifolds. A \emph{convex real projective structure} on a $d$-orbifold $\mathcal{O}$ is a $(\mathrm{PGL}_{d+1}(\R),\, \mathbb{RP}^d)$-structure on $\mathcal{O}$, where $\mathbb{RP}^d$ is the real projective $d$-space and $\mathrm{PGL}_{d+1}(\R)$ is the group of projective automorphisms of $\mathbb{RP}^d$, such that its developing map $\dev: \tilde{\mathcal{O}} \to \mathbb{RP}^d$ is a homeomorphism from the universal cover $\tilde{\mathcal{O}}$ of $\mathcal{O}$ onto a convex domain of $\mathbb{RP}^d$. Hyperbolic structures provide examples of convex real projective structures since the projective model of the hyperbolic $d$-space $\mathbb{H}^d$ is a round open ball $B$ in $\mathbb{RP}^d$, which is convex, and the isometry group $\mathrm{Isom}(\mathbb{H}^d)$ of $\mathbb{H}^d$ is the subgroup $\mathrm{PO}(d,1)$ of $\mathrm{PGL}_{d+1}(\R)$ preserving $B$.

\medskip

We particularly focus on a class of orbifolds, called \emph{Coxeter truncation} orbifolds. A \emph{Coxeter orbifold} is an orbifold whose underlying space is a convex polytope with some faces of codimension $\leqslant 2$ deleted and whose singular locus is its boundary. We do not need the technicality of orbifold in order to define the deformation space of convex real projective structures on a Coxeter orbifold since it may be identified with the space of isomorphism classes of Coxeter polytopes realizing an appropriate labeled polytope, which can be easier to define (see Section \ref{section:preliminary} for basic terminology). We will motivate why we are interested in Coxeter \emph{truncation} orbifolds in the following subsection.

\subsection{How truncation polytopes were used to construct new hyperbolic Coxeter polytopes}
\label{sebsec:truncation}

A \emph{hyperbolic $d$-polytope} is a convex $d$-polytope $P$  in an affine chart of $\mathbb{RP}^d$ such that every facet of $P$ has a nonempty intersection with $\mathbb{H}^d$. A hyperbolic polytope with dihedral angles integral submultiples of $\pi$, \ie $\nicefrac{\pi}{m}$ for some $m \in \{ 2, 3, \dotsc, \infty \}$, is called a \emph{hyperbolic Coxeter polytope}. By Poincaré's polyhedron theorem, the subgroup $\Gamma_P$ of $\mathrm{Isom}(\mathbb{H}^d)$ generated by the reflections in the facets of $P$ is discrete, and the $\Gamma_P$-translates of $P \cap \mathbb{H}^d$ form a tiling of $\mathbb{H}^d$. The quotient orbifold $\Quotient{\mathbb{H}^d}{\Gamma_P}$ is a hyperbolic Coxeter $d$-orbifold.

\medskip

Since this procedure is a very pleasant method to build discrete subgroups of $\mathrm{Isom}(\mathbb{H}^d)$, many people have made progress toward a far-reaching goal: the classification of all compact or finite volume hyperbolic Coxeter $d$-polytopes. Until now it has been achieved only when the dimension $d$ or the number of facets $n$ is small. The case $d=2$ is classical (see e.g. Beardon \cite{beardon_bible}), and the case $d=3$ follows from the work of Andreev \cite{MR0259734, MR0273510}. Starting from $d = 4$, only partial results are available. We refer the reader to the \href{http://www.maths.dur.ac.uk/users/anna.felikson/Polytopes/polytopes.html}{web page of Felikson}\footnote{\url{http://www.maths.dur.ac.uk/users/anna.felikson/Polytopes/polytopes.html}} for a detailed survey.

\medskip

Compact (resp. finite volume) hyperbolic $d$-polytopes with $n = d+1$ facets, i.e. simplices, were classified by Lannér \cite{MR0042129} (resp. Koszul \cite{LectHypCoxGrKoszul} and Chein \cite{MR0294181}), hence those polytopes are said to be \emph{Lannér} (resp. \emph{quasi-Lannér}). Note that Lannér (resp. quasi-Lannér) Coxeter $d$-simplices exist only when $d=2,3,4$ (resp. $d=2, 3, \dotsc, 9$).

\medskip

To build more examples of compact or finite volume hyperbolic Coxeter polytopes, we can use a simple and effective operation, called \emph{truncation}. First, find a hyperbolic Coxeter $d$-polytope $P$ such that every edge of $P$ intersects  $\mathbb{H}^d$ and at least one vertex of $P$ is \emph{hyperideal}, \ie in the complement of the closure $\overline{\mathbb{H}^d}$ in $\mathbb{RP}^d$. Second, for each hyperideal vertex $v$, take the dual hyperplane $H_v$ with respect to the quadratic form that defines $\mathbb{H}^d$. Then $H_v$ intersects perpendicularly all the edges containing $v$. Finally, truncate all the hyperideal vertices of $P$ via their dual hyperplanes in order to obtain a new polytope of finite volume (see Vinberg's survey \cite[Prop. 4.4]{MR783604}). In 1982, Maxwell \cite{MR679972} classified all the hyperbolic Coxeter simplices such that all their edges intersect $\mathbb{H}^d$. We call them \emph{$2$-Lannér}. 
A complete list of $2$-Lannér Coxeter simplices can be found in Chen--Labbé \cite{Chen:2013fk}. Since this list is essential for this paper, we reproduce it in Appendix \ref{classi_lorent}. Note that $2$-Lannér Coxeter $d$-simplices exist only when $d = 2, 3, \dotsc, 9$.
\medskip

Furthermore, after truncating the hyperideal vertices of these Coxeter simplices, one may glue them together to obtain new Coxeter polytopes if the new facets in place of the hyperideal vertices match each other. For example, using this technique Makarov \cite{MR0259735} built infinitely many compact hyperbolic Coxeter polytopes of dimension $d = 4, 5$. 

\medskip

The last two paragraphs motivate the following definition: a \emph{truncation polytope} is a polytope obtained from a simplex by successively truncating vertices, or equivalently obtained by gluing together \emph{once-truncated simplices} along some pairs of the simplicial facets (see Kleinschmidt \cite{MR0474043} for this equivalence). Here by an \emph{once-truncated simplex}, we mean a polytope obtained from a simplex $\Delta$ by truncating each vertex of $\Delta$ at most once. For example, a polygon with $n$ sides is always a truncation $2$-polytope, and it is an once-truncated $2$-simplex if and only if $n = 3, 4, 5$ or $6$.

\medskip

A truncation polytope can be characterized by a combinatorial invariant. We recall that a $d$-polytope $\GG$ is \emph{simple} if each vertex of $\GG$ is adjacent to exactly $d$ facets. Any truncation polytope is a simple polytope. If we denote by $f$ (resp. $r$) the number of facets (resp. ridges) of a $d$-polytope $\GG$, then 
$$ g_2 (\GG) = r - d f + \frac{d(d+1)}{2}$$ 
is an invariant of $\GG$. In \cite{Barnette}, Barnette proved that if $\GG$ is simple, then $g_2(\GG) \geqslant 0$. In addition, in the case $d \geqslant 4$, a simple polytope $\GG$ is a truncation polytope if and only if $g_2(\GG) = 0$ (see e.g. Br\o ndsted \cite{MR683612}). So, truncation polytopes are in some sense the polytopes with the “least complexity''. In \cite[Rem. 6.10.10]{MR2360474}, Davis mentioned that it might be a reasonable project to determine all possible hyperbolic Coxeter truncation polytopes of dimension $d \geqslant 4$. This paper indeed provides how to classify \emph{$2$-perfect} hyperbolic Coxeter truncation $d$-polytopes $P$, i.e. each edge of $P$ intersects with $\mathbb{H}^d$ (see Remark \ref{rem:geometrization}).

\subsection{Deformation space of Coxeter polytopes}\label{subsec:projective_Coxeter_polytope}

The main object of this paper is a generalization of hyperbolic Coxeter polytopes. A convex polytope $P$ of $\mathbb{RP}^d$ together with the projective reflections in the facets of $P$ is a \emph{projective Coxeter polytope}, or simply \emph{Coxeter polytope}, provided that if $\Gamma_P$ is the subgroup of $\mathrm{PGL}_{d+1}(\mathbb{R})$ generated by those reflections, then 
$$ \mathrm{Int}(P)  \cap  \gamma \cdot \mathrm{Int}(P) = \emptyset \quad \textrm{for each non-identity element } \gamma \in \Gamma_P,$$ 
where $\mathrm{Int}(P)$ denotes the interior of $P$. It follows from work of Vinberg \cite{MR0302779} that the interior $\Omega_P$ of the union of all $\Gamma_P$-translates of $P$ is a convex domain of $\mathbb{RP}^d$ and that the group $\Gamma_P$ acts properly discontinuously on $\Omega_P$. Then the quotient orbifold $\Quotient{\Omega_P}{\Gamma_P}$ is a convex real projective Coxeter orbifold.

\medskip

It is well known that if two finite volume hyperbolic Coxeter polytopes of dimension $d \geqslant 3$ have the same dihedral angles, then they are isometric, i.e. one is conjugate by an isometry of $\mathbb{H}^d$ to the other, by Mostow’s rigidity theorem (or the uniqueness theorem of Andreev \cite{MR0259734} for compact hyperbolic polytopes with non-obtuse dihedral angles). But in contrast to hyperbolic geometry, projective geometry allows for some Coxeter polytopes of dimension $d \geqslant 3$ to deform into non-$\mathrm{PGL}_{d+1} (\R)$-conjugate Coxeter polytopes with the same dihedral angles (see Section \ref{subsection:CoxeterPolytope} for the definition of dihedral angle). An interesting phenomenon in projective geometry, which cannot appear in hyperbolic geometry, is that some finite volume hyperbolic simplices of dimension $d \geqslant 3$ with at least one ideal vertex, \ie a vertex in the boundary of $\mathbb{H}^d$, can deform so that the ideal vertices become truncatable. Thus, a new family of Coxeter polytopes may be obtained by truncating such vertices of the deformed simplices and gluing the truncated simplices together. This is a strong motivation to understand the deformation space of Coxeter polytopes with \emph{fixed} dihedral angles, more precisely, the space $\B(\GG)$ of isomorphism classes $[P]$ of Coxeter polytopes $P$ realizing a \emph{labeled polytope} $\GG$, \ie a polytope whose ridges are labeled with dihedral angles (see Section \ref{subsection:Moduli}). 

\medskip

In hyperbolic geometry, a common hypothesis for the truncation process is that all the edges of a polytope meet the hyperbolic space. This hypothesis then implies that the truncated polytope has finite volume. In convex projective geometry, this hypothesis becomes that all the edges of a polytope $P$ meet the open convex domain $\O_P$. By Theorem \ref{theo_vinberg}.(5), it is equivalent to $P$ being \emph{$2$-perfect} that we define now.

\medskip

A Coxeter polytope $P$ is \emph{$m$-perfect} provided that for each face $f$ of dimension $(m-1)$ in $P$, the subgroup of $\Gamma_P$ generated by the reflections in the facets containing $f$ is finite. It is equivalent to the fact that the faces of $P$ not intersecting $\Omega_P$ have dimension $\leqslant m-2$. In particular, $P$ is $1$-perfect, simply called \emph{perfect}, if and only if  $P \subset \Omega_P$. If $P$ is $2$-perfect, then any face of $P$ not intersecting $\Omega_P$ has to be a vertex. The $m$-perfectness of a Coxeter polytope $P$ is a property of the underlying labeled polytope of $P$ (see Remark \ref{rem:independance}).

\medskip

In this paper, we describe the deformation space of Coxeter polytopes realizing a $2$-perfect labeled truncation $d$-polytope. We restrict ourselves to dimension $d \geqslant 4$ because the cases $d = 2$ and $d=3$ were already done by Goldman \cite{GoldmanThesis} and by the third author \cite{MR2660566} respectively.

\begin{theooo}\label{MainTheorem}
Let $\GG$ be an irreducible, large, $2$-perfect labeled truncation polytope of dimension $d \geqslant 4$ and let $\B(\GG)$ be the deformation space of $\GG$. Assume that $\B(\GG)$ is nonempty. Then:
\begin{itemize}
\item the dimension $d$ is less than or equal to $9$;
\item the space $\B(\GG)$ is a union of finitely many open cells of dimension $b(\GG) := e_+(\GG) - d$, where $e_+(\GG)$ is the number of ridges with label $\neq \nicefrac{\pi}{2}$ in $\GG$;
\item there exists a hyperbolic Coxeter polytope realizing $\GG$ if and only if $\B(\GG)$ is connected, \ie $\B(\GG)$ is an open cell.
\end{itemize}
\end{theooo}

\begin{rem}
The number $\kappa(\GG)$ of connected components of $\B(\GG)$ can be explicitly computed since the parametrization of $\B(\GG)$ is concretely constructed (see Theorem \ref{thm:number_of_components}). 
\end{rem}

\begin{rem}
In Theorem \ref{thm:geometrization}, we also give a characterization of irreducible, large, 2-perfect labeled truncation polytopes being \emph{hyperbolizable}, i.e. realized by a hyperbolic Coxeter polytope, or \emph{convex-projectivizable}, i.e. realized by a projective Coxeter polytope.
\end{rem}

\begin{rem}
It was proved by Choi--Lee \cite{MR3375519} and Greene \cite{MR3322033} that if $\GG$ is a perfect labeled truncation $d$-polytope realized by a hyperbolic Coxeter polytope $P$, then $\B(\GG)$ is smooth at $[P]$ and of dimension $b(\GG) = e_+(\GG) - d$.
\end{rem}

\begin{rem}\label{not:example}
If the labeled polytope $\GG$ is {\em not} a truncation polytope, then $\B(\GG)$ may not be a union of open cells. For example, there exist perfect labeled $4$-polytopes $\GG_1$ and $\GG_2$ such that $\B(\GG_1)$ is homeomorphic to a circle (see Choi--Lee--Marquis \cite{CLM2}) and $\B(\GG_2)$ is homeomorphic to $\{ (x,y) \in \mathbb{R}^2 \mid xy=0 \}$ (see Choi--Lee \cite{MR3375519}). In particular, $\B(\GG_2)$ is even {\em not} a manifold. 
\end{rem}

\subsection{Divisible and quasi-divisible convex domain}

Every properly convex domain $\Omega$ admits a Hilbert metric $d_{\Omega}$ so that the group  $\mathrm{Aut}(\Omega)$ of projective automorphisms preserving $\Omega$ acts on $\Omega$ by isometries for $d_{\Omega}$. Among such metric spaces $(\Omega,d_{\Omega})$, we are particularly interested in the one having the following property: there exists a discrete subgroup $\G$ of $\mathrm{Aut}(\Omega)$ such that $\Quotient{\O}{\G}$ is compact or of finite volume with respect to the Hausdorff measure induced by $d_{\Omega}$. In the case that $\Quotient{\O}{\G}$ is compact (resp. of finite volume), we call $\O$ \emph{divisible} (resp. \emph{quasi-divisible}) \emph{by $\Gamma$}. 
A natural question to ask is what kinds of (quasi-)divisible domains exist. We will give a short history of (quasi-)divisible domains.

\medskip

A properly convex domain $\O$ of $\mathbb{RP}^d$ is \emph{decomposable} if a cone of $\mathbb{R}^{d+1}$ lifting $\O$ is a non-trivial direct product of two smaller cones. So, only indecomposable convex domains are of interest to us, and all properly convex domains in this subsection are assumed to be indecomposable. Note that a strictly convex domain is always indecomposable.  

\medskip

First, there are \emph{homogeneous} (quasi-)divisible domains, i.e. the group $\Aut(\O)$ acts transitively on $\O$. All such domains except hyperbolic space are not strictly convex. They correspond to the symmetric spaces of the quasi-simple Lie groups $\mathrm{SL}_m (k)$ for $k$ the real, complex or quaternionic field or of the exceptional one $E_{6,-26}$ (see \cite{homogeneous_vin_1,koecher}).
 
\medskip

Second, the existence of \emph{inhomogeneous}, strictly convex, divisible (resp. quasi-divisible but not divisible) domains $\Omega$ in any dimension follows from the works of Koszul \cite{kos_open}, Johnson--Millson \cite{johnson_millson} and Benoist \cite{convexe_div_1} (resp. Ballas--Marquis \cite{BallasMarquis}). In these examples, the group $\Gamma$ (quasi-)dividing $\Omega$ is isomorphic to the fundamental group of a finite volume hyperbolic manifold $M$, and $(\Omega,\Gamma)$ is obtained by deforming the developing map and the holonomy of the hyperbolic structure on $M$, called \emph{bending} or \emph{bulging}.

\medskip

Third, there exist inhomogeneous, strictly convex, divisible $d$-domains $\Omega$ by $\Gamma$ such that $\Gamma$ is \emph{not} isomorphic to any lattice of $\mathrm{Isom} (\HH^d)$, by Benoist \cite{benoist_qi} for $d=4$ and by Kapovich \cite{kapo_qi} for any dimension $d\geqslant 4$. But, it is still an open question whether there exist inhomogeneous, strictly convex, \emph{quasi-divisible not divisible} domains $\Omega$ of any dimension $d \geqslant 4$ by a group $\Gamma$ non-isomorphic to a lattice of $\mathrm{Isom} (\HH^d)$. If a quasi-divisible $2$- or $3$-domain $\Omega$ by $\G$ is strictly convex, then $\G$ has to be isomorphic to a lattice of $\mathrm{Isom} (\HH^2)$ or $\mathrm{Isom} (\HH^3)$.

\medskip

Fourth, the examples of inhomogeneous, \emph{non-strictly convex}, divisible $d$-domains $\Omega$ by $\Gamma$ were found first by Benoist \cite{cd4} for $d=3, \dotsc,7$, and later by the authors \cite{CLM2} for $d=4, \dotsc,8$. It is also interesting to find such domains for any dimension $d > 8$. Note that inhomogeneous, non-strictly convex, quasi-divisible $2$-domain cannot exist by Benzécri \cite{benzecri} and the third author \cite{marquis_moduli_surf}. Except in dimension $3$ (see \cite{BDL_3d_geometrization}), all the known examples were built from Coxeter groups $\Gamma$, each of which is relatively hyperbolic with respect to a collection of virtually free abelian subgroups of rank $r_1, \dotsc, r_k \geqslant 2$ for some $k \in \mathbb{N} \smallsetminus \{0\}$. In the Benoist's examples, $r_i = d-1$ for all $i = 1, \dots, k$, but in the other examples, $r_i < d-1$.

\medskip

Finally, we consider inhomogeneous, non-strictly convex, \emph{quasi-divisible not divisible} domains $\Omega$ by $\Gamma$. It is slightly more complicated to explain them since non-trivial segments on the boundary $\partial \Omega$ may come from the ends or from the interior of the manifold (or orbifold) $\Quotient{\Omega}{\Gamma}$. To describe ends, Cooper, Long and Tillman \cite{clt_cusps,clt_koszul} and Ballas, Cooper and Leitner \cite{BallasCooperLeitner,BCL} developed a theory of generalized cusps, which can be of type $m \in \{0, 1, \dotsc, d \}$. Thanks to \cite{BallasMarquis,Ballas,bobb_gcusps}, we know that there exist inhomogeneous, non-strictly convex, quasi-divisible domains $\Omega$ by $\Gamma$ such that $\Quotient{\Omega}{\Gamma}$ has generalized cusps of type $m$, for $m = 1, \dotsc, d-2$. In those examples, the non-trivial segment in $\partial \Omega$ occurs because of the generalized cusp, and $\Quotient{\Omega}{\Gamma}$ is obtained again by bending cusped hyperbolic manifold. Note that the cusp of type $0$, which appears in hyperbolic geometry, cannot produce a non-trivial segment, and the cusp of type $\geqslant d-1$ prevents $\Quotient{\Omega}{\Gamma}$ from being finite volume (see \cite[Th.~0.6]{BallasCooperLeitner}).

\medskip

This paper exhibits examples of inhomogeneous, non-strictly convex, divisible (resp. quasi-divisible not divisible) domains $\Omega$ of dimension $d= 4, 5, 6, 7$ (resp. $d = 4, 8$) by $\Gamma$. If such domain is not divisible, then $\Quotient{\Omega}{\Gamma}$ has only cusps of type $0$ and $\Gamma$ is relatively hyperbolic with respect to a family of virtually $\Z^{d-1}$ subgroups. 

\begin{theooo}\label{thm:existence_quasi-divisible}
In dimension $d = 4$ and $8$, there exist indecomposable, inhomogeneous, non-strictly convex, quasi-divisible $d$-domains $\Omega$ by $\Gamma$ such that $\Quotient{\Omega}{\Gamma}$ has only generalized cusps of type $0$.
\end{theooo}

\begin{rem}
Such $d$-domains as in Theorem \ref{thm:existence_quasi-divisible} also exist in dimension $3$, $5$, $6$ and $7$ (see Remark \ref{rem:other_quasi-divisible}).
\end{rem}

\subsection{Geometrization}

From the point of view of geometrization “à la Thurston”, this paper provides the characterization of hyperbolization and convex-projectivization for Coxeter truncation orbifolds. The precise statement is somewhat technical, but nevertheless, we compare the surfaces and the truncation polytopes briefly for helping the reader to understand Theorem \ref{thm:geometrization}.

\medskip

To study the geometry and the topology of surfaces $S$ with negative Euler characteristic, one considers a finite collection of disjoint simple closed curves cutting $S$ into pairs of pants. Similarly, for irreducible, large, 2-perfect, labeled truncation $d$-polytopes $\GG$, we can find a finite collection of disjoint \emph{prismatic circuits} which decompose $\GG$ into irreducible once-truncated $d$-simplices $\GG_i$ (see Section \ref{section:gluing_splitting}). If a labeled polytope is considered as a Coxeter orbifold, then prismatic circuits may be identified with incompressible suborbifolds. 

\medskip

Each pair of pants admits hyperbolic structures, but it is not true that each once-truncated $d$-simplex $\GG_i$ is hyperbolizable or convex-projectivizable. We need an extra condition  on the prismatic circuits of $\GG_i$. To each prismatic circuit $\delta$ of $\GG_i$ is associated a Coxeter group $W_{\delta}$. Then (i) $\GG_i$ is hyperbolizable if and only if the Coxeter group $W_\delta$ is Lannér for each $\delta$, and (ii) $\GG_i$ is convex-projectivizable if and only if $W_\delta$ is either Lannér or $\tilde{A}_{d-1}$ for each $\delta$ (see Appendix \ref{classi_diagram} for the spherical and affine Coxeter groups). After once-truncated simplices $\GG_i$ are geometrized, they may be glued together whenever the geometry at the prismatic circuits matches up, analogous to gluing pairs of pants. This leads to a geometrization of $\GG$.

\begin{theooo}\label{thm:geometrization}
Let $\GG$ be an irreducible, large, 2-perfect, labeled truncation polytope of dimension $d \geqslant 4$, and let $\mathcal{P}$ be the set of prismatic circuits of $\GG$. Then: 
\begin{itemize}
\item $\GG$ is hyperbolizable if and only if $W_\delta$ is Lannér for each $\delta \in \mathcal{P}$;
\item $\GG$ is convex-projectivizable if and only if $W_\delta$ is Lannér or $\tilde{A}_{d-1}$ for each $\delta \in \mathcal{P}$.
\end{itemize}
In particular, in the case that $\GG$ is perfect, it is hyperbolizable if and only if it is convex-projectivizable and $W_\GG$ is word-hyperbolic.
\end{theooo}

\subsection*{Organization of the paper}

Section \ref{section:preliminary} recalls the background material including Vinberg’s theory of discrete reflection groups. Section \ref{section:moduli_of_simplex} discusses the deformation spaces of Coxeter simplices realizing an irreducible, large, $2$-perfect labeled simplex of dimension $d \geqslant 4$. In Section \ref{section:trucation_stacking}, we introduce two important operations on polytope, which are dual to each other: truncation and stacking. Section \ref{section:gluing_splitting} explains how to glue two Coxeter polytopes and how to do the reverse operation: to split one Coxeter polytope into two, and gives the proofs of Theorems \ref{thm:geometrization} and \ref{MainTheorem}. In Section \ref{section:number_of_components}, we count connected components of deformation space. Section \ref{section:higherDim} describes the deformation space of each individual labeled truncation polytopes of dimension $d \geqslant 6$ and Section \ref{section:Dim5} shows some features of truncation $5$-polytopes. In Section \ref{section:geometric_interpretation}, we explain some geometric properties of discrete reflection groups constructed in this paper, and give the proof of Theorem \ref{thm:existence_quasi-divisible}. 

\medskip

Finally, in five appendixes, we collect various Coxeter diagrams: the irreducible spherical or affine Coxeter diagrams (in Appendix \ref{classi_diagram}), the Lannér Coxeter diagrams of rank $4$ (in Appendix \ref{appendix:lanner}), the 2-Lannér Coxeter diagram of rank $\geqslant 5$ with colored nodes to encode their geometric properties (in Appendix \ref{classi_lorent}), the diagrams of $2$-perfect Coxeter prisms of dimension $d=6,7,8$ (in Appendix \ref{appendix:678prism}), and the diagrams of exceptional Coxeter $5$-prisms (in Appendix \ref{appendix:5prism}). 

\subsection*{Acknowledgements}

We are thankful for helpful conversations with Anna Wienhard and Maxime Wolff. This work benefited from the ICERM 2013 semester workshop on Exotic Geometric Structures, attended by the all three authors. We thank Balthazar Fléchelles, Stefano Riolo and the referee for carefully reading this paper and suggesting several improvements.

S. Choi was supported by the the National Research Foundation of Korea(NRF) grant funded by the Korea government(MSIP) under Grant number 2019R1A2C108454412. G.-S. Lee was supported by the DFG research grant “Higher Teichm\"{u}ller Theory”, by the European Research Council under ERC-Consolidator Grant 614733 and by the National Research Foundation of Korea(NRF) grant funded by the Korea government(MSIT) (No. 2020R1C1C1A01013667), and he acknowledges support from U.S. National Science Foundation grants DMS 1107452, 1107263, 1107367 “RNMS: GEometric structures And Representation varieties” (the GEAR Network). L. Marquis acknowledges support by the Centre Henri Lebesgue (ANR-11-LABX-0020 LEBESGUE).

\section{Preliminary}\label{section:preliminary}

In this section, we recall background material including Vinberg's results \cite{MR0302779}, which are essentially used in this paper (see also Benoist \cite{MR2655311}).

\subsection{Coxeter groups}

A \emph{Coxeter matrix} $M$ on a finite set $S$ is a symmetric $S \times S$ matrix $M=(M_{st})_{s,t \in S}$ with entries $M_{st} \in \{ 1, 2, \dotsc, \infty \} $ such that the diagonal entries $M_{ss}=1$ and the others $M_{st} \neq 1$. To a Coxeter matrix $M$ is associated a \emph{Coxeter group} $W_S$: the group presented by the set of generators $S$ and the relations $(st)^{M_{st}}=1$ for each $(s,t) \in S \times S$ with $M_{st} \neq \infty$. The cardinality $\# S$ of $S$ is called the \emph{rank} of the Coxeter group $W_S$. 

\medskip

All the information of a Coxeter group $W_S$ is encoded in a labeled graph $\mathcal{D}_{W}$, which we call the \emph{Coxeter diagram} of $W_S$: (i) the set of nodes\footnote{We prefer using the word “{\em node}” rather than “{\em vertex}” for the Coxeter diagram in order to distinguish a node of a diagram from a vertex of a polytope.} of $\mathcal{D}_W$ is $S$, (ii) two nodes $s,t \in S$ are connected by an edge $\overline{st}$ if and only if $M_{st} \in \{ 3, 4, \dotsc, \infty \} $, (iii) the label of the edge $\overline{st}$ is $M_{st}$. It is customary to omit the label of the edge $\overline{st}$ if $M_{st} = 3$.

\medskip

For any subset $S'$ of $S$, the $S' \times S'$ submatrix of $M$ is a Coxeter matrix $M'$ on $S'$. Since the natural homomorphism $W_{S'} \rightarrow W_S$ is injective, we may identify $W_{S'}$ with the subgroup of $W_S$ generated by $S'$. Such a subgroup is called a \emph{standard subgroup} of $W_S$.

\medskip

The connected components of the Coxeter diagram $\mathcal{D}_W$ are Coxeter diagrams of the form $\mathcal{D}_{W_{S_i}}$, where the $S_i$ form a partition of $S$. The subgroups $W_{S_i}$ are called the \emph{components} of $W_S$. A Coxeter group $W_S$ is \emph{spherical} (resp. \emph{affine}) if each component of $W_S$ is finite (resp. infinite and virtually abelian), and it is \emph{irreducible} if $\mathcal{D}_W$ is connected. Note that every irreducible Coxeter group $W_S$ is spherical, affine or \emph{large}, \ie $W_S$ has a finite index subgroup with a non-abelian free quotient (see Vinberg--Margulis \cite{MR1748082}). We often use the well-known classification of the irreducible spherical or irreducible affine Coxeter groups (see Appendix \ref{classi_diagram}).

\subsection{Coxeter polytopes}\label{subsection:CoxeterPolytope}

Let $V = \mathbb{R}^{d+1}$ and let $\mathbb{S}(V)$ be the projective sphere, \ie the space of half-lines in $V$ emanating from $0$. The automorphism group of $\mathbb{S}(V)$ is the group $\mathrm{SL}^{\pm}(V)$ of matrices of determinant $\pm 1$. We use the notation $\S^d$ to indicate the dimension of $\S(V)$. For example, the projective $0$-sphere $\S^0$ consists of two points. 

\medskip

A \emph{projective reflection} $\sigma$ is an element of $\mathrm{SL}^{\pm}(V)$ of order 2 which fixes a projective hyperplane of $\mathbb{S}(V)$ pointwise. In other words, there exists a vector $b \in V$ and a linear functional $\alpha \in V^{*}$, the dual vector space of $V$, such that 
$$ \sigma = \mathrm{\mathrm{Id}} -\alpha\otimes b \textrm{ with } \alpha(b)=2, \quad \textrm{i.e.} \quad \sigma(v) = v - \alpha(v)b \quad \forall v \in V.$$ 
We denote by $\hat{\mathbb{S}}$ the natural projection of $V \smallsetminus \{ 0 \}$ to $\mathbb{S}(V)$, and let $\mathbb{S}(W) := \hat{\mathbb{S}}(W \smallsetminus \{ 0 \})$ for any subset $W$ of $V$. The \emph{support} and the \emph{pole} of the reflection $\sigma$ are the hyperplane $\S(\ker(\alpha))$ and the point $[b] := \S(b)$ of $\S(V)$ respectively.

\medskip

The complement of a projective hyperplane in $\mathbb{S}(V)$ consists of two connected components, each of which we call an \emph{affine chart} of $\mathbb{S}(V)$. A subset $\mathcal{C}$ of $\S(V)$ is \emph{convex} if there exists a convex cone\footnote{By a \emph{cone} we mean a subset of $V$ which is invariant under multiplication by positive scalars.} $U$ of $V$ such that $\mathcal{C} = \mathbb{S}(U)$, \emph{properly convex} if it is convex and its closure lies in some affine chart, and \emph{strictly convex} if in addition its boundary does not contain any nontrivial projective line segment. A \emph{projective polytope} is a properly convex subset $P$ of $\S(V)$ with nonempty interior such that 
$$P = \bigcap_{i=1}^{n}  \S(\{ x \in V \,\,|\,\, \alpha_i(x) \leqslant 0 \})$$
for some nonzero $\alpha_i \in V^*$.
Recall that a face of codimension $1$ (resp. $2$) in $P$ is a \emph{facet} (resp. \emph{ridge}) of $P$. Two facets $s, t$ of $P$ are {\em adjacent} if the intersection $s \cap t$ is a ridge of $P$. We always assume that $P$ has $n$ facets, \ie in order to define $P$, we need all the $n$ linear functionals $(\alpha_i)_{i=1}^{n}$.

\begin{de}
A \emph{Coxeter polytope} is a pair $(P, (\sigma_s)_{s \in S})$ of a projective polytope $P$ with the set $S$ of its facets and the reflections $(\sigma_s = \mathrm{Id} - \alpha_s \otimes b_s)_{s \in S}$ with $\alpha_s(b_s) = 2$ such that:
\begin{itemize}
\item for each facet $s \in S$, the support of $\sigma_s$ is the supporting hyperplane of $s$;
\item for each pair of facets $s \neq t$ of $P$,
  \begin{enumerate} 
  \item $ \alpha_s(b_t)$ and $\alpha_t(b_s)$ are both zero or both negative,

  \item $ \alpha_s(b_t) \alpha_t(b_s) \geqslant 4$\, or \,$\alpha_s(b_t) \alpha_t(b_s) = 4 \cos^2 ( \nicefrac{\pi}{m_{st}})$ for some $m_{st} \in \mathbb{N} \smallsetminus \{ 0,1\}$.
  \end{enumerate}
\end{itemize}
\end{de}

We often denote the Coxeter polytope simply by $P$. For every pair of distinct facets $s, t$ of $P$, the composite $\sigma_s \sigma_t$ acts trivially on the subspace $U = \ker(\alpha_s) \cap \ker(\alpha_t)$ of codimension 2, hence $\sigma_s \sigma_t$ induces an element of $\mathrm{SL}(V/U)$, which is conjugate to the following matrix:
\begin{enumerate}
\item[(\texttt{N})]
$\left(\begin{matrix}
\lambda  & 0 \\ 
0            & \lambda^{-1}
\end{matrix}\right)$ for some $\lambda > 0$ 
 \;\; if $\alpha_s(b_t) \alpha_t(b_s) > 4 $;

\item[(\texttt{Z})]
$\left(\begin{matrix}
1  & 1 \\ 
0  & 1
\end{matrix}\right)$
\;\; if $\alpha_s(b_t) \alpha_t(b_s) = 4$;

 \item[(\texttt{P})]
$\left(\begin{matrix}
\cos( \frac{2\pi}{m_{st}} ) & -\sin( \frac{2\pi}{m_{st}} ) \\ 
\sin( \frac{2\pi}{m_{st}} )  & \phantom{-} \cos( \frac{2\pi}{m_{st}} )
\end{matrix}\right)$ \;\; if $\alpha_s(b_t) \alpha_t(b_s) = 4 \cos^2 ( \frac{\pi}{m_{st}}) $.
\end{enumerate}

\medskip

In the case (\texttt{P}) the two facets $s, t$ are adjacent by Vinberg \cite[Th. 7]{MR0302779}. For each pair of adjacent facets $s, t$ of $P$, the \emph{dihedral angle} of the ridge $s \cap t$ is said to be $\nicefrac{\pi}{m_{st}}$ in the case (\texttt{P}) and to be $0$ in the cases (\texttt{N}) and (\texttt{Z}). To a Coxeter polytope $P$ is associated a Coxeter matrix $M$ on $S$: (i) the set of facets of $P$ is $S$; (ii) for each pair of distinct facets $s, t$ of $P$, we set $M_{st}=m_{st}$ in the case (\texttt{P}) and $M_{st} =\infty$ otherwise. We denote by $W_P$ the Coxeter group associated to this Coxeter matrix $M$.

\subsection{Tits--Vinberg's Theorem}

If $P$ is a Coxeter polytope and $f$ is a face of $P$, then we let $S_f= \{ s \in S \mid f \subset s \}$.

\begin{theorem}[Tits {\cite[Chap. V]{MR0240238}} for the Tits simplex and Vinberg {\cite[Th. 2]{MR0302779}}]\label{theo_vinberg}
Let $P$  be a Coxeter polytope of $\S(V)$ with Coxeter group $W_P$, and let $\G_P$ be the group generated by the projective reflections $(\sigma_s)_{s \in S}$. Then the following hold:
\begin{enumerate}
\item the homomorphism $\sigma:W_P \rightarrow \Gamma_P \subset \mathrm{SL}^{\pm}(V)$ defined by $\sigma(s) =
\sigma_s$ is an isomorphism;

\item the group $\Gamma_P$ is a discrete subgroup of $\mathrm{SL}^{\pm}(V)$;

\item the union of the $\Gamma_P$-translates of $P$ is a convex subset $\C_P$ of $\S(V)$;

\item if $\Omega_P$ is the interior of $\C_P$, then $\Gamma_P$ acts properly discontinuously on $\Omega_P$;

\item an open face $f$ of $P$ lies in $\O_P$ if and only if $W_{S_f}$ is spherical.
\end{enumerate}
\end{theorem}

As a consequence of Theorem \ref{theo_vinberg}, the following are equivalent: 
\begin{enumerate}
\item $\C_{P}$ is open in $\S(V)$;
\item $W_{S_v}$ is spherical for each vertex $v$ of $P$;
\item the action of $\G_{P}$ on $\O_{P}$ is cocompact.
\end{enumerate} 
Following Vinberg, we call such $P$ \emph{perfect} (see \cite[Def. 8]{MR0302779}). 

\subsection{Deformation space of labeled polytope}\label{subsection:Moduli}

The \emph{face poset} $\mathcal{F}(P)$ of a projective polytope $P$ is the poset of all the faces of $P$ partially ordered by inclusion. Two polytopes $P$ and $P'$ are \emph{combinatorially equivalent} if there exists a bijection $\phi$ between $\mathcal{F}(P)$ and $\mathcal{F}(P')$  such that $\phi$ preserves the inclusion relation, \ie for every $f_1, f_2 \in \mathcal{F}(P)$, $f_1 \subset f_2$ $\Leftrightarrow$ $\phi(f_1) \subset \phi(f_2)$. We call $\phi$ a \emph{poset isomorphism}. A \emph{combinatorial polytope} is a combinatorial equivalence class of polytopes.
A \emph{labeled polytope} is a pair of a combinatorial polytope $\GG$ and a \emph{ridge labeling} on $\GG$, which is a function of the set of ridges of $\GG$ to $\{ \nicefrac{\pi}{m} \mid m = 2,3, \dotsc, \infty \}$.

\medskip

Let $\GG$ be a labeled $d$-polytope. A {\em Coxeter polytope realizing $\GG$} is a pair $(P,\phi)$ of a Coxeter $d$-polytope $P$ of $\S^d$ and a poset isomorphism $\phi$ between $\mathcal{F}(\GG)$ and $\mathcal{F}(P)$ such that the label of each ridge $r$ of $\GG$ is the dihedral angle of the ridge $\phi(r)$ of $P$. Two Coxeter polytopes $(P,\phi : \mathcal{F}(\GG) \rightarrow \mathcal{F}(P))$ and $(P',\phi' : \mathcal{F}(\GG) \rightarrow \mathcal{F}(P'))$ realizing $\GG$ are {\em isomorphic} if there exists a projective automorphism $\psi$ of $\S^d$ such that $\psi(P) = P'$ and $\hat{\psi} \circ \phi = \phi'$, where $\hat{\psi}$ is the poset isomorphism between $\mathcal{F}(P)$ and $\mathcal{F}(P')$ induced by $\psi$.

\begin{de}\label{def:DeformationSpace}
The \emph{deformation space $\B(\GG)$ of a labeled $d$-polytope $\GG$} is the space of isomorphism classes of projective Coxeter $d$-polytopes realizing $\GG$.
\end{de}

For convenience in notation, we often delete the poset isomorphism $\phi$ in the Coxeter polytope $(P,\phi : \mathcal{F}(\GG) \rightarrow \mathcal{F}(P))$ realizing $\GG$, and rely on the context to make clear which of these poset isomorphism is intended. And, we denote simply by $[P]$ the isomorphism class of a Coxeter polytope $P$ realizing $\GG$.

\begin{rem}\label{rem:HyperbolicStructure}
In the same way as Definition \ref{def:DeformationSpace}, we may introduce the space $\mathrm{Hyp}(\GG)$ of isomorphism classes of hyperbolic Coxeter polytopes realizing $\GG$. Here, two hyperbolic Coxeter polytopes $(P,\phi : \mathcal{F}(\GG) \rightarrow \mathcal{F}(P))$ and $(P',\phi' : \mathcal{F}(\GG) \rightarrow \mathcal{F}(P'))$ realizing $\GG$ are in the same isomorphism class if there exists an isometry $\psi$ between $P$ and $P'$ such that $\hat{\psi} \circ \phi = \phi'$, where $\hat{\psi}$ is the poset isomorphism induced by $\psi$. 
\end{rem}

\begin{rem}
A labeled polytope $\GG$ and its deformation space $\B(\GG)$ (resp. $\mathrm{Hyp}(\GG)$) may be considered as a Coxeter orbifold $\mathcal{O}$ and the deformation space of convex real projective structures (resp. hyperbolic structures) on $\mathcal{O}$ (see e.g. \cite{MR3375519,CLM_survey}). 
\end{rem}

\subsection{Cartan matrix of Coxeter polytope}

A matrix $A = (A_{ij})$ of size $n \times n$ is a \emph{Cartan matrix}\footnote{The term “Cartan matrix” may have several meanings in the literature. In this paper, we follow Vinberg \cite{MR0302779}.} if
\begin{itemize}
\item[(i)] $A_{ii} = 2$\: $\forall i = 1, \dotsc, n$;\quad (ii) $A_{ij} \leqslant 0$\: $\forall i \neq j$;\quad (iii) $A_{ij} =0 \Leftrightarrow A_{ji}=0$;
\item[(iv)] for all $i \neq j$,\, $A_{ij} A_{ji} \geqslant 4$ or $A_{ij} A_{ji}  = 4 \cos^{2} ( \nicefrac{\pi}{m} )$ with some $m \in \mathbb{N} \smallsetminus \{0, 1\} $.
\end{itemize}

A Cartan matrix $A$ is \emph{reducible} if (after a reordering of the indices) $A$ is the direct sum of smaller (square) matrices $A_1$ and $A_2$, \ie $A = \big(\begin{smallmatrix}
  A_1 & 0\\
  0 & A_2
\end{smallmatrix}\big)$.  Otherwise, $A$ is \emph{irreducible}. 
The Perron--Frobenius theorem implies that an irreducible Cartan matrix $A$ has a simple eigenvalue $\lambda_A$ which corresponds to an eigenvector with positive entries and has the smallest modulus among the eigenvalues of $A$. We say that $A$ is of \emph{positive}, \emph{zero} or \emph{negative type} when $\lambda_A$ is positive, zero or negative, respectively. Every Cartan matrix $A$ is the direct sum of irreducible submatrices, each of which we call a \emph{component} of $A$. We denote by $A^{+}$ (resp. $A^{0}$, resp. $A^{-}$) the direct sum of the components of positive (resp. zero, resp. negative) type of $A$. Obviously, the Cartan matrix $A$ is the direct sum of $A^{+}$, $A^{0}$ and $A^{-}$.

\medskip

Let $\mathbb{R}^*_{+}$ be the set of positive real numbers. A Coxeter polytope 
$$ \big(\, P, \,(\sigma_s = \mathrm{Id} - \alpha_s \otimes b_s)_{s \in S} \, \big)$$ 
determines the pairs $(\alpha_s,b_s)_{s \in S}$ in $(V^* \times V)^{S}$, unique up to the following action of $(\mathbb{R^*_{+})}^{S}$ on $(V^* \times V)^{S}$:
$$ (\lambda_s)_{s\in S} \, \cdot \,(\alpha_s,b_s)_{s\in S}  \mapsto (\lambda_s \, \alpha_s, \lambda_s^{-1} \, b_s)_{s\in S}$$
This action leads to define an equivalence relation on Cartan matrices: two Cartan matrices $A$ and $A'$ of the same size are \emph{equivalent} if there exists a positive diagonal matrix $D$ such that $A'=DAD^{-1}$. We denote by $[A]$ the equivalence class of $A$. A Cartan matrix $A$ is \emph{symmetrizable} if it is equivalent to a symmetric matrix.

\medskip

Now, to a Coxeter polytope $P$ is associated an $S \times S$ Cartan matrix $A_P$ defined by $(A_P)_{s t} = \alpha_{s}(b_{t})$ for each pair of facets $s, t$ of $P$. Here the Cartan matrix $A_P$ depends on the choice of the pairs $(\alpha_s,b_s)_{s \in S}$, but the equivalence class $[A_P]$ does not. We call $A_P$ a \emph{Cartan matrix of $P$}.

\begin{de}\label{def_3types}
A Coxeter polytope $P$ of $\S^d$ is \emph{elliptic} if $A_P = A_P^{+}$, \emph{parabolic} if $A_P = A_P^0$ and $A_P$ is of rank $d$, and \emph{loxodromic} if $A_P = A_P^{-}$ and $A_P$ is of rank $d+1$.
\end{de}

The Cartan matrix $A_P$ is irreducible if and only if the Coxeter group $W_P$ is irreducible. In this case, we call $P$ \emph{irreducible}. The following theorem shows how important the Cartan matrix is: 

\begin{theorem}[{Vinberg \cite[Cor. 1]{MR0302779}}]\label{surj}
Let $A$ be a Cartan matrix. Assume that $A$ is irreducible, of negative type and of rank $d+1$. Then there exists a Coxeter $d$-polytope $P$ unique up to automorphism of $\S^d$ such that $A_{P} = A$.
\end{theorem}

\subsection{Perfect, quasi-perfect and 2-perfect polytopes}\label{subsection:2-perfect} 

Let $P$ be a Coxeter polytope of dimension $d$. For each vertex $v$ of $P$, we shall construct a new Coxeter polytope $P_v$ of dimension $d-1$, which is the Coxeter polytope “seen” from $v$. We call $P_v$ the \emph{link} of $P$ at $v$. First consider the set $S_v$ of facets containing $v$. Second notice that for each $s \in S_v$, the reflection $\sigma_s$ acts trivially on the subspace $\langle v\rangle$ of $\R^{d+1}$ spanned by $v$, hence $\sigma_s$ induces a reflection of the projective sphere $\S \Big( \Quotient{\R^{d+1}}{\langle v\rangle} \Big)$ of dimension $d-1$. Finally the projective polytope 
$$\bigcap_{s \in S_v}  \S \big( \{ x \in \Quotient{\R^{d+1}}{\langle v\rangle}  \mid \alpha_s(x) \leqslant 0 \}  \big)  $$
together with the induced reflections gives us the \emph{link} $P_v$ of $P$ at $v$.  A Coxeter polytope $P$ is \emph{$2$-perfect} if for each vertex $v$ of $P$, the link $P_v$ is perfect. This definition of $2$-perfectness is equivalent to the one in the introduction. If all vertex links are elliptic or parabolic, then $P$ is said to be \emph{quasi-perfect}. If $P$ is quasi-perfect, then $P$ is $2$-perfect because every elliptic or parabolic Coxeter polytope is perfect.

\medskip

In the case that $\GG$ is a \emph{labeled} polytope, the \emph{link} $\GG_v$ of $\GG$ at a vertex $v$ is simply the link of the underlying combinatorial polytope together with the obvious ridge labeling induced from $\GG$. To a labeled polytope $\GG$ is naturally associated a Coxeter group $W_{\GG}$, which we call the \emph{Coxeter group of $\GG$}. A labeled polytope $\GG$ is \emph{irreducible}, \emph{spherical}, \emph{affine} or \emph{large} when so is $W_{\GG}$ respectively. A labeled polytope $\GG$ is \emph{perfect} (resp. \emph{2-perfect}) if the link $\GG_v$ is spherical (resp. perfect) for each vertex $v$ of $\GG$. 

\medskip

\begin{rem}\label{rem:independance}
Let $\GG$ be a labeled polytope and $P$ a Coxeter polytope realizing $\GG$. Then $P$ is perfect (resp. 2-perfect) if and only if $\GG$ is perfect (resp. 2-perfect).
\end{rem}

Another construction of new Coxeter polytope from old ones is the \emph{join} of two Coxeter polytopes. We denote by $\hat{\S}_d$ the natural projection $\mathbb{R}^{d+1} \smallsetminus \{ 0 \} \rightarrow \S^{d}$, and let $\S_d^{-1}(A) := \hat{\S}_d^{-1}(A) \cup \{ 0 \} $ for any subset $A$ of $\S^{d}$. Given two Coxeter polytopes $(P,(\sigma_s)_{s \in S})$ and $(P',(\sigma_{s'})_{{s'} \in S'})$ of dimension $d$ and $d'$ respectively, we construct a Coxeter polytope of dimension $(d + d' +1)$, denoted by $P \otimes P'$: the projective polytope $\S_{d + d' +1}(\S_{d}^{-1}(P) \times \S_{d'}^{-1}(P'))$ together with $(\# S + \# S')$ reflections $(\sigma_s \times \mathrm{Id})_{s \in S}$ and $(\mathrm{Id} \times \sigma_{s'})_{s' \in S'}$ in $\mathrm{SL}^{\pm}_{d+d'+2}(\mathbb{R})$. For example, the join of a Coxeter $d$-polytope $P$ and a Coxeter $0$-polytope is a Coxeter $(d+1)$-polytope, denoted by $P \otimes \cdot\,$, whose underlying polytope is the cone over $P$, \ie the pyramid with base $P$. Here the Coxeter $0$-polytope is a point of $\S^0$ with Coxeter group $\Quotient{\mathbb{Z}}{2 \mathbb{Z}}$. We can also define the join $\Omega \otimes \Omega'$ of two convex subsets $\Omega$ and $\Omega'$. 

\medskip

The following theorem allows us to focus on \emph{irreducible}, \emph{large}, $2$-perfect labeled polytopes by giving the complete description of the deformation space of any $2$-perfect labeled polytope except large ones.

\begin{theorem}[{Vinberg \cite[Prop.\,26]{MR0302779}} for perfect polytopes and {Marquis \cite[Prop.\,5.1]{Marquis:2014aa}}]\label{tri}
Let $\GG$ be a $2$-perfect labeled $d$-polytope with Coxeter group $W_{\GG}$. Assume that the deformation space $\B(\GG)$ is nonempty. Then:

\begin{enumerate}\setlength{\itemsep}{5pt}
\item if $W_{\GG}$ is spherical, then $\B(\GG)$ consists of only one isomorphism class $[P]$, which is elliptic, and $\O_P = \S^d$;

\item if $W_{\GG} = \tilde{A}_{d}$, then $\B(\GG)$ consists of one parameter family of isomorphism classes $[P]$:
\begin{itemize}
\item either $P$ is parabolic and $\O_P$ is an affine chart $\mathbb{A}^d$ of $\mathbb{S}^d$,
\item or $P$ is loxodromic, and $\O_P$ is a simplex $\Delta^d$ of dimension $d$;
\end{itemize} 
\item if $W_{\GG} = \tilde{A}_{d-1} \times A_1$, then $\B(\GG)$ consists of one parameter family of classes $[P] = [Q \otimes \cdot \,]$: 
\begin{itemize}
\item either $Q$ is parabolic and $\O_Q$ is an affine chart $\mathbb{A}^{d-1}$ (so $\O_P = \mathbb{A}^{d-1} \otimes \S^0$), 
\item or $Q$ is loxodromic, and $\O_Q$ is a simplex $\Delta^{d-1}$ (so $\O_P = \Delta^{d-1} \otimes \S^0$);
\end{itemize}

\item if $W_{\GG}$ is infinite and virtually abelian but is neither $\tilde{A}_{d}$ nor $\tilde{A}_{d-1} \times A_1$, then $\B(\GG) = \{[P]\}$:
\begin{itemize}
\item either $P$ is parabolic and $\O_P = \mathbb{A}^d$,
\item or $P =  Q \otimes \cdot \,$ with $Q$ parabolic, and $\O_P = \mathbb{A}^{d-1} \otimes \S^0$;
\end{itemize}

\item otherwise, $W_{\GG}$ is large, and for each $[P] \in \B(\GG)$,
\begin{itemize}
\item either $P$ is irreducible and loxodromic, and $\O_P$ is properly convex,
\item or $P =  Q \otimes \cdot \,$ with $Q$ irreducible and loxodromic, and $\O_P = \O_Q \otimes \S^0$ with $\O_Q$ properly convex.
\end{itemize}
\end{enumerate}
Moreover, if $\GG$ is perfect, then each $[P] \in \B(\GG)$ is either elliptic, parabolic, or irreducible and loxodromic.
\end{theorem}

\begin{rem}\label{rem:parabolic_inv=0}
In the case $W_{\GG} = \tilde{A}_{d}$, a class $[P] \in \B(\GG)$ is parabolic if and only if $\det(A_P)=0$. In Section \ref{section:moduli_of_simplex}, we introduce a more interesting invariant $R : \B(\GG) \rightarrow \mathbb{R}$ such that  $R$ is a homeomorphism, and $[P]$ is parabolic if and only if $R([P])= 0$.
\end{rem}

\subsection{Invariant of Cartan matrix}

Let $\GG$ be a labeled polytope with Coxeter group $W_S$, and let $P$ be a Coxeter polytope realizing $\GG$ with Cartan matrix $A$. A $k$-tuple of distinct elements of $S$ is called a \emph{$k$-circuit} of $W_S$. For each $k$-circuit $\C= (i_1, i_2, \dotsc, i_k)$, we define a number 
$$\C(A) = A_{i_1 i_2} A_{i_2 i_3} \cdots A_{i_k i_1}$$ 
 which does not change upon the cyclic permutation of $\C$ and upon the choice of a representative in the class $[A]$. Such a number is called a \emph{cyclic product} of $A$. From now on, a $k$-circuit of $W_S$ is always considered as the $k$-circuit up to cyclic permutation, \ie
$$(i_1, i_2, \dotsc, i_{k-1}, i_k) = (i_2, i_{3}, \dotsc, i_{k}, i_{1}) = \dotsm = (i_k, i_1, \dotsc, i_{k-2}, i_{k-1}).$$

\medskip

The cyclic products are useful because of the following:

\begin{theorem}[{Vinberg \cite[Prop.\,16]{MR0302779}}]\label{inj}
Let $\GG$ be a labeled $d$-polytope. Assume that two Coxeter $d$-polytopes $P$ and $P'$ realize $\GG$. Then the following are equivalent:
\begin{itemize}
\item the Coxeter polytopes $P$ and $P'$ are isomorphic;
\item the Cartan matrices $A_P$ and $A_{P'}$ are equivalent;
\item all the cyclic products of $A_{P}$ and $A_{P'}$ are equal.
\end{itemize}
\end{theorem}

To avoid redundant cyclic products, we are motivated to introduce the following: A $k$-circuit $\C$ of $W_S$ is \emph{relevant} if $\C$ corresponds to a cycle of the underlying graph of $\mathcal{D}_W$ or to an edge of label $\infty$ in $\mathcal{D}_W$. Note that for any $[P] \in \B(\GG)$,

\begin{itemize}
\item in the case $k=1$, $\C(A_P)$ is always $2$, which justifies that $\C$ is not relevant;
\item in the case $k=2$, if $\C$ is not relevant, then $\C(A_P) = 4 \cos^2(\nicefrac{\pi}{m})$ for a fixed $m$;
\item in the case $k \geqslant 3$, if $\C$ is not relevant, then $\C$ contains two consecutive elements $i, j$ such that $(A_P)_{ij} = (A_P)_{ji} = 0 $, so $\C(A_P) = 0$.
\end{itemize}

\medskip

A slightly modified cyclic product is more useful than the original one when $W_S$ has no edge of label $\infty$: let $ \C = (i_1, i_2, \dotsc, i_k)$ be a relevant $k$-circuit of $W_S$ and $\overline{\C} =  (i_k, i_{k-1} \dotsc, i_1)$, which we call the \emph{opposite} circuit of $\C$ or the circuit with \emph{opposite orientation}. A \emph{normalized} cyclic product of $\C$ is defined by:
$$ R_{\C}(A_P) = \log \left(   \frac{\C(A_P)}{\overline{\C}(A_P)} \right)$$
Since there is no edge of label $\infty$ in $\mathcal{D}_W$, the quantity $\C(A_P) \, \overline{\C}(A_P)$ is constant only depending on $W_S$. So the normalized cyclic product $R_{\C}(A_P)$ contains the same amount of information as $\C(A_P)$. Clearly, $R_{\C}(A_P) + R_{\overline{\C}}(A_P) = 0$.

\begin{rem}
The topological properties of the underlying graph $U_W$ of $\mathcal{D}_W$ are important, hence we say that the Coxeter group $W_S$ is \emph{of type “something”} if the graph $U_W$ is “something”. For example, the word “something” can be replaced by “tree”, “cycle”, and so on.
\end{rem}

\subsection{Tits simplices}\label{sub:tits_simplex}

Given a Coxeter group $W_S$, we build a labeled polytope $\Ss_W$ and a Coxeter polytope $\Delta_W$. Their underlying polytopes are simplex of dimension $\# S-1$. 

\medskip

The construction of $\Ss_W$ is straightforward. Suppose $W_S$ is associated to a Coxeter matrix $M$ on $S$. The underlying combinatorial polytope of $\Ss_W$ is simplex with $\# S$ facets, the set of facets of $\Ss_W$ identifies with $S$, and for every pair of distinct facets $s, t$ of $S$, the label of the ridge $s \cap t$ of $\Ss_W$ is $\nicefrac{\pi}{M_{st}}$. 

\medskip

We now construct the Coxeter simplex $\Delta_W$ of $\mathbb{S}(\mathbb{R}^S)$. A key observation is that to any Cartan matrix $A = (A_{st})_{s,t \in S}$ can be associated a Coxeter simplex $\Delta_A$ of $\mathbb{S}(\mathbb{R}^S)$ as follows:
\begin{itemize}
\item for each $t \in S$, we set $\alpha_t = e_t^*$, where $(e_t^*)_{t \in S}$ is the canonical dual basis of $\mathbb{R}^S$;
\item for each $t \in S$, we take the unique vector $b_t \in \mathbb{R}^S$ such that
 $ \alpha_s(b_t) = A_{st}$ for all $s \in S$;
\item the Coxeter simplex $\Delta_A$ is the pair of the projective simplex
 $$\bigcap_{s \in S} \mathbb{S}(\{x \in \mathbb{R}^S \mid \alpha_s(x) \leqslant 0 \})$$ and the set of reflections $(\sigma_s = \mathrm{Id} - \alpha_s \otimes b_s)_{s \in S}$.
\end{itemize}

\begin{rem}\label{remark:compatible}
An $S \times S$ Cartan matrix $A$ is \emph{compatible} with a Coxeter group $W_S$ provided that for every $s, t \in S$, $A_{st}A_{ts} = 4 \cos^2(\nicefrac{\pi}{M_{st}})$ if $M_{st} \neq \infty$, and $A_{st} A_{ts} \geqslant 4$ otherwise. If this is the case, then $[\Delta_A] \in \B(\Ss_{W})$. 
\end{rem}

Let $\Gamma_{A} := \Gamma_{\Delta_A}$ be the discrete subgroup of $\mathrm{SL}^{\pm}(\mathbb{R}^S)$ generated by the reflections $(\sigma_s)_{s\in S}$ (see Theorem \ref{theo_vinberg}). In particular, if $A$ is \emph{symmetric}, then by Vinberg \cite[Th. 6]{MR0302779}, there exists a $\Gamma_{A}$-invariant symmetric form $B_A$ on $V_A$ such that $B_A(b_s,b_t) = A_{st}$ for every $s, t \in S$, where $V_A$ is the subspace of $\mathbb{R}^S$ spanned by $(b_s)_{s \in S}$. For example, for any Coxeter group $W_S$, the $S \times S$ \emph{Cosine matrix} $\mathrm{Cos}(W)$ with entries
$$ (\mathrm{Cos}(W))_{st} = -2 \cos\left( \frac{\pi}{M_{st}} \right) $$
is a symmetric Cartan matrix compatible with $W_S$. We call $\Delta_W := \Delta_{\mathrm{Cos}(W)}$ the \emph{Tits simplex} of $W_S$ and $B_W := B_{\mathrm{Cos}(W)}$ the \emph{Tits symmetric form} of $W_S$. If $B_W$ is nondegenerate, then $\G_{W}$ is a subgroup of the orthogonal group $O(B_W)$ of the form $B_W$ on $\mathbb{R}^S$.

\begin{rem}
Each vertex $v$ of $\Ss_W$ has a unique opposite facet $s_v \in S$, since the polytope $\Ss_W$ is a simplex. The link of $\Ss_W$ at $v$ is isomorphic to $\Ss_{W_{S \smallsetminus \{ s_v \} }}$.
\end{rem}

A Coxeter group $W_S$ is \emph{Lannér} (resp. \emph{quasi-Lannér})\footnote{Sometimes quasi-Lannér Coxeter groups are called \emph{Koszul} Coxeter groups.} if it is large and $W_{S \smallsetminus \{ s \}} $ is spherical (resp. spherical or irreducible affine) for each $s \in S$. These are classical terms, and those Coxeter groups were classified by Lannér \cite{MR0042129}, Koszul \cite{LectHypCoxGrKoszul} and Chein \cite{MR0294181}. Note that quasi-Lannér Coxeter groups are irreducible. We now introduce a less classical terminology: a Coxeter group $W_S$ is \emph{2-Lannér} if it is irreducible, large, and $W_{S \smallsetminus \{ s,t \}} $ is spherical for every $s \neq t \in S$. The 2-Lannér Coxeter groups were classified by Maxwell \cite{MR679972} (see Theorem \ref{t:maxwell}). He actually enumerated the list of all \emph{Lorentzian} Coxeter groups $W_S$, \ie $W_{S \smallsetminus \{ s,t \}} $ is spherical or irreducible affine for every $s \neq t \in S$. The following easy lemmas justify our terminology:

\begin{lemma}\label{prop:label_simplex}
A labeled simplex $\Ss$ is perfect (resp. $2$-perfect), irreducible and large if and only if the Coxeter group $W_{\Ss}$ is Lannér (resp. 2-Lannér). 
\end{lemma}

\begin{lemma}\label{prop:Cox_simplex}
Let $W_S$ be an irreducible, large Coxeter group. Then $W_S$ is Lannér (resp. quasi-Lannér, resp. 2-Lannér) if and only if the Tits simplex $\Delta_W$ is perfect (resp. quasi-perfect, resp. 2-perfect). 
\end{lemma}

Lemma \ref{prop:label_simplex} implies that there exists a one-to-one correspondence between the irreducible, large, $2$-perfect labeled $d$-simplices and the 2-Lannér Coxeter groups of rank $d+1$.

\begin{rem}\label{remark:hyperbolic}
In general, the signature of the Tits symmetric form $B_W$ of a Coxeter group $W_S$ can be arbitrary. However, Maxwell \cite[Th.\,1.9]{MR679972} proved that if $W_S$ is 2-Lannér, then $B_W$ is nondegenerate and of signature $(p,1)$ with $p = \# S-1$. In other words, the group $\Gamma_\Delta$ generated by the reflections of $\Delta_W$ is conjugate to a discrete subgroup of $\mathrm{O}^+_{p,1}(\mathbb{R})$, which is isomorphic to $\mathrm{Isom}(\mathbb{H}^{p})$, and hence $\Delta_W$ is a hyperbolic Coxeter simplex.
\end{rem}

\begin{rem}\label{remark:loxo_imply}
If $P$ is a loxodromic perfect Coxeter $d$-simplex, then the Coxeter group $W_P$ is either Lannér or $\tilde{A}_{d}$ by Theorem \ref{tri} and Lemma \ref{prop:label_simplex}.
\end{rem}

\subsection{Classification of Lannér, quasi-Lannér and 2-Lannér Coxeter groups}

Recall that if a Coxeter group $W_S$ has rank $d+1$, then the labeled polytope $\Ss_W$ and the Tits simplex $\Delta_W$ of $W$ has dimension $d$.

\subsubsection*{Dimension $d = 1, 2, 3$}

It is obvious that there exists no Lannér Coxeter group of rank $2$. Every irreducible large Coxeter group $W_S$ of rank $3$ is quasi-Lannér, and it is Lannér if and only if the Coxeter diagram $\mathcal{D}_W$ has no edge of label $\infty$. An irreducible large Coxeter group $W_S$ of rank $4$ is 2-Lannér if and only if $\mathcal{D}_W$ has no edge of label $\infty$.

\subsubsection*{Dimension $d \geqslant 4$}

\begin{theorem}[Maxwell \cite{MR679972}]\label{t:maxwell}
Let $d \in \mathbb{N}$. If $4 \leqslant d \leqslant 9$, then there exist finitely many 2-Lannér Coxeter groups of rank $d+1$. The numbers of such Coxeter groups are given in Table \ref{tab:nb_2lanner}. The complete list can be found in Chen--Labbé \cite{Chen:2013fk} or Appendix \ref{classi_lorent}. If $d \geqslant 10$, then there is no 2-Lannér Coxeter group of rank $d+1$.
\end{theorem}

\newcommand{\siz}{0.15}
\newcommand{\size}{0.11}

\begin{table}[h]
\centering
\begin{tabular}{ccccc}
\toprule[\siz em]
Dimension  & $\#$ of & $\#$ of 2-Lannér & $\#$ of quasi-Lannér & $\#$ of Lannér\\
$d$     &             2-Lannér               & not quasi-Lannér   & not Lannér                 &  Coxeter groups\\
\toprule[\siz em]
4 & 45 & 31 & 9 & 5\\
5 & 23 & 11 & 12& 0\\
\hline
6 & 3 & 0 & 3 & 0\\
7 & 4 & 0 & 4 & 0\\
8 & 4 & 0 & 4 & 0\\
9 & 3 & 0 & 3 & 0\\
\toprule[\siz em]
\end{tabular}
\vspace*{1em}
\caption{The numbers of 2-Lannér Coxeter groups}
\label{tab:nb_2lanner}
\end{table}

\section{Deformation space of 2-perfect simplex}\label{section:moduli_of_simplex}

The aim of this section is to parametrize the deformation space $\B(\Ss)$ of an irreducible, large, $2$-perfect labeled simplex $\Ss$ of dimension $d \geqslant 4$. The parameterization is explicitly described in the proof of Theorem \ref{t:moduli_simplex}. Recall that a Coxeter $d$-polytope $P$ of $\S^d$ is \emph{hyperbolic} if the reflection group $\Gamma_P$ lies in a conjugate of $\mathrm{O}^+_{d,1}(\mathbb{R}) \subset \mathrm{SL}^{\pm}_{d+1}(\mathbb{R})$. 

\begin{theorem}\label{t:moduli_simplex}
Let $\Ss$ be an irreducible, large, $2$-perfect labeled simplex of dimension $d \geqslant 4$, and $W$ its Coxeter group. If $e_+$ denotes the number of edges of the Coxeter diagram $\mathcal{D}_W$, then the deformation space $\B(\Ss)$ is an open cell of dimension $b(\Ss) = e_+ \! - d \in \{0, 1, 2 \}$. Moreover, $\B(\Ss)$ contains exactly one isomorphism class of hyperbolic Coxeter $d$-simplex, which is the Tits simplex $\Delta_W$ of $W$.
\end{theorem}

In the case that $\Ss$ is perfect, the similar statement for Theorem \ref{t:moduli_simplex} can be found in Nie \cite{MR3449173}. The proof essentially follows from a simple computation, together with some results of Vinberg \cite{MR0302779} and classification Theorem \ref{t:maxwell}.

\begin{proof}
By Lemma \ref{prop:label_simplex}, the Coxeter group $W$ is a 2-Lannér Coxeter group of rank $\geqslant 5$. Hence, Theorem \ref{t:maxwell} implies that $W$ is of type either tree, cycle, pan or $K_{2,3}$ (see Appendix \ref{classi_lorent} and Figure \ref{fig:terminology}). 

\begin{figure}[ht!]
\centering
\begin{tabular}{ccc}
\begin{tikzpicture}[thick,scale=0.7, every node/.style={transform shape}] 
\node[draw,circle] (A) at (-90:1){};
\node[draw,circle] (B) at (-90+72:1){};
\node[draw,circle] (C) at (-90+144:1){};
\node[draw,circle] (D) at (-90+216:1){};
\node[draw,circle] (E) at (-90-72:1){};  
\draw (A) -- (B) node[above,midway] {};
\draw (B) -- (C) node[above,midway] {};
\draw (C) -- (D) node[above,midway] {};
\draw (D) -- (E) node[above,midway] {};
\draw (E) -- (A) node[above,midway] {};
\end{tikzpicture}
\quad\quad\quad
&
\begin{tikzpicture}[thick,scale=0.7, every node/.style={transform shape}] 
\node[draw,circle] (A) at (0:1){};
\node[draw,circle] (B) at (90:1){};
\node[draw,circle] (C) at (180:1){};
\node[draw,circle] (D) at (-90:1){};
\node[draw,circle] (E) at (-1-1.414,0){};
\draw (A) -- (B) node[above,midway] {};
\draw (B) -- (C) node[above,midway] {};
\draw (C) -- (D) node[above,midway] {};
\draw (D) -- (A) node[above,midway] {};
\draw (E) -- (C) node[above,midway] {};
\end{tikzpicture}
\quad\quad\quad
&
\begin{tikzpicture}[thick,scale=0.7, every node/.style={transform shape}] 
\node[draw,circle] (G1) at (0:1){};
\node[draw,circle] (G2) at (90:1) {};
\node[draw,circle] (G3) at (180:1) {}; 
\node[draw,circle] (G4) at (-90:1) {};
\node[draw,circle] (G5) at (0,0) {};
\draw (G1) -- (G2)  node[above,midway] {};
\draw (G2) -- (G3)  node[above,midway] {};
\draw (G3) -- (G4) node[above,midway] {};
\draw (G4) -- (G1) node[above,midway] {};
\draw (G4) -- (G5) node[above,midway] {};
\draw (G5) -- (G2) node[above,midway] {};
\end{tikzpicture}
\\
\end{tabular}
\caption{A $5$-cycle, a $4$-pan and $K_{2,3}$ from left to right}
\label{fig:terminology}
\end{figure}
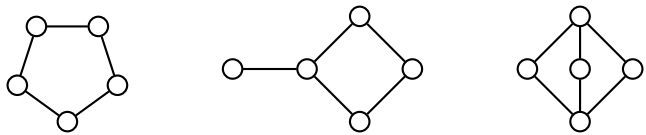

\medskip

Since the Coxeter diagram $\mathcal{D}_W$ has no edge of label $\infty$, we may use \textit{normalized} cyclic products instead of cyclic products to parametrize the space $\B(\Ss)$. We now claim that:
\begin{enumerate}
\item if $W$ is of tree type, then $\B(\Ss)$ is a singleton;
\item if $W$ is of cycle type or pan type, then $\B(\Ss)$ is homeomorphic to $\mathbb{R}$;
\item if $W$ is of $K_{2,3}$ type, then $\B(\Ss)$ is homeomorphic to $\mathbb{R}^2$.
\end{enumerate}

\begin{enumerate}\setlength{\itemsep}{5pt}
\item \label{para:tree} In the case of tree type, there is no relevant circuit of $W$. So, Theorem \ref{inj} implies that $\B(\Ss) = \{ [\Delta_W] \}$.

\item\label{para:circle}
In the case of cycle or pan type, there exist only two relevant circuits $\C$ and $\overline{\C}$ in $W$. If $W$ is of cycle type (resp. of pan type), then $\C$ is a $(d+1)$-circuit (resp. $d$-circuit). The map $R : \B(\Ss) \to \R$ defined by $R([P]) = R_{\C}(A_{P})$, the normalized cyclic product of $\C$, is a homeomorphism since $R$ is injective and surjective respectively by Theorem \ref{inj} and Remark \ref{remark:compatible}. Clearly $R([\Delta_W]) = 0$.

\item\label{para:davy}
In the case of $K_{2,3}$ type, there exists three pairs of relevant circuits $\{(\C_i,\overline{\C}_i)\}_{i = 1,2,3}$. Such circuits have length $d$ and $d=4$. For each $[P] \in \B(\Ss)$, we denote by $R ([P]) \in \mathbb{R}^3 $ the triple 
of the normalized cyclic products $(R_{\C_i}(A_{P}))_{i = 1,2,3}$. We choose the orientations of $\{ \C_i \}_{i =1,2,3}$ coherently so that $ \sum_{i=1}^3 R_{\C_i}(A_P) = 0$. 
For example, if $\C_1 = (2, 5, 4, 3)$, $\C_2 = (1, 4, 5, 2)$ and $\C_3 = (1, 2, 3, 4)$ as in the left diagram of Figure \ref{fig:coherent_orientations}, then:
\begin{align*}
& R_{\C_1}(A_P) + R_{\C_2}(A_P) + R_{\C_3}(A_P)
\\ 
 & =  \log \left(   \frac{A_{25}A_{54}A_{43}A_{32}}{A_{23}A_{34}A_{45}A_{52}} \right) + \log \left(   \frac{A_{14}A_{45}A_{52}A_{21}}{A_{12}A_{25}A_{54}A_{41}} \right) + \log \left(   \frac{A_{12}A_{23}A_{34}A_{41}}{A_{14}A_{43}A_{32}A_{21}} \right) 
= 0
\end{align*}

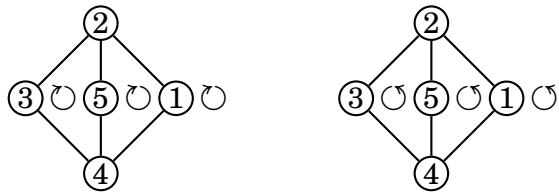
\begin{figure}[ht!]
\centering
\begin{tabular}{cc}
\begin{tikzpicture}[thick,scale=1, every node/.style={transform shape}] 
\node[draw,circle, inner sep=1pt] (G1) at (0:1){1};
\node[draw,circle, inner sep=1pt] (G2) at (90:1) {2};
\node[draw,circle, inner sep=1pt] (G3) at (180:1) {3}; 
\node[draw,circle, inner sep=1pt] (G4) at (-90:1) {4};
\node[draw,circle, inner sep=1pt] (G5) at (0,0) {5};
\node[draw=none] (G6) at (0.5,0) {$\circlearrowright$};
\node[draw=none] (G7) at (-0.5,0) {$\circlearrowright$};
\node[draw=none] (G8) at (1.5,0) {$\circlearrowright$};
\draw (G1) -- (G2) node[above,midway] {};
\draw (G2) -- (G3) node[above,midway] {};
\draw (G3) -- (G4) node[above,midway] {};
\draw (G4) -- (G1) node[above,midway] {};
\draw (G4) -- (G5) node[above,midway] {};
\draw (G5) -- (G2) node[above,midway] {};
\end{tikzpicture}
\quad \quad \quad 
\begin{tikzpicture}[thick,scale=1, every node/.style={transform shape}] 
\node[draw,circle, inner sep=1pt] (G1) at (0:1){1};
\node[draw,circle, inner sep=1pt] (G2) at (90:1) {2};
\node[draw,circle, inner sep=1pt] (G3) at (180:1) {3}; 
\node[draw,circle, inner sep=1pt] (G4) at (-90:1) {4};
\node[draw,circle, inner sep=1pt] (G5) at (0,0) {5};
\node[draw=none] (G6) at (0.5,0) {$\circlearrowleft$};
\node[draw=none] (G7) at (-0.5,0) {$\circlearrowleft$};
\node[draw=none] (G8) at (1.5,0) {$\circlearrowleft$};
\draw (G1) -- (G2) node[above,midway] {};
\draw (G2) -- (G3) node[above,midway] {};
\draw (G3) -- (G4) node[above,midway] {};
\draw (G4) -- (G1) node[above,midway] {};
\draw (G4) -- (G5) node[above,midway] {};
\draw (G5) -- (G2) node[above,midway] {};
\end{tikzpicture}
\end{tabular}
\caption{Two choices of coherent orientations}
\label{fig:coherent_orientations}
\end{figure}

\noindent Theorem \ref{inj} and Remark \ref{remark:compatible} again show that the map $R : \B(\Ss) \to H$ is a homeomorphism, where $H = \{ (x_1, x_2, x_3) \in \mathbb{R}^3 \mid x_1 + x_2 + x_3 = 0 \}$. Again, $R([\Delta_W]) = 0$. 
\end{enumerate}

\medskip

By Remark \ref{remark:hyperbolic}, the Tits simplex $[\Delta_W]$ in $\B(\Ss)$ is hyperbolic. Moreover, if the Coxeter simplex $[P] \in \B(\Ss)$ is hyperbolic, then $A_P$ is symmetrizable hence $R([P])=0$. Then by the previous paragraph, $[P] = [\Delta_W]$. Finally, observe that $e_+\! - d$ is the dimension of $\B(\Ss)$.
\end{proof}

The parametrization described in the proof will be used in the sequel. 

\begin{rem}\label{rem:color1}
A labeled polytope $\GG$ (resp. a Coxeter group $W_S$) is \emph{rigid} if $\B(\GG)$ (resp. $\B(\Ss_{W})$) is a singleton. Otherwise, it is \emph{flexible}. For a $2$-Lannér Coxeter group $W_S$, a node $s \in S$ is “{\em something}” if $W_{S \smallsetminus \{ s \} }$ is “something”. In Appendix \ref{classi_lorent}, some important properties of the nodes are encoded in color: a node $s \in S$ is colored in black, orange, blue or green when $W_{S \smallsetminus \{ s \} }$ is rigid affine, flexible affine, rigid Lannér, or flexible Lannér respectively. In other words, a node is $\tilde{A}_n$ for some $n \geqslant 2$ (resp. Lannér) if and only if it is colored in orange (resp. green or blue). A Lannér node is of cycle type (resp. tree type) if and only if it is green (resp. blue).  
\end{rem}

\medskip

\begin{center}
\begin{tabular}{| c | c | }
\hline
Color & Property of the node $s \in S$ \\
\hline
White & Spherical \\
\hline
Black & Irreducible affine of tree type,\\
        & \ie not $\tilde{A}_n$, so $\Ss_{W_{ S \smallsetminus \{ s \}}}$ is rigid.\\
\hline
Orange & Irreducible affine of cycle type,\\
      & \ie $\tilde{A}_n$, so $\Ss_{ W_{ S \smallsetminus \{ s \}} }$ is flexible.\\
\hline
Blue & Lannér of tree type,\\
       & so $ \Ss_{ W_{ S \smallsetminus \{ s \}} }$ is rigid.\\
\hline
Green & Lannér of cycle type,\\
    & so $\Ss_{ W_{ S \smallsetminus \{ s \}} }$ is flexible.\\
\hline
\end{tabular}
\end{center}

\medskip

\section{Truncation and stacking}\label{section:trucation_stacking}

\subsection{Truncation and stacked polytopes}\label{subsec:truncation}

Let $\GG$ be a combinatorial polytope, and let $v$ be a vertex of $\GG$. A \emph{truncation} of $\GG$ at $v$ is the operation that cuts the vertex $v$, creating a new facet $s$ in place of $v$ (see Figure \ref{fig:truncation}). We denote by $\GG^{\dagger v}$ the polytope obtained by the truncation of $\GG$ at $v$. A polytope is a \emph{truncation $d$-polytope} if it is built from a $d$-simplex by successively truncating vertices. For example, a \emph{$d$-prism} is a truncation $d$-polytope obtained by truncating a $d$-simplex at one vertex. 

\medskip

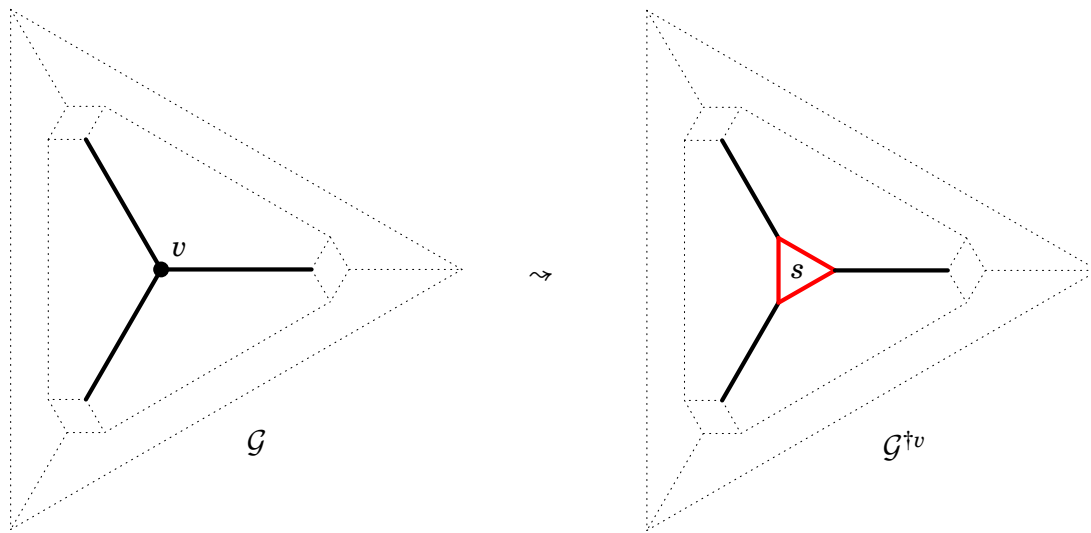
\begin{figure}[ht!]
\centering
\begin{tabular}{>{\centering\arraybackslash}m{.4\textwidth} c >{\centering\arraybackslash}m{.4\textwidth}}
\begin{tikzpicture}[line cap=round,line join=round,>=triangle 45,x=1.0cm,y=1.0cm]
\clip(-2.5,-4) rectangle (4.5,4);
\draw [line width=1.6pt] (-1,1.73)-- (0,0);
\draw [line width=1.6pt] (0,0)-- (-1,-1.73);
\draw [line width=1.6pt] (0,0)-- (2,0);
\draw [dotted] (-1.5,1.73)-- (-1,1.73);
\draw [dotted] (-1,1.73)-- (-0.75,2.17);
\draw [dotted] (2,0)-- (2.25,0.43);
\draw [dotted] (2,0)-- (2.25,-0.43);
\draw [dotted] (-1,-1.73)-- (-1.5,-1.73);
\draw [dotted] (-1,-1.73)-- (-0.75,-2.17);
\draw [dotted] (-1.5,1.73)-- (-1.25,2.17);
\draw [dotted] (-1.25,2.17)-- (-0.75,2.17);
\draw [dotted] (2.25,0.43)-- (2.5,0);
\draw [dotted] (2.5,0)-- (2.25,-0.43);
\draw [dotted] (-1.5,-1.73)-- (-1.25,-2.17);
\draw [dotted] (-1.25,-2.17)-- (-0.75,-2.17);
\draw [dotted] (-1.5,1.73)-- (-1.5,-1.73);
\draw [dotted] (-0.75,-2.17)-- (2.25,-0.43);
\draw [dotted] (2.25,0.43)-- (-0.75,2.17);
\draw [dotted] (-2,3.46)-- (-1.25,2.17);
\draw [dotted] (2.5,0)-- (4,0);
\draw [dotted] (-1.25,-2.17)-- (-2,-3.46);
\draw [dotted] (-2,3.46)-- (-2,-3.46);
\draw [dotted] (-2,-3.46)-- (4,0);
\draw [dotted] (4,0)-- (-2,3.46);
\draw (0,0.5) node[anchor=north west] {$v$};
\draw (1,-2) node[anchor=north west] {$\GG$};

\fill [color=black] (0,0) circle (3pt);
\end{tikzpicture}
&
$\quad \rightsquigarrow \quad$
&
\definecolor{ffqqqq}{rgb}{1,0,0}
\begin{tikzpicture}[line cap=round,line join=round,>=triangle 45,x=1.0cm,y=1.0cm]
\clip(-2.49,-3.99) rectangle (4.5,4.01);
\draw [dotted] (-1.5,1.73)-- (-1,1.73);
\draw [dotted] (-1,1.73)-- (-0.75,2.17);
\draw [dotted] (2,0)-- (2.25,0.43);
\draw [dotted] (2,0)-- (2.25,-0.43);
\draw [dotted] (-1,-1.73)-- (-1.5,-1.73);
\draw [dotted] (-1,-1.73)-- (-0.75,-2.17);
\draw [dotted] (-1.5,1.73)-- (-1.25,2.17);
\draw [dotted] (-1.25,2.17)-- (-0.75,2.17);
\draw [dotted] (2.25,0.43)-- (2.5,0);
\draw [dotted] (2.5,0)-- (2.25,-0.43);
\draw [dotted] (-1.5,-1.73)-- (-1.25,-2.17);
\draw [dotted] (-1.25,-2.17)-- (-0.75,-2.17);
\draw [dotted] (-1.5,1.73)-- (-1.5,-1.73);
\draw [dotted] (-0.75,-2.17)-- (2.25,-0.43);
\draw [dotted] (2.25,0.43)-- (-0.75,2.17);
\draw [dotted] (-2,3.46)-- (-1.25,2.17);
\draw [dotted] (2.5,0)-- (4,0);
\draw [dotted] (-1.25,-2.17)-- (-2,-3.46);
\draw [dotted] (-2,3.46)-- (-2,-3.46);
\draw [dotted] (-2,-3.46)-- (4,0);
\draw [dotted] (4,0)-- (-2,3.46);
\draw [line width=1.6pt] (-1,1.73)-- (-0.25,0.43);
\draw [line width=1.6pt,color=ffqqqq] (-0.25,0.43)-- (-0.25,-0.43);
\draw [line width=1.6pt] (-0.25,-0.43)-- (-1,-1.73);
\draw [line width=1.6pt,color=ffqqqq] (-0.25,-0.43)-- (0.5,0);
\draw [line width=1.6pt,color=ffqqqq] (-0.25,0.43)-- (0.5,0);
\draw [line width=1.6pt] (0.5,0)-- (2,0);
\draw (-0.23,0.23) node[anchor=north west] {$s$};
\draw (1,-2) node[anchor=north west] {$\GG^{\dagger v}$};

\end{tikzpicture}
\end{tabular}
\caption{A truncation of a polytope $\GG$ at a vertex $v$}
\label{fig:truncation}
\end{figure}

For a \emph{labeled} polytope $\GG$, after truncating a vertex $v$ of $\GG$, we additionally attach the labels $\nicefrac{\pi}{2}$ to all the new ridges of $\GG^{\dagger v}$ to obtain a new labeled polytope, which we denote again by $\GG^{\dagger v}$. Similarly, given a set $\mathcal{V}$ of some vertices of $\GG$, we denote by $\GG^{\dagger \mathcal{V}}$ the labeled polytope obtained by successively truncating all the vertices $v\in \mathcal{V}$. 

\begin{rem}\label{rem:2perfectAfterTruncation}
Let $\GG$ be a labeled polytope and let $v$ be a vertex of $\GG$. Each vertex $w$ in the new facet of $\GG^{\dagger v}$ corresponds to a vertex $w'$ of the link $\GG_v$ of $\GG$ at $v$, and 
$$W_{(\GG^{\dagger v})_w} = W_{(\GG_v)_{w'}} \times \Quotient{\mathbb{Z}}{2\mathbb{Z}}$$ So, every vertex in the new facet of $\GG^{\dagger v}$ is elliptic if and only if $\GG_v$ is perfect. As a consequence, if $\GG$ is $2$-perfect then so is $\GG^{\dagger v}$.
\end{rem}

The dual concept of truncation is useful for the later discussion: two combinatorial $d$-polytopes $\GG$ and ${\GG}^{*}$ are \emph{dual} to each other if there exists an inclusion-{\em reversing} bijection $\phi$ between the face posets $\F(\GG)$ and $\F({\GG}^{*})$. The map $\phi$ is called the \emph{dual isomorphism} between $\GG$ and $\GG^*$.

\medskip

Let $\GG$ be a combinatorial polytope and $s$ a facet of $\GG$. A \emph{stacking} of $\GG$ at $s$ is the gluing of a pyramid $\mathcal{Y}$ with base $s$ onto the facet $s$ of $\GG$, where the apex of $\mathcal{Y}$ lies in the interior of the region bounded by the supporting hyperplanes of the facet $s$ and of the facets of $\GG$ adjacent to $s$ (see Figure \ref{fig:staked}).

\begin{figure}[ht!]
\centering
\begin{tabular}{>{\centering\arraybackslash}m{.3\textwidth} c >{\centering\arraybackslash}m{.3\textwidth}}
\includegraphics[scale=.7]{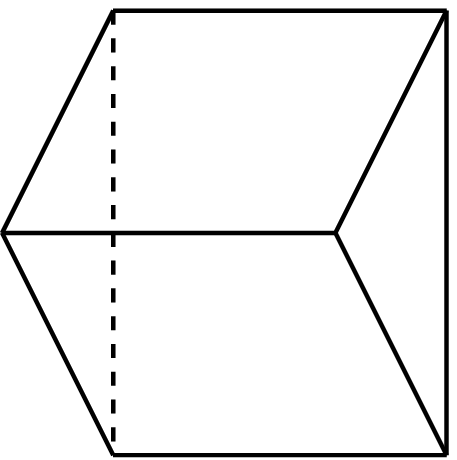}
\small
\put (-13, 43){$s$}
\put (-90, 10){$\GG$}
&
$\rightsquigarrow $
&
\includegraphics[scale=.7]{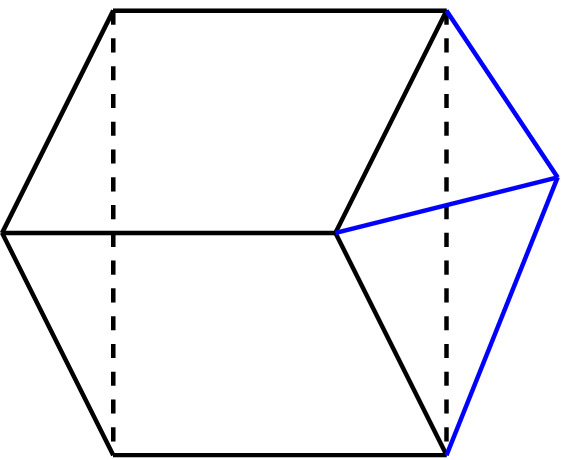}
\\
\end{tabular}
\caption{A stacking of a polytope $\GG$ at a facet $s$}
\label{fig:staked}
\end{figure}

The truncation of $\GG$ at the vertex $v$ is dual to the stacking of the dual polytope ${\GG}^{*}$ at the facet $\phi(v)$, which is dual to $v$. So the polytope $\GG^{\dagger v}$ is dual to the polytope obtained from ${\GG}^{*}$ by stacking at the facet $\phi(v)$. A $d$-polytope is a \emph{stacked polytope} if it is built from the $d$-simplex by a finite number of stacking operations.

\begin{rem}
 A stacked $d$-polytope has a natural triangulation given by the successive stacking operations. This triangulation satisfies the following property $(\star)$: \textit{all the interior faces\footnote{A face $f$ of a triangulation of $\GG$ is {\em interior} if the relative interior of $f$ lies in the interior of $\GG$.} of the triangulation are of codimensions $0$ or $1$.} We know from work of Kleinschmidt \cite{MR0474043} that if $d$ is bigger than $2$, then there exists a \emph{unique} triangulation of a stacked $d$-polytope satisfying $(\star)$. Hence from now on, we call this triangulation the \emph{stacking triangulation} or simply the \emph{triangulation}. 
\end{rem}

\subsection{Truncatable vertex}

We introduce the “geometric” truncation of a Coxeter polytope, which is comparable with the “combinatorial” truncation of a labeled polytope in Section \ref{subsec:truncation}.

\begin{de}
Let $P$ be a Coxeter polytope of $\S^d$, $v$ a vertex of $P$ and $S_v$ the set of facets of $P$ containing $v$. The vertex $v$ of $P$ is \emph{truncatable} if the projective subspace $\Pi_v$ spanned by the poles $\left\{ [b_s] \right\}_{s \in S_v}$ is a hyperplane such that for each edge $e$ containing $v$, the intersection of $\Pi_v$ and the relative interior of $e$ is a singleton.
\end{de}

Suppose $P$ is a Coxeter polytope and $v$ is a truncatable vertex of $P$. We define a new Coxeter polytope $P^{\dagger v}$ as follows: let $\Pi_v^{+}$ (resp. $\Pi_v^{-}$) be the connected component of $\mathbb{S}^d \smallsetminus \Pi_v$ which contains $v$ (resp. which does not contain $v$), and let $\overline{\Pi_v^{-}}$ be the closure of $\Pi_v^{-}$. The underlying polytope of $P^{\dagger v}$ is $P \cap \overline{\Pi_v^{-}}$, which has one \emph{new facet} given by the hyperplane $\Pi_v$ and the \emph{old facets} given by $P$.
The reflection in the new facet of $P^{\dagger v}$ is determined by the support $\Pi_v$ and the pole $v$, and the reflections in the old facets are unchanged. The following properties of $P^{\dagger v}$ can be easily checked:
\begin{itemize}
\item the dihedral angles of the ridges in the new facet of $P^{\dagger v}$ are all $\nicefrac{\pi}{2}$;
\item the hyperplane $\Pi_v$ is preserved by the reflections in the facets in $S_v$, and $P \cap \Pi_v$ is a Coxeter polytope in $\Pi_v$, which is isomorphic to $P_v$.
\end{itemize}

\begin{de}
Let $P$ be a Coxeter polytope and let $\V$ be a set of some vertices of $P$. The set $\V$ is \emph{truncatable} if each vertex $v\in \V$ is truncatable and $P\cap \Pi_v \cap \Pi_w = \varnothing$ for any two vertices $v \neq w \in \V$. In other words, the new facets do not intersect each other.
\end{de}

The following theorem provides a simple criterion when $\V$ is truncatable or not. Recall that a vertex $v$ of a polytope $\GG$ is \emph{simple} if the link $\GG_v$ is a simplex, and a polytope is \emph{simple} if all its vertices are simple. 

\begin{theorem}[{Marquis \cite[Prop.\,4.14 \& Lem.\,4.17]{Marquis:2014aa}}]\label{thm:truncation}
Let $P$ be an irreducible, loxodromic, 2-perfect Coxeter polytope. Assume that $\V$ is a set of some simple vertices of $P$. Then $\V$ is truncatable if and only if the vertex link $P_v$ is loxodromic for each $v \in \V$.
\end{theorem}

\begin{rem}
Let $\GG$ be a labeled polytope. A vertex $v$ of $\GG$ is “{\em something}” if the link $\GG_v$ or its Coxeter group $W_{\GG_v}$ is “something”. For example, the word “something” can be replaced by “Lannér”, “$\tilde{A}$” and so on.
\end{rem}

Recall that if $W$ is of cycle type and $\mathcal{D}_W$ has no edge of label $\infty$, then $W$ has a unique pair $(\mathcal{C}, \overline{\mathcal{C}})$ of relevant circuits. The following is a consequence of Remark \ref{remark:loxo_imply} and Theorem \ref{thm:truncation}.

\begin{cor}\label{cor:truncatable}
Let $\GG$ be an irreducible, large, 2-perfect labeled simple polytope of dimension $d \geqslant 4$, and let $v$ be a vertex of $\GG$. Assume that $[P] \in \B(\GG)$. Then:
\begin{itemize}
\item if $v$ is a truncatable vertex of $P$, then $v$ is Lannér or $\tilde{A}_{d-1}$;
\item if $v$ is Lannér, then $v$ is a truncatable vertex of $P$;
\item if $v$ is $\tilde{A}_{d-1}$, then $v$ is a truncatable vertex of $P$ if and only if the normalized cyclic product $R_{\mathcal{C}_v}(A_P) \neq 0$, where $\mathcal{C}_v$ is a relevant circuit of $W_{\GG_v}$.
\end{itemize}
\end{cor}

The following is now immediate:

\begin{cor}\label{cor:moduli_truncation}
Let $\GG$ be an irreducible, large, 2-perfect labeled simple polytope of dimension $d \geqslant 4$, $\V_{A}$ the set of all $\tilde{A}_{d-1}$ vertices of $\GG$, and $\V$ a set of some Lann\'{e}r or $\tilde{A}_{d-1}$ vertices of $\GG$. Define  
$$
\B(\GG)^{\dagger \V} = \{ [P] \in \B(\GG) \mid R_{\C_v}(A_P) \neq 0 \, \textrm{ for each } v \in \V \cap \V_{A}  \}.
$$
Then the map $\B(\GG)^{\dagger \V} \rightarrow \B(\GG^{\dagger \V})$ induced by the truncation is a homeomorphism. In particular, if the link $\GG_v$ is Lannér for each $v \in \V$, then $\B(\GG)$ is homeomorphic to $\B(\GG^{\dagger \V})$ .
\end{cor}

Finally, we parametrize the deformation spaces of \emph{once-truncated simplices}, which are polytopes obtained from a simplex $\Ss$ by truncating each vertex of $\Ss$ at most once.

\begin{prop}\label{prop:deformation_block}
Let $\Ss$ be an irreducible, large, 2-perfect labeled simplex of dimension $d \geqslant 4$ and $W$ its Coxeter group. Let $\mathcal{V}_{LA}$ be the set of Lann\'{e}r or $\tilde{A}_{d-1}$ vertices of $\Ss$, and $\mathcal{V}$ a subset of $\mathcal{V}_{LA}$. Then:
\begin{itemize}
\item if $W$ is the left Coxeter group in Figure \ref{fig:exampleA} and $\mathcal{V} = \mathcal{V}_{LA}$, then $\B(\Ss^{\dagger \mathcal{V}})$ is the union of six open cells of dimension $b(\Ss)=2$.
\item otherwise, $\B(\Ss^{\dagger \mathcal{V}})$ is the union of $2^{k_A}$ open cells of dimension $b(\Ss) \in\{ 0,1,2 \}$, where $k_A$ is the number of $\tilde{A}_{d-1}$ vertices in $\mathcal{V}$.
\end{itemize}
In particular, $\B(\Ss^{\dagger \mathcal{V}})$ is non-empty.
\end{prop}

\begin{proof}
Assume that $W$ is the left Coxeter group in Figure \ref{fig:exampleA} and $\mathcal{V} = \mathcal{V}_{LA}$. By (the proof of) Theorem \ref{t:moduli_simplex} and Corollary \ref{cor:moduli_truncation}, the space $\B(\Ss^{\dagger \mathcal{V}})$ is homeomorphic to
$$
\{ (x_1, x_2, x_3) \in \R^3 \mid x_1 + x_2 + x_3 = 0 \textrm{ and } x_1, x_2, x_3 \neq 0 \}.
$$
Thus it is the union of six open cells of dimension $2$. The proof of the other cases also follows from Theorem \ref{t:moduli_simplex} and Corollary \ref{cor:moduli_truncation}. For example, if $W$ is the right Coxeter group in Figure \ref{fig:exampleA} and $\mathcal{V} = \mathcal{V}_{LA}$, then $\B(\Ss^{\dagger \mathcal{V}})$ is homeomorphic to
$$
\{ (x_1, x_2, x_3) \in \R^3 \mid x_1 + x_2 + x_3 = 0 \textrm{ and } x_1 \neq 0 \},
$$
which is the union of two open cells of dimension $2$.
\end{proof}

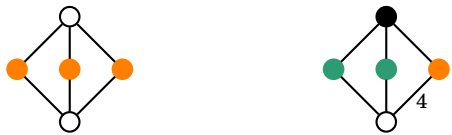
\begin{figure}[ht!]
\begin{minipage}{0.2\textwidth}
\centering
\begin{tikzpicture}[thick,scale=0.7, every node/.style={transform shape}] 
\node[draw=orange,circle,fill=orange] (G1) at (0:1){};
\node[draw,circle] (G2) at (90:1) {};
\node[draw=orange,circle,fill=orange] (G3) at (180:1) {}; 
\node[draw,circle] (G4) at (-90:1) {};
\node[draw=orange,circle,fill=orange] (G5) at (0,0) {};
\draw (G1) -- (G2)  node[above,midway] {};
\draw (G2) -- (G3)  node[above,midway] {};
\draw (G3) -- (G4) node[above,midway] {};
\draw (G4) -- (G1) node[above,midway] {};
\draw (G4) -- (G5) node[above,midway] {};
\draw (G5) -- (G2) node[above,midway] {};
\end{tikzpicture}
\end{minipage}
\qquad
\begin{minipage}{0.2\textwidth}
\centering
\begin{tikzpicture}[thick,scale=0.7, every node/.style={transform shape}] 
\node[draw=orange,circle,fill=orange] (G1) at (0:1){};
\node[draw,circle,fill=black] (G2) at (90:1) {};
\node[draw=green_okabe_ito,circle,fill=green_okabe_ito] (G3) at (180:1) {}; 
\node[draw,circle] (G4) at (-90:1) {};
\node[draw=green_okabe_ito,circle,fill=green_okabe_ito] (G5) at (0,0) {};
\draw (G1) -- (G2)  node[above,midway] {};
\draw (G2) -- (G3)  node[above,midway] {};
\draw (G3) -- (G4) node[above,midway] {};
\draw (G4) -- (G1) node[below,near end] {$4$};
\draw (G4) -- (G5) node[above,midway] {};
\draw (G5) -- (G2) node[above,midway] {};
\end{tikzpicture}
\end{minipage}
\caption{Examples of 2-Lann\'{e}r Coxeter groups of type $K_{2,3}$}
\label{fig:exampleA}
\end{figure}

\section{Gluing and splitting theorem}\label{section:gluing_splitting}

\subsection{Definitions}\label{subsec:gluing:definition}

Let $\GG_1$ and $\GG_2$ be two combinatorial polytopes. For each $i \in \{1, 2\}$, let $v_i$ be a vertex of $\GG_i$, and $({\GG}_{i})_{v_i}$ the link of $\GG_i$ at $v_i$. Given an isomorphism $\phi$ between $({\GG}_{1})_{v_1}$ and $({\GG}_{2})_{v_2}$, we define the \emph{gluing of $\GG_1^{\dagger v_1}$ and $\GG_2^{\dagger v_2}$ via $\phi$} as follows (see Figure \ref{fig:gluing}): 
\begin{itemize}
\item take the dual polytope $\GG^*_i$ of $\GG_i$ for each $i \in \{1,2\}$;
\item glue $\GG^*_1$ and $\GG^*_2$ using the induced isomorphism $\phi^*$ of $\phi$ between the dual facets of $v_1$ and $v_2$ to obtain a new polytope $\GG^* = \GG^*_1 \cup_{\phi^*} \GG^*_2$;
\item take the dual polytope $\GG$ of $\GG^*$.
\end{itemize}

\begin{figure}[ht!]
\centering
\begin{tabular}{>{\centering\arraybackslash}m{.3\textwidth} c >{\centering\arraybackslash}m{.3\textwidth} }
\includegraphics[scale=.7]{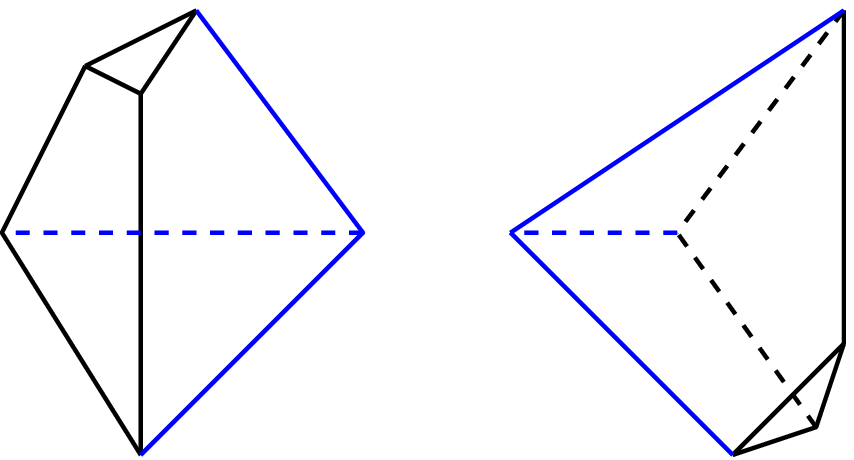}
\small
\put (-161, 5){$\GG_1$}
\put (-47, 5){$\GG_2$}
\put (-96,43){$v_1$}
\put (-80,43){$v_2$}
\put (70, 5){$\GG$}
&
\quad \quad \quad \quad $\rightsquigarrow$ 
&
\includegraphics[scale=.7]{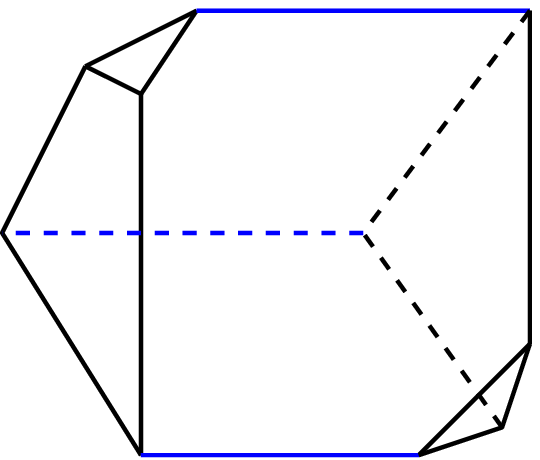} 
\\
\end{tabular}
\caption{Gluing two polytopes $\GG_1^{\dagger v_1}$ and $\GG_2^{\dagger v_2}$ to obtain $\GG$}
\label{fig:gluing}
\end{figure}

We say that $\GG$ is obtained by \emph{gluing $\GG_1^{\dagger v_1}$ and $\GG_2^{\dagger v_2}$ via $\phi$}, and we denote it by $\GG_1^{\dagger v_1} \sharp_{\phi} \GG_2^{\dagger v_2}$. Note that if $\GG_1$ and $\GG_2$ are truncation polytopes, then the dual facet of $v_i$ in the staked polytope $\GG^*_i$ become an interior face of codimension 1 in the triangulation of $\GG^*$. 

\begin{rem}
More directly, we may construct the polytope $\GG$ by gluing the truncated polytopes $\GG_i^{\dagger v_i}$ of $\GG_i$ at $v_i$ via the induced isomorphism of $\phi$ between the new facets of $\GG_i^{\dagger v_i}$ so that each old facet of $\GG_1^{\dagger v_1}$ formerly containing $v_1$ amalgamates with the corresponding old facet of $\GG_2^{\dagger v_2}$ formerly containing $v_2$. 
\end{rem}

Let $\GG$ be a combinatorial $d$-polytope, and let $\psi$ be the dual isomorphism between $\GG$ and $\GG^*$. A {\em prismatic poset} of $\GG$ is a subposet $\hat{\delta}$ of the face poset $\mathcal{F}(\GG)$ of $\GG$ such that:
\begin{itemize}
\item the subposet $\psi(\hat{\delta})$ of $\mathcal{F}(\GG^*)$ is isomorphic to the face poset of the boundary of $(d-1)$-simplex;
\item there is no $f \in \mathcal{F}(\GG^*)$ such that the poset $\psi(\hat{\delta}) \cup \{ f \}$ is isomorphic to the face poset of $(d-1)$-simplex.
\end{itemize}

The set $\delta$ of facets in a prismatic poset is called a \emph{prismatic circuit}. In more geometric way, $\delta$ consists of exactly $d$ facets of $\GG$ such that: 
\begin{itemize}
\item the convex hull $\Delta_\delta$ of the dual vertices of $\delta$ in $\GG^*$ is a $(d-1)$-simplex; 
\item the relative interior of $\Delta_\delta$ lies in the interior of $\GG^*$ (see Figure \ref{fig:prismatic}). 
\end{itemize}

\begin{figure}[ht!]
\centering
\begin{tabular}{>{\centering\arraybackslash}m{.3\textwidth} c >{\centering\arraybackslash}m{.3\textwidth}}
\includegraphics[scale=.7]{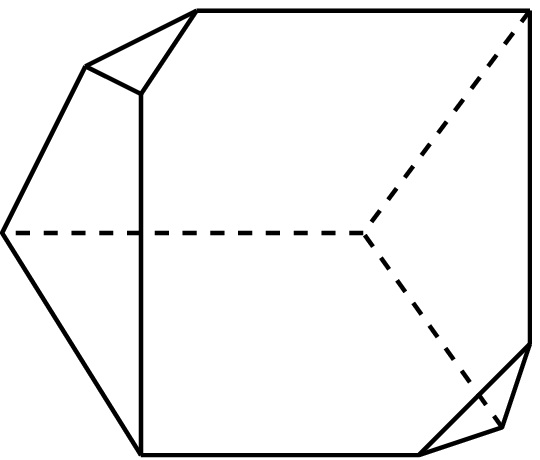}
&
$\rightsquigarrow $
&
\includegraphics[scale=.7]{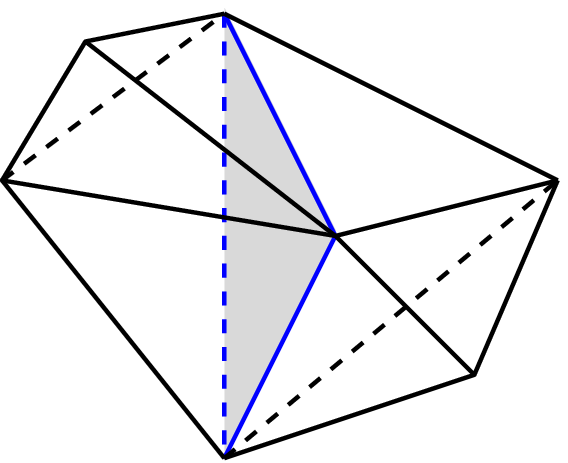}
\small
\put (-277, 10){$\GG$}
\put (-100, 10){$\GG^*$}
\put (-65, 37){\textcolor{blue}{$\Delta_\delta$}}
\\
\end{tabular}
\caption{A prismatic circuit of a polytope $\GG$}
\label{fig:prismatic}
\end{figure}

\begin{rem}
If $\GG$ is a truncation polytope, then the prismatic circuits of $\GG$ are in correspondence with the interior faces of codimension 1 in  the triangulation of $\GG^*$.
\end{rem}

Now, given a prismatic circuit $\delta$ of $\GG$, we define the \emph{splitting of $\GG$ along $\delta$} as follows:
\begin{enumerate}
\item take the dual polytope $\GG^*$ of $\GG$;
\item using the convex hull $\Delta_\delta$ in $\GG^*$, we decompose $\GG^*$ into two polytopes $\GG_1^*$ and $\GG_2^*$ with the new facets $\Delta_{\delta, 1}$ and $\Delta_{\delta, 2}$ corresponding to $\Delta_\delta$ respectively (see Figure \ref{fig:splitting});
\item take the dual polytope $\GG_i$ of $\GG_i^*$ and truncate $\GG_i$ at the vertex $v_i$ dual to $\Delta_{\delta, i}$, for each $i \in \{ 1,2 \}$.
\end{enumerate}

\begin{figure}[ht!]
\centering
\begin{tabular}{>{\centering\arraybackslash}m{.3\textwidth} c >{\centering\arraybackslash}m{.3\textwidth}}
\includegraphics[scale=.7]{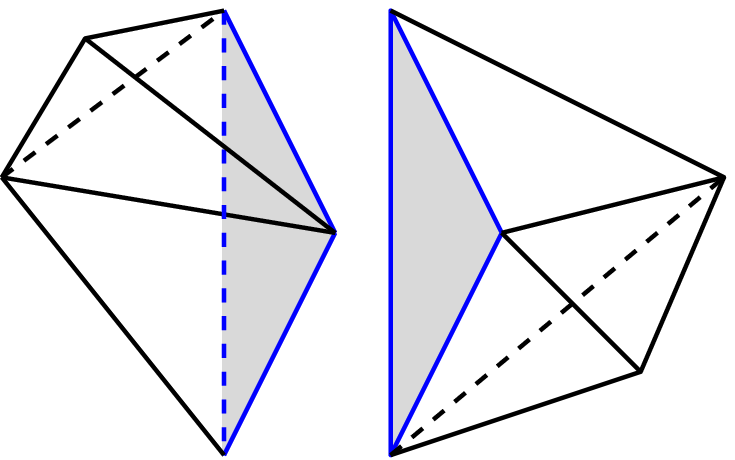}
\small
\put (-130, 5){$\GG_1^*$}
\put (-25, 5){$\GG_2^*$}
\put (52, 5){$\GG_1^{\dagger v_1}$}
\put (164, 5){$\GG_2^{\dagger v_2}$}
\put (124,43){$v_1$}
\put (138,43){$v_2$}
\put (-102, 40){\textcolor{blue}{$\Delta_{\delta,1}$}}
\put (-68, 41){\textcolor{blue}{$\Delta_{\delta,2}$}}
&
\quad $\rightsquigarrow$ \quad \quad
&
\includegraphics[scale=.7]{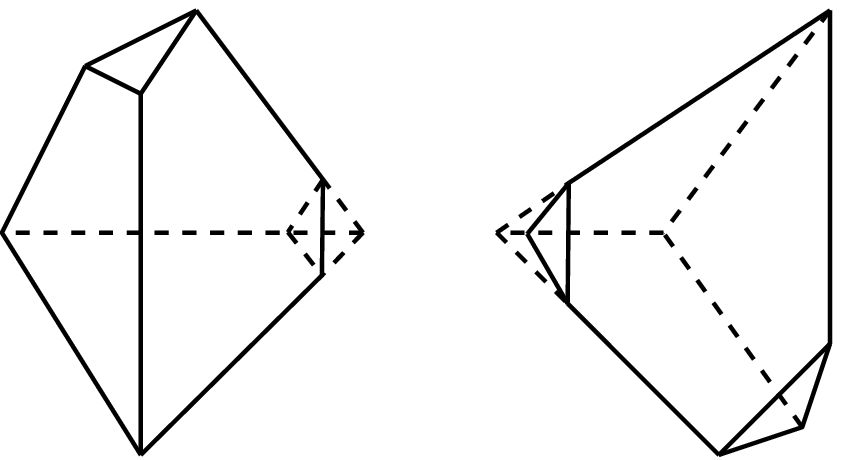}
\\
\end{tabular}
\caption{The splitting of $\GG$ into $\GG_1^{\dagger v_1}$ and $\GG_2^{\dagger v_2}$}
\label{fig:splitting}
\end{figure}

The original polytope $\GG$ is the gluing of $\GG_1^{\dagger v_1}$ and $\GG_2^{\dagger v_2}$ via the obvious isomorphism between the links $(\GG_{1})_{v_1} $ and $(\GG_{2})_{v_2} $. For labeled polytopes, we take care of the ridge labels in an apparent way to define the \emph{gluing} and \emph{splitting} operations.

\subsection{Property of prismatic circuit}

In this subsection, we use the same notation as in Section \ref{subsec:gluing:definition}. Let $\GG$ be a labeled $d$-polytope and $\delta$ a prismatic circuit of $\GG$. Suppose $\GG$ splits along $\delta$ into $\GG_1^{\dagger v_1}$ and $\GG_2^{\dagger v_2}$. 
A prismatic circuit $\delta$ of $\GG$ is  
\begin{itemize}
\item \emph{useless} if both $\GG_1$ and $\GG_2$ are simplices, and for each $i \in \{1, 2\}$, the facet $s_i$ of $\GG_i$ opposite to $v_i$ is \emph{orthogonal to $\delta$}, \ie the dihedral angle between $s_i$ and $s$ is $\nicefrac{\pi}{2}$ for each facet $s$ of $\GG_i$ that contains $v_i$; 
\item {\em non-essential} if there exists a unique $i \in  \{1, 2 \}$  such that $\GG_{i}$ is a simplex and the facet $s_{i}$ of $\GG_{i}$ opposite to $v_i$ is orthogonal to $\delta$; 
\item {\em essential}, otherwise.
\end{itemize}

\medskip

The following is immediate from the definition:

\begin{lemma}\label{lemma:useless}
Let $\GG$ be a labeled polytope and $\delta$ a prismatic circuit of $\GG$. Suppose that $\GG$ splits along $\delta$ into $\GG_1^{\dagger v_1}$ and $\GG_2^{\dagger v_2}$. Then:
\begin{itemize}
\item if $\delta$ is useless, then $\GG_1^{\dagger v_1} = \GG_2^{\dagger v_2} = \GG$ and $W_\GG = W_{\delta} \times \tilde{A}_1$. In particular, $\GG$ is not irreducible.
\item if $\delta$ is non-essential, then there exist $ i \neq j \in \{ 1, 2 \}$ such that $\GG_i^{\dagger v_i} = \GG$ and  $W_{\GG_j^{\dagger v_j}} = W_{\delta} \times \tilde{A}_1$.
\end{itemize}
\end{lemma}

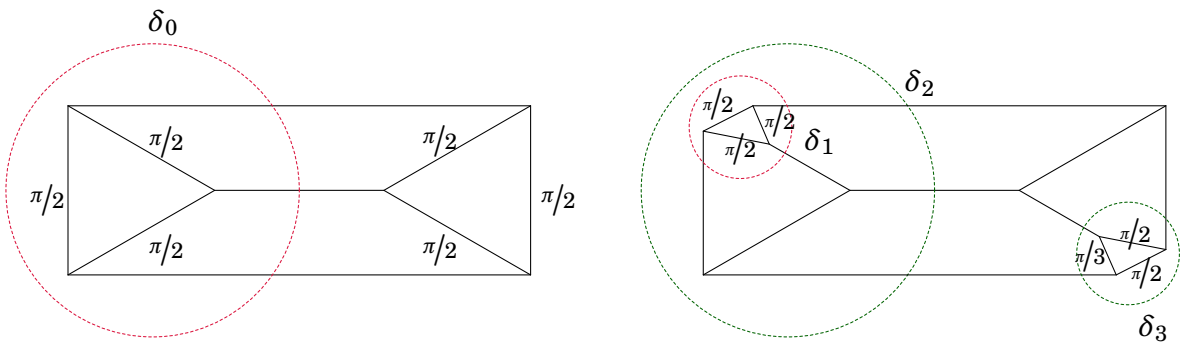
\begin{figure}[ht!]
\begin{tabular}{c c c}
\definecolor{dcrutc}{rgb}{0.86,0.08,0.24}
\begin{tikzpicture}[scale=0.75,line cap=round,line join=round,>=triangle 45,x=1.0cm,y=1.0cm]
\clip(-1.34,-2) rectangle (9.16,5);
\draw (0,3)-- (0,0);
\draw (0,0)-- (8.2,0);
\draw (8.2,0)-- (8.2,3);
\draw (0,3)-- (2.6,1.5);
\draw (2.6,1.5)-- (5.6,1.5);
\draw (5.6,1.5)-- (8.2,3);
\draw (8.2,3)-- (0,3);
\draw (2.6,1.5)-- (0,0);
\draw (5.6,1.5)-- (8.2,0);

\draw (1.25,2.82) node[anchor=north west] {$\nicefrac{\pi}{2}$};
\draw (1.25,0.9) node[anchor=north west] {$\nicefrac{\pi}{2}$};
\draw (-0.86,1.8) node[anchor=north west] {$\nicefrac{\pi}{2}$};

\draw (6.1,2.85) node[anchor=north west] {$\nicefrac{\pi}{2}$};
\draw (8.2,1.8) node[anchor=north west] {$\nicefrac{\pi}{2}$};
\draw (6.1,0.9) node[anchor=north west] {$\nicefrac{\pi}{2}$};
\draw [rotate around={0:(1.5,1.5)},dash pattern=on 1pt off 1pt,color=dcrutc] (1.5,1.5) ellipse (2.6cm and 2.6cm);
\draw (1.24,4.9) node[anchor=north west] {$\delta_0$};
\end{tikzpicture}
&
\definecolor{dcrutc}{rgb}{0.86,0.08,0.24}
\definecolor{qqwuqq}{rgb}{0,0.39,0}
\begin{tikzpicture}[scale=0.75,line cap=round,line join=round,>=triangle 45,x=1.0cm,y=1.0cm]
\clip(-1.34,-2) rectangle (8.66,5);
\draw (2.6,1.5)-- (5.6,1.5);
\draw (5.6,1.5)-- (8.2,3);
\draw (2.6,1.5)-- (0,0);
\draw [rotate around={0:(1.5,1.5)},dash pattern=on 1pt off 1pt,color=qqwuqq] (1.5,1.5) ellipse (2.6cm and 2.6cm);
\draw (1.61,2.79) node[anchor=north west] {$\delta_1$};
\draw (0,2.55)-- (0.88,3);
\draw (1.17,2.32)-- (0,2.55);
\draw (0,2.55)-- (0,0);
\draw (0.88,3)-- (8.2,3);
\draw (1.17,2.32)-- (0.88,3);
\draw (1.17,2.32)-- (2.6,1.5);
\draw (7.02,0.68)-- (7.32,0);
\draw (7.02,0.68)-- (8.2,0.45);
\draw (8.2,0.45)-- (7.32,0);
\draw (0,0)-- (7.32,0);
\draw (7.02,0.68)-- (5.6,1.5);
\draw (8.2,0.45)-- (8.2,3);
\draw (-0.76,-4.54) node[anchor=north west] {Reglage des triangles avec les points K et M};
\draw [rotate around={-30:(0.66,2.62)},dash pattern=on 1pt off 1pt,color=dcrutc] (0.66,2.62) ellipse (0.91cm and 0.91cm);
\draw (-0.63,-4.98) node[anchor=north west] {Reglage ellipse autour des triangles avec U,T,V};
\draw (-0.1,-6.17) node[anchor=north west] {Centre de la figure G};
\draw [rotate around={150:(7.53,0.38)},dash pattern=on 1pt off 1pt,color=qqwuqq] (7.53,0.38) ellipse (0.91cm and 0.91cm);
\draw (3.39,3.82) node[anchor=north west] {$\delta_2$};
\draw (7.53,-0.52) node[anchor=north west] {$\delta_3$};
\draw (0.2,2.7) node[anchor=north west] {$\tiny \nicefrac{\pi}{2}$};
\draw (0.9,3.2) node[anchor=north west] {$\tiny \nicefrac{\pi}{2}$};
\draw (-0.2,3.5) node[anchor=north west] {$\tiny \nicefrac{\pi}{2}$};
\draw (7.2,1.2) node[anchor=north west] {$\tiny \nicefrac{\pi}{2}$};
\draw (7.4,0.5) node[anchor=north west] {$\tiny \nicefrac{\pi}{2}$};
\draw (6.4,0.8) node[anchor=north west] {$\tiny \nicefrac{\pi}{3}$};
\end{tikzpicture}
\end{tabular}
\caption{Examples of useless, non-essential or essential prismatic circuits: 
$\delta_0$ is useless, $\delta_1$ is non-essential, and $\delta_2$ and $\delta_3$ are essential. The ridges with no label may have any label, and it does not affect the properties of the prismatic circuits.}
\label{fig:examples_circuits}
\end{figure}

\begin{lemma}\label{lemma:property:prismatic}
Let $\GG$ be an irreducible, large, $2$-perfect labeled polytope of dimension $d \geqslant 4$ and $\delta$ a prismatic circuit of $\GG$. Assume that $\GG$ splits along $\delta$ into $\GG_1^{\dagger v_1}$ and $\GG_2^{\dagger v_2}$. If $\GG$ is convex-projectivizable, i.e. $\B(\GG) \neq \varnothing$, then the following hold:
\begin{enumerate}
\item the polytopes $\GG_1$ and $\GG_2$ are 2-perfect;
\item the Coxeter group $W_\delta$ is Lann\'{e}r or $\tilde{A}_{d-1}$;
\item if $\delta$ is essential, then $\GG_1$ and $\GG_2$ are also irreducible and large.
\end{enumerate}
\end{lemma}

\begin{proof}
Suppose $(P,\, (\sigma_s = \mathrm{Id} - \alpha_s \otimes b_s)_{s \in S} )$ is a Coxeter $d$-polytope realizing $\GG$. Since the intersection of the facets of $P$ in $\delta$ is empty, the group $W_\delta$ is infinite by Vinberg \cite[Th.\,4]{MR0302779}. We denote by $\Pi_\delta$ the subspace spanned by $(b_s)_{s \in \delta}$, by $\sigma_s^\delta$ the induced reflection of $\sigma_s$ on $\Pi_{\delta}$, and by $\Gamma_\delta$ the subgroup of $\mathrm{SL}^{\pm}(\Pi_{\delta})$ generated by $(\sigma_s^\delta)_{s \in \delta}$. Then by \cite[Lem.\,8.19]{Marquis:2014aa}, 
 $$ P_{\delta} := \bigcap_{s \in \delta} \S( \{ x \in \Pi_\delta \mid \alpha_s (x) \leqslant 0 \})$$
is a perfect Coxeter $(d-1)$-simplex such that the $\Gamma_\delta$-orbit of $P_{\delta}$ is a properly convex domain $\Omega_P \cap \mathbb{S}(\Pi_\delta)$ of $\mathbb{S}(\Pi_\delta)$. So, the vertex link of $\GG_i$ at $v_i$ is perfect, and by Theorem \ref{tri} and Proposition \ref{prop:label_simplex}, $W_\delta$ is Lann\'{e}r or $\tilde{A}_{d-1}$.

\medskip

Now, assume that $\delta$ is essential. There is no facet of $\GG_i$ orthogonal to $\delta$ and $W_\delta$ is irreducible. Thus $W_{\GG_i}$ is also irreducible. Finally, since $W_\delta$ is Lann\'{e}r or $\tilde{A}_{d-1}$, and since $W_\delta$ is a proper standard subgroup of $W_{\GG_i}$, the group $W_{\GG_i}$ must be large.
\end{proof}

\subsection{The splitting of Coxeter polytope}

Let $\GG$ be an irreducible, large, 2-perfect labeled polytope of dimension $d \geqslant 4$, $S$ the set of facets of $\GG$, and $\delta$ an essential prismatic circuit of $\GG$. Assume that $\GG$ splits along $\delta$ into $\GG_1^{\dagger v_1}$ and $\GG_2^{\dagger v_2}$. Then $\GG = \GG _1^{\dagger v_1} \sharp_{\phi} \GG_2^{\dagger v_2}$ for the induced isomorphism $\phi$ between the links $(\GG_1)_{v_1}$ and $(\GG_2)_{v_2}$. For each $i = 1, 2$, we denote by $S_i$ the subset of $S$ consisting of the facets of $\GG$ that correspond to the facets of $\GG_i$. Then $S=S_1 \cup S_2$ and $\delta=S_1 \cap S_2$.

\medskip

Now, we define a {\em splitting} map 
$$\textrm{Cut}_{\delta} : \B(\GG) \to \B(\GG_1^{\dagger v_1}) \times \B(\GG_2^{\dagger v_2})$$ 
by sending $[P] \in \B(\GG)$ with $P= \bigcap_{s \in S} \S( \{ x \in V \mid \alpha_s (x) \leqslant 0 \})$ to $([P_1^{\dagger v_1}],[P_2^{\dagger v_2}]) \in \B(\GG_1^{\dagger v_1}) \times \B(\GG_2^{\dagger v_2})$, where 
$ P_i = \bigcap_{s \in S_{i}} \S( \{ x \in V \mid \alpha_s (x) \leqslant 0 \}) $  for each $i \in 1, 2$. By definition, $P = P_1 \cap P_2$. The splitting map is well-defined again by \cite[Lem.\,8.19]{Marquis:2014aa}: for each $[P] \in \B(\GG)$, the subspace $\Pi_{\delta}$ of $V$ spanned by $(b_s)_{s \in \delta}$ is a hyperplane, and the intersection of $\S(\Pi_{\delta})$ and the relative interior of $e$ is a singleton, for each edge $e$ of $P$ in the prismatic poset associated to $\delta$. Hence the vertex $v_i$ of $P_i$ is truncatable and $P_i^{\dagger v_i}$ lies in $\B(\GG_i^{\dagger v_i})$. Moreover, the links $(P_{1})_{v_1}$ and $(P_{2})_{v_2}$ are isomorphic.

\medskip

Let $\delta_i$ be the set of facets of $\GG_i$ that contain $v_i$. Since $W_\delta$ is Lann\'{e}r or $\tilde{A}_{d-1}$ by Lemma \ref{lemma:property:prismatic}, $W_\delta = W_{\delta_i}$ is of cycle type or of tree type. If $W_\delta$ is of cycle type, it has only two relevant circuits $\C_\delta$ and $\overline{\C}_\delta$. Otherwise, it has no relevant circuit. We denote by $R_{\delta_i}$ the normalized cyclic product of $\C_{\delta_i}$ if $W_{\delta}$ is of cycle type, and $0$ otherwise. Then we may choose an orientation of $\C_{\delta_i}$ so that $R_{\delta_1}([P_1^{\dagger v_1}]) = R_{\delta_2}([P_2^{\dagger v_2}])$. So, we introduce the following subspace of $\B(\GG_1^{\dagger v_1}) \times \B(\GG_2^{\dagger v_2})$:
$$
\B(\GG_1^{\dagger v_1}) \boxtimes_{\phi} \B(\GG_2^{\dagger v_2}) := \{([P_1^{\dagger v_1}],[P_2^{\dagger v_2}]) \in \B(\GG_1^{\dagger v_1}) \times \B(\GG_2^{\dagger v_2}) \mid  R_{\delta_1}([P_1^{\dagger v_1}]) =  R_{\delta_2}([P_2^{\dagger v_2}])\}
$$
Since the image of the map $\textrm{Cut}_{\delta}$ lies in $\B(\GG_1^{\dagger v_1}) \boxtimes_{\phi} \B(\GG_2^{\dagger v_2})$, we shall restrict the range of $\textrm{Cut}_{\delta}$ accordingly. 

\begin{lemma}\label{lem:gluing}
Let $\GG$ be an irreducible, large, 2-perfect labeled polytope of dimension $d \geqslant 4$, and let $\delta$ be an essential prismatic circuit of $\GG$. Suppose that $\GG$ splits along $\delta$ into $\GG_1^{\dagger v_1}$ and $\GG_2^{\dagger v_2}$. Then there exists an $\mathbb{R}$-action $\Psi$ on $\B(\GG)$ such that $\textrm{Cut}_{\delta}$ is a $\Psi$-invariant fibration onto $\B(\GG_1^{\dagger v_1}) \boxtimes_{\phi} \B(\GG_2^{\dagger v_2})$ and $\Psi$ is simply transitive on each fiber of $\textrm{Cut}_{\delta}$. 
\end{lemma}

The above lemma was proved by the third author \cite[Lem.\,4.36]{MR2660566} in dimension $d = 3$. One can extend the proof to any dimension $d \geqslant 4$ without difficulty, but we give an outline of a proof for the reader’s convenience.

\begin{proof}
We first define the $\mathbb{R}$-action $\Psi$ on $\B(\GG)$.  Given any $[P] \in \B(\GG)$, we have Coxeter polytopes $P_1^{\dagger v_1}$ and $P_2^{\dagger v_2}$ such that 
$P = P_1 \cap P_2$ and $\textrm{Cut}_{\delta}([P]) = ([P_1^{\dagger v_1}],[P_2^{\dagger v_2}])$. Let $(e_i)_{i=1}^{d+1}$ be the canonical basis of $\R^{d+1}$ and $(e^\ast_i)_{i=1}^{d+1}$ its dual basis. We may assume that the supporting hyperplanes of the facets in $\delta$ are $ \{ \S(\ker(e_{i}^*)) \}_{i=1}^{d}$ and that the subspace $\Pi_{\delta}^P$ spanned by $(b_s)_{s \in \delta}$ equals $\ker(e_{d+1}^*)$, where $(\sigma_s = \mathrm{Id} - \alpha_s \otimes b_s)_{s \in S}$ is the set of reflections of $P$. Now, if $g_u \in \mathrm{SL}^{\pm}_{d+1}(\mathbb{R})$ is the diagonal matrix with entries $e^u, \dotsc, e^u, e^{-du}$, then 
$$\Psi_u([P]) := [P_1 \cap g_u (P_2)] \in \B(\GG)$$ 
lies in the same fiber of $\textrm{Cut}_{\delta}$ as $[P] = \Psi_0([P])$. It gives us the required $\mathbb{R}$-action $\Psi$ that preserves each fiber of $\mathrm{Cut}_{\delta}$.

\medskip

To show that $\Psi$ is free on each fiber of $\mathrm{Cut}_{\delta}$, we may choose a facet $s$ in $\delta$, $s'$ of $\GG_1$ not in $\delta$, and $s''$ of $\GG_2$ not in $\delta$ such that the dihedral angles between $s$ and $s'$ and between $s$ and $s''$ are different from $\nicefrac{\pi}{2}$. If we denote by $A^u $ the Cartan matrix of $\Psi_u([P])$, then the map $R :  \mathbb{R} \rightarrow \mathbb{R}$ given by
 $$ R(u) = \log\left( \frac{A^u_{s s'}A^u_{s' s''}A^u_{ s'' s }}{A^u_{ s s''}A^u_{ s'' s' }A^u_{ s' s }} \right)$$
 is a homeomorphism. So, the action on each fiber is free.

\medskip

Finally we show that $\Psi$ is transitive on each fiber. Let $[P], [Q] \in \B(\GG)$ be on the same fiber, \ie $P = P_1 \cap P_2$, $Q = Q_1 \cap Q_2$ and $([P_1^{\dagger v_1}],[P_2^{\dagger v_2}]) = ([Q_1^{\dagger v_1}],[Q_2^{\dagger v_1}])$. Then we may assume that $Q_1 = P_1$ and $Q_2 = g(P_2)$ for some $g \in \mathrm{SL}^{\pm}_{d+1}(\mathbb{R})$, that the supporting hyperplanes of the facets of $P$ (hence also of $Q$) in $\delta$ are $ \{ \S(\ker(e_{i}^*)) \}_{i=1}^{d}$, and that the subspace $\Pi_\delta^P$ (hence also $\Pi_\delta^Q$)  equals $\ker(e_{d+1}^*)$. The restriction of $g$ on $\S(\ker(e_{d+1}^*))$ is the identity and $g([e_{d+1}]) = [e_{d+1}]$ because $[e_{d+1}] = \cap_{i=1}^{d} \S(\ker (e_{i}^*))$. In other words, $g = g_u$ for some $u$, hence $\Phi_u([P])=[Q]$.
\end{proof}

\begin{rem}
A similar construction as in this section may be found in one of Fenchel-Nielsen coordinates that parametrize hyperbolic structures on surface. An essential simple closed curve on the surface (resp. the length of the unique geodesic isotopic to that curve) plays a role of the essential prismatic circuit of polytope (resp. the cyclic product). The cutting and gluing operations along the geodesic are analogous to cutting and gluing along the essential prismatic circuit. And, there is a gluing parameter called the Dehn twist parameter. Instead, in the case of hyperbolic polygon, a pair of nonadjacent edges (resp. the distance between those edges) plays a role of the essential prismatic circuit (resp. the cyclic product), but there is no gluing parameter. In our case, there is a gluing parameter, called \emph{bending} (or \emph{bulging}) which comes from projective geometry. One can find a description of bending deformation for convex projective manifold in \cite{johnson_millson} or \cite{bulging}, and a lemma \cite[Lem. 5.3]{surf_gold} for convex projective surface, analogous to Lemma \ref{lem:gluing}. 
\end{rem}

Let $\GG$ be an irreducible, large, $2$-perfect labeled polytope of dimension $d \geqslant 4$ and let $\delta$ be a prismatic circuit of $\GG$. As in the proof of Lemma \ref{lemma:property:prismatic}, one can show that if $\GG$ is hyperbolizable, then $W_{\delta}$ is Lann\'{e}r. Assume that $\delta$ is essential and that $\GG$ splits along $\delta$ into $\GG_1^{\dagger v_1}$ and $\GG_2^{\dagger v_2}$. One can also define the splitting map
$$
\mathrm{Cut}_{\delta}^{\hyp} \,:\, \mathrm{Hyp}(\GG) \to \mathrm{Hyp}(\GG_1^{\dagger v_1}) \boxtimes_{\phi} \mathrm{Hyp}(\GG_2^{\dagger v_2})
$$
similar to the splitting map $\mathrm{Cut}_{\delta}$ of $\B(\GG)$. Clearly, $\mathrm{Cut}_{\delta}^{\hyp}$ is bijective. 

\begin{proof}[Proof of Theorem \ref{thm:geometrization}]

Let $\GG$ be an irreducible, large, 2-perfect, labeled truncation polytope of dimension $d \geqslant 4$, $\mathcal{P}$ the set of prismatic circuits of $\GG$, and $\mathcal{P}_e$ the set of essential prismatic circuits of $\GG$.

\medskip

By Lemma \ref{lemma:property:prismatic}, if $\GG$ is convex-projectivizable, then $W_\delta$ is Lannér or $\tilde{A}_{d-1}$ for each $\delta \in \mathcal{P}$. Conversely, suppose that $W_\delta$ is Lannér or $\tilde{A}_{d-1}$ for each $\delta \in \mathcal{P}$. The polytope $\GG$ splits along $\mathcal{P}_e$ into once-truncated $d$-simplices $\{ \Ss_i^{\dagger \mathcal{V}_i} \}_{i=1}^{k_e + 1}$, where each $\Ss_i$ is an irreducible, large, 2-perfect labeled simplex, $\mathcal{V}_i$ is a set of vertices in $\Ss_i$ that correspond to $\mathcal{P}$, and $k_e = \# \mathcal{P}_e$. By Proposition \ref{prop:deformation_block}, each once-truncated simplex $\Ss_i^{\dagger \mathcal{V}_i}$ is convex-projectivizable and by Lemma \ref{lem:gluing}, $\GG$ is also convex-projectivizable.

\medskip

In the similar fashion, one can show that $\GG$ is hyperbolizable if and only if $W_\delta$ is Lannér for each $\delta \in \mathcal{P}$. 

\medskip

We now assume that $\GG$ is perfect. It is well-known that if $\GG$ is hyperbolizable, then $W_\GG$ is word-hyperbolic. Conversely, suppose that $\GG$ is convex-projectivizable and $W_\GG$ is word-hyperbolic. The previous statements show that $W_\delta$ is Lannér or $\tilde{A}_{d-1}$ for each $\delta \in \mathcal{P}$. But $W_\delta$ cannot be $\tilde{A}_{d-1}$. Otherwise, the word-hyperbolic group $W_\GG$ would contain a virtually free abelian group $\tilde{A}_{d-1}$ of rank $d-1 \geqslant 2$. Thus $\GG$ is hyperbolizable, again by the previous statements.
\end{proof}

\begin{de}
The set of once-truncated labeled simplices $\{ \Ss_i^{\dagger \mathcal{V}_i} \}_{i=1}^{k_e + 1}$ in the proof of Theorem \ref{thm:geometrization} is called the \emph{decomposition of $\GG$ along the essential prismatic circuits}, and each $\Ss_i$ is a \emph{block of $\GG$}. Note that $W_{\Ss_i}$ is a 2-Lann\'{e}r Coxeter group.
\end{de}

\begin{rem}\label{rem:geometrization}
The proof of Theorem \ref{thm:geometrization} explains how to obtain the complete list of convex-projectivizable or hyperbolizable, irreducible, large, 2-perfect labeled truncation polytopes of dimension $d \geqslant 4$ from the list of 2-Lann\'{e}r Coxeter groups of rank $\geqslant 5$ in Appendix \ref{classi_lorent}. 
\end{rem}

\subsection{Evaluation map}\label{subsec:evaluation}

Let $\GG$ be a convex-projectivizable, irreducible, large, 2-perfect labeled truncation polytope of dimension $\geqslant 4$. We denote by $ \{ \Ss_i^{\dagger \mathcal{V}_i} \}_{i=1}^{k_e + 1} $ the decomposition of $\GG$ along the essential prismatic circuits. A prismatic circuit $\delta$ (resp. a vertex $v$, resp. a block $\Ss_i$) is “\emph{something}” if the Coxeter group $W_\delta$ (resp, $W_{S_v}$, resp, $W_{\Ss_i}$) is “something”. For example, the block $\Ss_i$ can have four different types: tree, cycle, pan, or $K_{2,3}$. We now introduce the following notation:

\begin{itemize}
\item  $\mathcal{P}_{fL}$ is the set of flexible Lann\'{e}r prismatic circuits of $\GG$, and $k_L(\GG) = \# \mathcal{P}_{fL}$;

\item $\mathcal{P}_{A}$ is the set of $\tilde{A}_{d-1}$ prismatic circuits of $\GG$, and $k_A(\GG) = \# \mathcal{P}_{A}$; 

\item $\mathcal{V}_{f}$ is the set of flexible vertices of $\GG$, and $k_{v}(\GG) = \# \mathcal{V}_{f}$;

\item $\mathcal{B}_{c}$ is the set of blocks $\Ss_i$ of cycle type, and $k_c(\GG) = \# \mathcal{B}_{c}$;

\item $\mathcal{B}_{K}$ is the set of blocks $\Ss_i$ of $K_{2,3}$ type, and $k_K(\GG) = \# \mathcal{B}_{K}$.
\end{itemize}

A map 
$$ \Theta : \B(\GG) \rightarrow Y(\GG) :=  \R^{ k_L(\GG) } \times (\R^*)^{ k_A(\GG) } \times \R^{k_v(\GG)} \times \R^{ k_c(\GG) } $$
is given by the evaluation of $[P] \in \B(\GG)$ on the circuits that correspond to the elements in $ \mathcal{P}_{fL}$,  $\mathcal{P}_{A}$, $\mathcal{V}_{f}$ and $\mathcal{B}_{c} $. More precisely, if $\delta \in \mathcal{P}_{fL} \cup \mathcal{P}_{A}$ (resp. $v \in \mathcal{V}_f$, resp. $\Ss_i \in \mathcal{B}_c$), then the Coxeter group $W_\delta$ (resp. $W_{S_v}$, resp. $W_{\Ss_i}$) is of cycle type. So, it has a unique relevant circuit up to orientation, denoted by $\mathcal{C}_{\delta}$ (resp. $\mathcal{C}_{v}$, resp. $\mathcal{C}_{i}$). Then 
$$\Theta([P]) := \left( (R_{{\mathcal{C}}_{\delta}}(A_{P}))_{\delta \in \mathcal{P}_{fL} }, (R_{{\mathcal{C}}_{\delta}}(A_{P}))_{\delta \in \mathcal{P}_{A} }, (R_{{\mathcal{C}}_{v}}(A_{P}))_{v \in \mathcal{V}_{f} }, (R_{{\mathcal{C}}_{i}}(A_{P}))_{\Ss_i \in \mathcal{B}_{c}} \right),$$
where $R_{\mathcal{C}}$ denotes the normalized cyclic product of $\mathcal{C}$.

\medskip

Each $\Ss_j \in \mathcal{B}_K$ has three relevant circuits 
$\{ \mathcal{C}_{j}^\ell \}_{\ell=1,2,3}$ up to orientation. And, $\mathcal{C}_{j}^{\ell}$ is either $\mathcal{C}_\delta$ $(\delta \in \mathcal{P}_{fL} \cup \mathcal{P}_A) $, or $\mathcal{C}_v$ $(v \in \mathcal{V}_f)$ again up to orientation. We may choose the orientations of $\mathcal{C}_\delta$ $(\delta \in \mathcal{P}_{fL} \cup \mathcal{P}_{A})$ and $\mathcal{C}_v$ $(v \in \mathcal{V}_f)$ coherently so that each $\mathcal{C}_{j}^{\ell}$ equals $\mathcal{C}_\delta$ or $\mathcal{C}_v$ \emph{with orientation} and that $ \sum_{\ell=1}^3 R_{\mathcal{C}_{j}^{\ell}}(A_P) = 0$. Then we consider the following subspace of $Y(\GG)$:

$$
X(\GG) = \{\, y \in Y(\GG) \,\mid\,  \sum_{\ell=1}^3 y_{j}^\ell = 0 \,\, \textrm{ for every block } \Ss_j \in \mathcal{B}_K \,\},
$$
where $y_{j}^\ell$ is the $\delta$-coordinate of $y$ if $ \mathcal{C}_{j}^\ell = \mathcal{C}_\delta$, or the $v$-coordinate of $y$ if $ \mathcal{C}_{j}^\ell = \mathcal{C}_v$. Since the image of the map $\Theta$ lies in $X(\GG)$, we shall restrict the range of $\Theta$ accordingly. Now, (the proof of) Proposition \ref{prop:deformation_block} and Lemma \ref{lem:gluing} show that:

\begin{propo}\label{prop:evaluation}
Let $\GG$ be a convex-projectivizable, irreducible, large, 2-perfect labeled truncation polytope of dimension $d \geqslant 4$. Then there exists an $\mathbb{R}^{k_e}$-action $\Psi$ on $\B(\GG)$ such that $\Theta$ is a $\Psi$-invariant fibration onto $X(\GG)$ and $\Psi$ is simply transitive on each fiber of $\Theta$.
\end{propo}

\begin{proof}[Proof of Theorem \ref{MainTheorem}]
Let $\GG$ be a convex-projectivizable, irreducible, large, 2-perfect labeled truncation polytope of dimension $d \geqslant 4$. We denote by $e_+(\GG)$ the number of ridges with label $\neq \nicefrac{\pi}{2}$ in $\GG$, by $\mathcal{P}_e$ the set of essential prismatic circuits,  and by $ \{ \Ss_i^{\dagger \mathcal{V}_i} \}_{i=1}^{k_e + 1} $ the decomposition of $\GG$ along $\mathcal{P}_e$.

\medskip

We first claim that $d \leqslant 9$. Indeed, each $\Ss_i$ is an irreducible, large, 2-perfect $d$-simplex, and such a simplex exists only in dimension $d \leqslant 9$, by Theorem \ref{t:maxwell}, as claimed.

\medskip

Proposition \ref{prop:evaluation} shows that $\B(\GG)$ is a union of finitely many open cells and that $\B(\GG)$ is connected if and only if $k_A(\GG) = 0$, i.e. $W_\delta$ is Lann\'{e}r for each prismatic circuit $\delta$ of $\GG$. This is equivalent to require that $\GG$ is hyperbolizable, by Theorem \ref{thm:geometrization}.

\medskip

We finally compute the dimension of $\B(\GG)$. Again, by Proposition \ref{prop:evaluation}, we have that
$$\dim \B(\GG) = k_L(\GG) + k_A(\GG) + k_v(\GG) + k_c(\GG) - k_K(\GG) + k_e(\GG)$$
We now prove that $\dim \B(\GG) = e_+(\GG) - d$ by induction on the number $k_e$ of essential prismatic circuits. If $k_e(\GG) = 0$, then $\GG = \Ss_i^{\dagger \mathcal{V}_i}$ with $i=1$. Proposition \ref{prop:deformation_block} shows $\dim \B(\Ss_i^{\dagger \mathcal{V}_i}) = e_+(\Ss_i^{\dagger \mathcal{V}_i}) - d$, or it may readily verified as follows: 
\begin{itemize}
\item if $\Ss_i$ is of tree type, then $k_*(\Ss_i^{\dagger \mathcal{V}_i}) = 0$ for $* \in \{ L, A, v,c,  K, e \}$;
\item if $\Ss_i$ is of cycle type, then $k_c(\Ss_i^{\dagger \mathcal{V}_i})=1$ and $k_*(\Ss_i^{\dagger \mathcal{V}_i}) = 0$ for $* \in \{ L, A, v, K, e \}$;
\item if $\Ss_i$ is of pan type, then $(k_L+k_A+k_v)(\Ss_i^{\dagger \mathcal{V}_i})=1$ and $k_*(\Ss_i^{\dagger \mathcal{V}_i}) = 0$ for $* \in \{ c, K, e \}$;
\item if $\Ss_i$ is of $K_{2,3}$ type, then $(k_L+k_A+k_v)(\Ss_i^{\dagger \mathcal{V}_i})=3$, $k_K(\Ss_i^{\dagger \mathcal{V}_i}) = 1$ and $k_*(\Ss_i^{\dagger \mathcal{V}_i}) = 0$ for $* \in \{ c, e\}$.
\end{itemize}

\medskip

If $k_e(\GG) > 0$, then the polytope $\GG$ splits along $\delta \in \mathcal{P}_e$ into two polytopes $\GG_j$ $(j = 1,2)$ with $k_e(\GG_j) < k_e(\GG)$. There are two cases to consider: (i) $\delta$ is rigid, and (ii) $\delta$ is flexible.

\medskip

In the case (i), we have $e_+(\GG_1 \sharp \GG_2) = e_+(\GG_1) + e_+(\GG_2) - (d-1)$. Then 
\begin{align}
e_+(\GG) - d & =  (e_+(\GG_1) - d) + (e_+(\GG_2) - d) + 1 \nonumber\\
& = \dim \B(\GG_1) + \dim \B(\GG_2) + 1  \quad\quad \textrm{by induction hypothesis} \nonumber\\
& = k_L(\GG) + k_A(\GG) + k_v(\GG) + k_c(\GG) - k_K(\GG) + k_e(\GG) \nonumber \\
& = \dim \B(\GG) \nonumber
\end{align}
The second last equality follows from the fact that $k_{e}(\GG) = k_e(\GG_1) + k_e(\GG_2) +1 $ and $k_{*}(\GG) = k_{*}(\GG_1) + k_{*}(\GG_2)$ for $* \in \{ L, A, v, c, K \}$. 

\medskip

In the case (ii), we have $e_+(\GG_1 \sharp \GG_2) = e_+(\GG_1) + e_+(\GG_2) - d$. Then 
\begin{align}
e_+(\GG) - d & =  (e_+(\GG_1) - d) + (e_+(\GG_2) - d) \nonumber\\
& = \dim \B(\GG_1) + \dim \B(\GG_2) \quad\quad \textrm{by induction hypothesis} \nonumber\\
& = \dim \B(\GG) \nonumber
\end{align}
The last equality follows from the fact that $(k_{L}+k_{A})(\GG) = (k_{L}+k_{A})(\GG_1) + (k_{L}+k_{A})(\GG_2) - 1 $, $k_{e}(\GG) = k_{e}(\GG_1) + k_{e}(\GG_2) +1 $, and $k_{*}(\GG) = k_{*}(\GG_1) + k_{*}(\GG_2)$ for $* \in \{ v, c, K \}$.
\end{proof}

\section{Components of deformation space}\label{section:number_of_components}

The purpose of this section is to calculate the number $\kappa(\GG)$ of connected components of the deformation space of a labeled polytope in dimension $d \geqslant 4$. Let $W_{\exc}$ be the left Coxeter group in Figure \ref{fig:exampleA}, and $\Ss_{\mathrm{exc}} := \Ss_{W_{\exc}}$ the associated labeled simplex. The simplex $\Ss_{\exc}$ has two spherical vertices and three $\tilde{A}_3$ vertices. We denote by $\Ss_{\exc}^{\dagger}$ the once-truncated simplex obtained by truncating those three $\tilde{A}_3$ vertices. Proposition \ref{prop:deformation_block} shows that the deformation space of $\Ss_{\exc}^{\dagger}$ has $6$ connected components \emph{not $8=2^3$}, which makes it difficult to compute $\kappa(\GG)$. The technique we shall use is very similar to the one in \cite{MR2660566}.

\subsection{The forest of labeled truncation polytope}

Let $\GG$ be a convex-projectivizable, irreducible, large, 2-perfect labeled truncation polytope of dimension $d \geqslant 4$. We denote by $\mathcal{P}$ the set of prismatic circuits of $\GG$ (both essential and non-essential). The polytope $\GG$ splits along $\mathcal{P}$ into once-truncated simplices $\{ \mathcal{S}_{i}^{\dagger \mathcal{V}_i} \}_{i=1}^{k+1} $, where each $\mathcal{S}_i$ is a 2-perfect labeled simplex (not necessarily irreducible or large), $\mathcal{V}_i$ is the set of vertices in $\mathcal{S}_i$ that correspond to $\mathcal{P}$, and $k = \# \mathcal{P}$. If $v_i \in \mathcal{V}_i$ and $v_j \in \mathcal{V}_j$ both correspond to $\delta \in \mathcal{P}$, then $\delta$ is called the \emph{common prismatic circuit} of $\mathcal{S}_i^{\dagger \mathcal{V}_i}$ and $\mathcal{S}_j^{\dagger \mathcal{V}_j}$, and $\mathcal{S}_i^{\dagger \mathcal{V}_i}$ and $\mathcal{S}_j^{\dagger \mathcal{V}_j}$ \emph{share the prismatic circuit $\delta$}.

\medskip

We now introduce a tool to compute the number of connected components of $\B(\GG)$. The \emph{forest} $\EuScript{F}_{\GG}$  (resp. \emph{orange forest} $\EuScript{F}_{\GG}^{o}$, resp. \emph{green forest} $\EuScript{F}_{\GG}^{g}$) of $\GG$ is a graph with edge coloring such that:
\begin{itemize}
\item the set of nodes consists of all simplices $\mathcal{S}_i$ such that $\mathcal{S}_i^{\dagger \mathcal{V}_i}$ has a flexible (resp. $\tilde{A}_{d-1}$, resp. flexible Lann\'{e}r) prismatic circuit of $\GG$;

\item two nodes $\mathcal{S}_i$ and $\mathcal{S}_j$ are connected by an edge $\overline{\mathcal{S}_i\mathcal{S}_j}$ if and only if $\mathcal{S}_i^{\dagger \mathcal{V}_i}$ and $\mathcal{S}_j^{\dagger \mathcal{V}_j}$ share a flexible (resp. $\tilde{A}_{d-1}$, resp. flexible Lann\'{e}r) prismatic circuit; 

\item the edge $\overline{\mathcal{S}_i \mathcal{S}_j}$ is orange in color if the common prismatic circuit $\delta$ of $\mathcal{S}_i^{\dagger \mathcal{V}_i}$ and $\mathcal{S}_j^{\dagger \mathcal{V}_j}$ is $\tilde{A}_{d-1}$, and it is green if $\delta$ is flexible Lann\'{e}r.
\end{itemize}

Each node of $\EuScript{F}_{\GG}$ has valence $1$, $2$ or $3$. The orange forest $\EuScript{F}_{\GG}^{o}$ and the green forest $\EuScript{F}_{\GG}^{g}$ may be considered as subgraphs of $\EuScript{F}_{\GG}$, and their union is then $\EuScript{F}_{\GG}$. A function from the set of edges of a forest $\EuScript{F}$ to $\{ + , - \}$ is called a \emph{sign function} of $\EuScript{F}$. A sign function is \emph{balanced} if there exists no node $v$ of valence $3$ such that all three edges incident on $v$ have the same sign. A sign function of $\EuScript{F}_{\GG}^{o}$ is \emph{admissible} if it may be extended to a balanced sign function of $\EuScript{F}_{\GG}$.

\begin{lemma}\label{lem:admissible}
Let $\GG$ be a convex-projectivizable, irreducible, large, 2-perfect labeled truncation polytope of dimension $d \geqslant 4$. Then the number $\kappa(\GG)$ of connected components of $\B(\GG)$ equals the number of balanced sign functions of $\EuScript{F}_{\GG}^{o}$. 
\end{lemma}

\begin{proof}
We consider the evaluation map $\Theta : \B(\GG) \rightarrow X(\GG)$ defined in Section \ref{subsec:evaluation}. In particular, the evaluation of $[P] \in \B(\GG)$ on the circuits that correspond to the $\tilde{A}_{d-1}$ prismatic circuits is positive or negative. So, it gives us a balanced sign function $\phi^{o}_{[P]}$ of $\EuScript{F}_{\GG}^{o}$. Then $[P]$ and $[Q]$ lie in the same connected component of $\B(\GG)$ if and only if $\phi^{o}_{[P]} = \phi^{o}_{[Q]}$. Given a balanced sign function $\phi^{o}$ of $\EuScript{F}_{\GG}^{o}$, there exists $[P] \in \B(\GG)$ such that $\phi^{o}_{[P]} = \phi^{o}$, when $\phi^{o}$ is \emph{admissible}. Thus $\kappa(\GG)$ equals the number of admissible sign functions of $\EuScript{F}_{\GG}^{o}$.

\medskip

Now, it only remains to show that any balanced sign functions of $\EuScript{F}_{\GG}^{o}$ is admissible. Let $\psi^{o}$ be any balanced sign function of $\EuScript{F}_{\GG}^{o}$. Since $\EuScript{F}_{\GG}^{g}$ is a forest, we can define a balanced sign function $\psi^{g}$ of $\EuScript{F}_{\GG}^{g}$ so that for any node $v$ of valence $2$, two edges incident on $v$ have different signs. Combining $\psi^{o}$ and $\psi^{g}$, we obtain a balanced sign function of $\EuScript{F}_{\GG}$, since each vertex of valence $3$ in $\EuScript{F}_{\GG}$ is incident to three edges in $\EuScript{F}_{\GG}^{g}$, to three edges in $\EuScript{F}_{\GG}^{o}$, or to two edges in $\EuScript{F}_{\GG}^{g}$ and one edge in $\EuScript{F}_{\GG}^{o}$, by the classification of 2-Lannér Coxeter groups (see Theorem \ref{t:maxwell}).
\end{proof}

We denote by $n_{2}(\GG)$ (resp. $n_{3}(\GG)$) the number of nodes of valence $2$ (resp. $3$) in $\EuScript{F}_{\GG}^{o}$, and by $n_c(\GG)$ the number of connected components of $\EuScript{F}_{\GG}^{o}$ (see Figure \ref{fig:forest}).

\begin{theorem}\label{thm:number_of_components}
Let $\GG$ be a convex-projectivizable, irreducible, large, 2-perfect labeled truncation polytope of dimension $d \geqslant 4$. Then the number $\kappa(\GG)$ of connected components of $\B(\GG)$ is 
$$ 2^{n_2(\GG) + n_c(\GG)} \cdot 3^{n_3(\GG)}$$   
\end{theorem}

\begin{proof}
Let $\left\{ \EuScript{F}_{\GG,j}^{o} \right\}_{j=1}^{n_c(\GG)}$ be the set of connected components of $\EuScript{F}_{\GG}^{o}$. It is easy to see that the number of balanced sign functions of $\EuScript{F}_{\GG,j}^{o}$ equals $ 2^{n_{2,j}+1}  \cdot 3^{n_{3,j}}$, where $n_{i,j}$ $(i=2,3)$ denote the number of nodes of valence $i$ of $\EuScript{F}_{\GG,j}^{o}$. Thus the number of balanced sign functions of $\EuScript{F}_{\GG}^{o}$ equals
$$ \prod_{j=1}^{ n_c(\GG) } 2^{n_{2,j}+1}  \cdot 3^{n_{3,j}} \,=\, 2^{n_2(\GG) + n_c(\GG)} \cdot 3^{n_3(\GG)},$$
since $\sum_j n_{i,j} = n_{i}(\GG)$. Our theorem follows from Lemma \ref{lem:admissible}.
\end{proof}

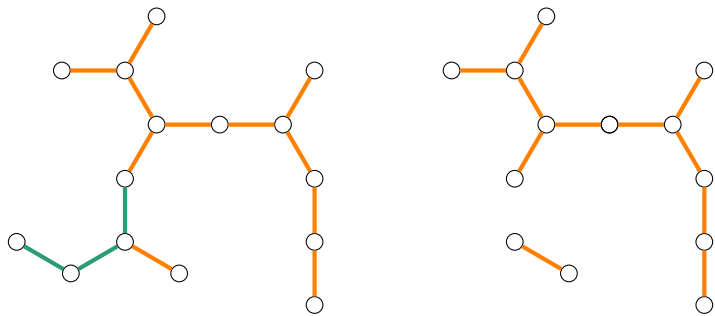
\begin{figure}[ht!]
\begin{minipage}{0.2\textwidth}
\centering
\begin{tikzpicture}[scale=0.6, every node/.style={transform shape}] 
\node[draw,circle] (A) at (0,0){};
\node[draw,circle] (B) at (1.4,0){};
\node[draw,circle] (C) at (2.1,1.2){};
\node[draw,circle] (D) at (2.1,-1.2){};
\node[draw,circle] (E) at (3.5,-1.2){};
\node[draw,circle] (F) at (1.4,-2.4){};
\node[draw,circle] (G) at (4.9,-1.2){};
\node[draw,circle] (H) at (5.6,0){};
\node[draw,circle] (I) at (5.6,-2.4){};
\node[draw,circle] (J) at (5.6,-3.8){};
\node[draw,circle] (K) at (5.6,-5.2){};
\node[draw,circle] (L) at (1.4,-3.8){};
\node[draw,circle] (M) at (0.2,-4.5){};
\node[draw,circle] (N) at (-1.0,-3.8){};
\node[draw,circle] (O) at (2.6,-4.5){};
\draw[orange,ultra thick] (A) -- (B) node[above,midway] {};
\draw[orange,ultra thick] (B) -- (C) node[above,midway] {};
\draw[orange,ultra thick] (B) -- (D) node[above,midway] {};
\draw[orange,ultra thick] (D) -- (E) node[above,midway] {};
\draw[orange,ultra thick] (D) -- (F) node[above,midway] {};
\draw[orange,ultra thick] (E) -- (G) node[above,midway] {};
\draw[orange,ultra thick] (G) -- (H) node[above,midway] {};
\draw[orange,ultra thick] (G) -- (I) node[above,midway] {};
\draw[orange,ultra thick] (I) -- (J) node[above,midway] {};
\draw[orange,ultra thick] (J) -- (K) node[above,midway] {};
\draw[green_okabe_ito,ultra thick] (F) -- (L) node[above,midway] {};
\draw[green_okabe_ito,ultra thick] (L) -- (M) node[above,midway] {};
\draw[green_okabe_ito,ultra thick] (M) -- (N) node[above,midway] {};
\draw[orange,ultra thick] (L) -- (O) node[above,midway] {};
\end{tikzpicture}
\end{minipage}
\quad\quad\quad\quad\quad\quad
\begin{minipage}{0.2\textwidth}
\centering
\begin{tikzpicture}[scale=0.6, every node/.style={transform shape}] 
\node[draw,circle] (A) at (0,0){};
\node[draw,circle] (B) at (1.4,0){};
\node[draw,circle] (C) at (2.1,1.2){};
\node[draw,circle] (D) at (2.1,-1.2){};
\node[draw,circle] (E) at (3.5,-1.2){};
\node[draw,circle] (E2) at (3.5,-1.2){};
\node[draw,circle] (F) at (1.4,-2.4){};
\node[draw,circle] (G) at (4.9,-1.2){};
\node[draw,circle] (H) at (5.6,0){};
\node[draw,circle] (I) at (5.6,-2.4){};
\node[draw,circle] (J) at (5.6,-3.8){};
\node[draw,circle] (K) at (5.6,-5.2){};
\node[draw,circle] (L) at (1.4,-3.8){};
\node[draw,circle] (O) at (2.6,-4.5){};
\draw[orange,ultra thick] (A) -- (B) node[above,midway] {};
\draw[orange,ultra thick] (B) -- (C) node[above,midway] {};
\draw[orange,ultra thick] (B) -- (D) node[above,midway] {};
\draw[orange,ultra thick] (D) -- (E) node[above,midway] {};
\draw[orange,ultra thick] (D) -- (F) node[above,midway] {};
\draw[orange,ultra thick] (E) -- (G) node[above,midway] {};
\draw[orange,ultra thick] (G) -- (H) node[above,midway] {};
\draw[orange,ultra thick] (G) -- (I) node[above,midway] {};
\draw[orange,ultra thick] (I) -- (J) node[above,midway] {};
\draw[orange,ultra thick] (J) -- (K) node[above,midway] {};
\draw[orange,ultra thick] (L) -- (O) node[above,midway] {};
\end{tikzpicture}
\end{minipage}
\caption{The forest $\EuScript{F}_{\GG}$ and the orange forest $\EuScript{F}_{\GG}^{o}$ of a labeled polytope $\GG$. The number $\kappa(\GG)$ equals $2^5\cdot 3^3$, since $n_2(\GG)=3$, $n_3(\GG)=3$ and $n_c(\GG)=2$.}
\label{fig:forest}
\end{figure}

\section{Dimension $\geqslant 6$}\label{section:higherDim}

Due to the rarity of 2-perfect labeled simplices in dimension $d > 5$, it is easy to describe the deformation spaces of each individual 2-perfect labeled polytope in those dimension. So, we exhibit them in the decreasing order of dimension $d$.

\subsection{Dimension beyond $9$}

By Theorem \ref{MainTheorem}, there exists no convex-projectivizable, irreducible, large, 2-perfect labeled truncation polytope in dimension $d > 9$.  

\subsection{Dimension $9$}

There are three 2-Lannér Coxeter groups of rank $10$ (see Theorem \ref{t:maxwell} and Table \ref{Table_dim9}). All are of tree type, and quasi-Lannér not Lannér. We denote them by $W_{9i}$ $(i = 1, 2, 3)$ and let $\Ss_{9i} := \Ss_{W_{9i}}$. 

\begin{theorem}\label{classi_dim9}
In dimension $9$, there exist only three convex-projectivizable, irreducible, large, 2-perfect labeled truncation polytopes: $\Ss_{9i}$ $(i=1,2,3)$. Each labeled simplex $\Ss_{9i}$ is hyperbolizable and rigid.
\end{theorem}

\begin{proof}
Let $\GG$ be a convex-projectivizable, irreducible, large, 2-perfect labeled truncation polytope of dimension $9$. We denote by $\mathcal{P}_e$ the set of essential prismatic circuits of $\GG$ and by $ \{ \Ss_j^{\dagger \mathcal{V}_j} \}_{j=1}^{k_e + 1} $ the decomposition of $\GG$ along $\mathcal{P}_e$. Then each $\Ss_j$ equals $\Ss_{9i}$ for $i \in \{ 1,2,3\}$. Every vertex of $\Ss_{9i}$ is either spherical, or affine but not $\tilde{A}$. Thus $\mathcal{P}_e = \varnothing$, by Lemma \ref{lemma:property:prismatic}, and $\GG = \Ss_{9i}$ for $i \in \{ 1,2,3\}$. Each simplex $\Ss_{9i}$ is hyperbolizable and rigid, by Theorem \ref{MainTheorem}.
\end{proof}

\subsection{Dimension $8$}\label{subsec:dim8}

There are four 2-Lannér Coxeter groups of rank $9$ (see Theorem \ref{t:maxwell} and Table \ref{Table_dim8}). Three of them, $W_{8i}$ $(i=1,2,3)$, are of tree type and one of them, $W_{8\tau}$, is of pan type. All are quasi-Lannér but not Lannér. We set $\Ss_{8i} := \Ss_{W_{8i}}$ and $\Ss_{8\tau} := \Ss_{W_{8\tau}}$.

\begin{theorem}\label{t:final8}
In dimension $8$, there exist ten convex-projectivizable, irreducible, large, 2-perfect labeled truncation polytopes: $\Ss_{8i}$ $(i =1, 2, 3)$, $\Ss_{8\tau}$, $\Ss_{8\tau}^{\dagger}$ and $\Ss_{8\tau}^{\dagger} \sharp_{\phi_j} \Ss_{8\tau}^{\dagger}$ $(j=1, \dots, 5)$.\footnote{The definition of labeled polytopes $\Ss_{8\tau}^{\dagger}$ and $\Ss_{8\tau}^{\dagger} \sharp_{\phi_j} \Ss_{8\tau}^{\dagger}$ ($j=1, \dotsc, 5$) is given in the proof.}
\begin{enumerate}
\item each simplex $\Ss_{8i}$ $(i = 1, 2, 3)$ is hyperbolizable and rigid;
\item the simplex $\Ss_{8\tau}$ is hyperbolizable, and $\B(\Ss_{8\tau}) \simeq \R$;\footnote{By $X \simeq Y$, we mean that two spaces $X$ and $Y$ are homeomorphic.}
\item the prism $\Ss_{8\tau}^{\dagger}$ is not hyperbolizable, and $\B(\Ss_{8\tau}^{\dagger}) \simeq \R^{\star}$;
\item each prism $\Ss_{8\tau}^{\dagger}\sharp_{\phi_j} \Ss_{8\tau}^{\dagger}$ $(j=1, \dotsc, 5)$ is not hyperbolizable, and $\B(\Ss_{8\tau}^{\dagger}\sharp_{\phi_j}  \Ss_{8\tau}^{\dagger}) \simeq \R^{\star} \times \R$.
\end{enumerate}
\end{theorem}

\begin{proof}
Let $\GG$ be a convex-projectivizable, irreducible, large, 2-perfect labeled truncation polytope of dimension $8$. We denote by $ \{ \Ss_j^{\dagger \mathcal{V}_j} \}_{j=1}^{k_e+1} $ the decomposition of $\GG$ along the essential prismatic circuits. Then each $\Ss_j$ is equal to $\Ss_{8i}$ $(i=1,2,3)$ or $\Ss_{8\tau}$. There are two cases to consider: (i) one of $\Ss_j$ equals $\Ss_{8i}$ $(i=1,2,3)$ and (ii) all $\Ss_j$ equal $\Ss_{8\tau}$.

\medskip

In case (i), $\GG = \Ss_{8i}$ for $i \in \{ 1,2,3\}$, as in the proof of Theorem \ref{classi_dim9}.

\medskip

In case (ii), only one vertex of $\Ss_{8\tau}$, say $v$, is $\tilde{A}_7$, and the other vertices are either spherical or affine but not $\tilde{A}_7$.
Thus $k_e = 0$ or $1$, by Lemma \ref{lemma:property:prismatic}. If $k_e = 0$, then $\GG = \Ss_{8\tau}$ or $\Ss_{8\tau}^{\dagger} := \Ss_{8\tau}^{\dagger v}$. Otherwise, $\GG = \Ss_{8\tau}^{\dagger} \sharp_{\phi_j} \Ss_{8\tau}^{\dagger}$ for $j=1, \dotsc, 5$ (see Table \ref{tab:excep_8prism} for their Coxeter groups). Here, $\phi_j$ indicates that there exist five different gluing of two copies of $\Ss_{8\tau}^{\dagger}$. 

\medskip

Theorem \ref{MainTheorem} completes the proof.
\end{proof}

\subsection{Dimension $7$}

There are four 2-Lannér Coxeter groups of rank $8$ (see Theorem \ref{t:maxwell} and Table \ref{Table_dim7}). Three of them, $W_{7i}$ $(i=1,2,3)$, are of tree type and one of them, $W_{7\tau}$, is of pan type. All are quasi-Lannér but not Lannér. We set $\Ss_{7i} := \Ss_{W_{7i}}$ and $\Ss_{7\tau} := \Ss_{W_{7\tau}}$. The situation is very similar to the one in dimension $8$:

\begin{theorem}\label{t:final7}
In dimension $7$, there exist nine convex-projectivizable, irreducible, large, 2-perfect labeled truncation polytopes: $\Ss_{7i}$ $(i=1,2,3)$, $\Ss_{7\tau}$, $\Ss_{7\tau}^{\dagger}$ and $\Ss_{7\tau}^{\dagger}\sharp_{\phi_j} \Ss_{7\tau}^{\dagger}$ $(j=1,\dots,4)$.\footnote{The definition of polytopes $\Ss_{7\tau}^{\dagger}$ and $\Ss_{7\tau}^{\dagger} \sharp_{\phi_j} \Ss_{7\tau}^{\dagger}$ is analogous to the one in the proof of Theorem \ref{t:final8}.}
\begin{enumerate}
\item each simplex $\Ss_{7i}$ $(i=1,2,3)$ is hyperbolizable and rigid;
\item the simplex $\Ss_{7\tau}$ is hyperbolizable, and $\B(\Ss_{7\tau}) \simeq \R$;
\item the prism $\Ss_{7\tau}^{\dagger}$ is not hyperbolizable, and $\B(\Ss_{7\tau}^{\dagger}) \simeq \R^{\star}$;
\item each prism $\Ss_{7\tau}^{\dagger} \sharp_{\phi_j} \Ss_{7\tau}^{\dagger}$ $(j=1,\dots, 4)$ is not hyperbolizable, and $\B(\Ss_{7\tau}^{\dagger}\sharp_{\phi_j} \Ss_{7\tau}^{\dagger}) \simeq \R^{\star} \times \R$.
\end{enumerate}
\end{theorem}

\begin{proof}
The proof is similar to the one of Theorem \ref{t:final8}. See Table \ref{tab:excep_7prism} for the Coxeter groups of the four different prisms $\Ss_{7\tau}^{\dagger} \sharp_{\phi_j} \Ss_{7\tau}^{\dagger}$ $(j=1,\dots,4)$.
\end{proof}

\subsection{Dimension $6$}

There are three 2-Lannér Coxeter groups of rank $7$ (see Theorem \ref{t:maxwell} and Table \ref{Table_dim6}). Two of them, $W_{6i}$ $(i=1,2)$, are of tree type and one of them, $W_{6\tau}$, is of pan type. All are quasi-Lannér but not Lannér. The situation is again very similar to the one in dimension $8$.

\begin{theorem}\label{t:final6}
In dimension $6$, there exist eight convex-projectivizable, irreducible, large, 2-perfect labeled truncation polytopes: $\Ss_{6i}$ $(i=1,2)$ $\Ss_{6\tau}$, $\Ss_{6\tau}^{\dagger}$ and $\Ss_{6\tau}^{\dagger} \sharp_{\phi_j} \Ss_{6\tau}^{\dagger}$ $(j=1,\dotsc,4)$.\footnote{The definition of those polytopes is analogous to the one in Section \ref{subsec:dim8}.}
\begin{enumerate}
\item two simplices $\Ss_{6i}$ are hyperbolizable and rigid;
\item the simplex $\Ss_{6\tau}$ is hyperbolizable, and $\B(\Ss_{6\tau}) \simeq \R$;
\item the prism $\Ss_{6\tau}^{\dagger}$ is not hyperbolizable, and $\B(\Ss_{6\tau}^{\dagger}) \simeq \R^{\star}$;
\item four prisms $\Ss_{6\tau}^{\dagger} \sharp_{\phi_j} \Ss_{6\tau}^{\dagger}$ are not hyperbolizable, and $\B(\Ss_{6\tau}^{\dagger}\sharp_{\phi_j} \Ss_{6\tau}^{\dagger}) \simeq \R^{\star} \times \R$.
\end{enumerate}
\end{theorem}

\begin{proof}
The proof is similar to the one of Theorem \ref{t:final8}. See Table \ref{tab:excep_6prism} for the Coxeter groups of the four different prisms $\Ss_{6\tau}^{\dagger} \sharp_{\phi_j} \Ss_{6\tau}^{\dagger}$ $(j=1,\dots,4)$.
\end{proof}

\section{Dimension $5$}\label{section:Dim5}

The situation in dimension $5$ is richer than that in higher dimensions. There are twenty three 2-Lannér Coxeter groups of rank $6$ (see Theorem \ref{t:maxwell} and Table \ref{Table_dim5}). Eighteen of them, $W_{5i}$ $(i=1,\dotsc, 9)$ and $W_{5ti}$ $(i=1, \dotsc, 9)$, are of tree type, two of them, $W_{5ci}$ $(i=1,2)$, are of cycle type, and three of them, $W_{5\tau}$ and $W_{5pi}$ $(i=1,2)$, are of pan type. We set $\Ss_{*} := \Ss_{W_{*}}$ for $* \in \{ 5i, 5ti, 5ci, 5\tau, 5pi \}$. Unlike higher dimensions, there exist irreducible, large, 2-perfect labeled simplices in dimension $5$ that have at least two Lann\'{e}r vertices. This makes it possible to build infinitely many truncation 5-polytopes.

\begin{theorem}\label{t:final5}
Let $\GG$ be a convex-projectivizable, irreducible, large, 2-perfect labeled truncation $5$-polytope. Then $\B(\GG)$ is homeomorphic to $\R^*$, $\R^*\times \R$ or $\R^{m}$ for $m \in \mathbb{N} \cup \{ 0 \}$. More precisely, the following hold:
\begin{enumerate}
\item $\GG = \Ss_{5\tau}^\dagger$ if and only if $\B(\GG) \simeq \R^*$;

\item $\GG = \Ss_{5\tau}^{\dagger}\sharp_{\phi_j} \Ss_{5\tau}^{\dagger}$ $(j=1,2,3)$ if and only if $\B(\GG) \simeq \R^* \times \R$; 

\item otherwise, $\GG$ is hyperbolizable and $\B(\GG) \simeq \R^{b(\GG)}$.
\end{enumerate}
In addition, for any $m \in \mathbb{N} \cup \{ 0 \} $, there exists an irreducible, large, \emph{perfect} labeled truncation 5-polytope $\GG$ such that $\B(\GG) \simeq \mathbb{R}^m$.   
\end{theorem}

\begin{proof}
By Theorems \ref{MainTheorem} and \ref{thm:geometrization}, the space $\B(\GG)$ is disconnected if and only if there exists an $\tilde{A}_{4}$ prismatic circuit of $\GG$. This is equivalent to require that one of $\Ss_j$ equals $\Ss_{5\tau}$, i.e. $\GG = \Ss_{5\tau}^\dagger$ or $\GG = \Ss_{5\tau}^{\dagger}\sharp_{\phi_j} \Ss_{5\tau}^{\dagger}$ $(j=1,2,3)$ (see Table \ref{tab:excep_prism}). All three items of Theorem \ref{t:final5} then follow again from Theorem \ref{MainTheorem}.

\medskip

The labeled simplices $\Ss_{5t1}$, $\Ss_{5t3}$, $\Ss_{5t7}$ and $\Ss_{5t9}$ have only spherical or Lann\'{e}r vertices, and each of them has two or three Lann\'{e}r vertices. For each $* \in \{ 5t1, 5t3, 5t7, 5t9 \}$, the once-truncated simplex $\Ss_{*}^{\dagger}$ obtained from $\Ss_*$ by truncating all its Lann\'{e}r vertices is perfect. So, if a polytope $\GG$ is obtained by gluing $m+1$ copies of $\Ss_{*}^{\dagger}$, then $\GG$ is also perfect. Since $m$ is the number of essential prismatic circuits of $\GG$, the space $\B(\GG)$ is homeomorphic to $\mathbb{R}^m$, by (the proof of) Theorem \ref{MainTheorem}.
\end{proof}

\section{Geometric interpretation}\label{section:geometric_interpretation}

Let $P$ be an irreducible, loxodromic Coxeter polytope, $\Gamma_P$ the group generated by the reflections in the facets of $P$, and $\Omega_P$ the interior of the union of $\Gamma_P$-translates of $P$. Then $\Omega_P$ is a properly convex domain, hence it admits a Hilbert metric $d_{\Omega_P}$. The polytope $P$ is said to be of \emph{finite volume} if $P \cap \Omega_P$ has finite volume with respect to the Hausdorff measure $\mu_{\Omega_P}$ induced by $d_{\Omega_P}$, \emph{convex-cocompact} if $P \cap C(\LGP) \subset \O_P$, or \emph{geometrically finite} if $\mu_{\O_P}(P \cap C(\LGP)) < \infty$, where $\Lambda_P$ is the limit set of $\G_P$ and $C(\LGP)$ is the convex hull of $\Lambda_P$ in $\O_P$. Those notions are studied in details in \cite{Marquis:2014aa}.

\medskip

The action of $\Gamma_P$ on $\Omega_P$ is cocompact if and only if $P$ is perfect. In this case, the convex domain $\Omega_P$ is strictly convex if and only if $\Gamma_P$ is word-hyperbolic, by work of Benoist \cite[Prop.~2.5]{benoist_qi}. The following theorems state analogous results for $2$-perfect Coxeter polytopes.

\begin{theorem}[{\cite[Th.~A]{Marquis:2014aa}}]\label{theo_action1}
Let $P$ be an irreducible, loxodromic, 2-perfect Coxeter polytope. Then:
\begin{itemize}
\item $P$ is geometrically finite;
\item $P$ is convex cocompact if and only if the link $P_v$ of each vertex $v$ of $P$ is elliptic or loxodromic;
\item $P$ is of finite volume if and only if $P$ is quasi-perfect.
\end{itemize}
\end{theorem}

\begin{theorem}[{\cite[Th.~E]{Marquis:2014aa}}]\label{theo_action2}
Let $P$ be an irreducible, loxodromic, quasi-perfect Coxeter polytope, and $\mathcal{V}$ the set of all parabolic vertices of $P$. Then the convex domain $\O_P$ is strictly convex if and only if the group $\G_P$ is relatively hyperbolic with respect to the collection $\left\{ \G_v \right\}_{v \in \mathcal{V}}$ of the subgroups $\G_v$ generated by the reflections in the facets of $P$ that contain $v$. 
\end{theorem}

As in the proof of {\cite[Th.~F]{Marquis:2014aa}}, one can prove the following lemma:

\begin{lemma}\label{lem:strictlyconvex}
Let $P$ be an irreducible, loxodromic, quasi-perfect Coxeter truncation $d$-polytope, and $\mathcal{V}$ the set of all parabolic vertices of $P$. If $P$ has an $\tilde{A}_{d-1}$ prismatic circuit, then $\Gamma_P$ is not relatively hyperbolic with respect to $\left\{ \G_v \right\}_{v \in \mathcal{V}}$.
\end{lemma}

\begin{rem}
The converse of Lemma \ref{lem:strictlyconvex} also holds: if the Coxeter group $W_P$ is large and all the prismatic circuits of $P$ are Lann\'{e}r, then $\Gamma_P$ is relatively hyperbolic with respect to $\left\{ \G_v \right\}_{v \in \mathcal{V}}$ (see Theorem \ref{thm:geometrization}). 
\end{rem}

The space of finite volume (resp. convex-cocompact, resp. geometrically finite) Coxeter polytopes realizing $\GG$ is denoted $\B^{vf}(\GG)$ (resp. $\B^{cc}(\GG)$, resp.  $\B^{gf}(\GG)$).

\begin{prop}
Let $\GG$ be an irreducible, large, $2$-perfect labeled truncation polytope of dimension $d \geqslant 4$, and $\mathcal{V}_{\tilde{A}}$ the set of $\tilde{A}_{d-1}$ vertices of $\GG$. Then:
\begin{itemize}
\item $\B(\GG) = \B^{gf}(\GG)$; 
\item $\B^{cc}(\GG)$ is an open subset of $\B(\GG)$. Moreover, $\B^{cc}(\GG) = \B(\GG^{\dagger \mathcal{V}_{\tilde{A}}})$;
\item $\B^{vf}(\GG)$ is a submanifold of $\B(\GG)$.
\end{itemize}
\end{prop}

\begin{proof}
The first item is a consequence of the first item of Theorem \ref{theo_action1}. 

\medskip

For each $v \in \mathcal{V}_{\tilde{A}}$, the subspace
$$
\Sigma_v := \{ [P] \in \B(\GG) \mid R_{\mathcal{C}_v}([P])=0 \}
$$
is a hypersurface of $\B(\GG)$, where $\mathcal{C}_v$ is the circuit that corresponds to $v$. By the second item of Theorem \ref{theo_action1}, the complement of $\bigcup_{v \in \mathcal{V}_{\tilde{A}}} \Sigma_v$ in $\B(\GG)$ is $\B^{cc}(\GG)$ and also equals $\B(\GG^{\dagger \mathcal{V}_{\tilde{A}}})$ thanks to Corollary \ref{cor:moduli_truncation}.

\medskip

Theorem \ref{theo_action1} implies that if $\GG$ has a Lann\'{e}r vertex, then $\B^{vf}(\GG) = \varnothing$. Assume that $\GG$ has no Lann\'{e}r vertex. By the third item of Theorem \ref{theo_action1} and 
Remark \ref{rem:parabolic_inv=0}, the intersection $\bigcap_{v \in V_{\tilde{A}}} \Sigma_v$ is $\B^{vf}(\GG)$. In the coordinates defined in Section \ref{subsec:evaluation}, the hypersurface $\Sigma_v$ is a hyperplane. So $\B^{vf}(\GG)$ is a submanifold of $\B(\GG)$.
\end{proof}

\begin{proof}[Proof of Theorem \ref{thm:existence_quasi-divisible}]
Let $P$ be a Coxeter polytope of dimension $d = 8$ realizing the labeled prism $\Ss_{8\tau}^{\dagger}$ or $\Ss_{8\tau}^{\dagger} \sharp_{\phi_j} \Ss_{8\tau}^{\dagger}$ $(j=1,\dotsc,5)$ of Table \ref{tab:excep_8prism}, or a Coxeter polytope of dimension $d = 4$ realizing the prism $\Ss_{4\tau}^{\dagger}$ or $\Ss_{4\tau}^{\dagger} \sharp_{\phi_j} \Ss_{4\tau}^{\dagger}$ $(j=1,\dotsc,3)$ of Table \ref{fig:example_dim4}. Then $P$ is an irreducible, loxodromic, quasi-perfect Coxeter $d$-polytope with one $\tilde{A}_{d-1}$ prismatic circuit, and hence $\Gamma_P$ is not relatively hyperbolic with respect to $\left\{ \G_v \right\}_{v \in \mathcal{V}}$, by Lemma \ref{lem:strictlyconvex}. So, $\Omega_P$ is an indecomposable, inhomogeneous, non-strictly convex, quasi-divisible $d$-domain by $\Gamma_P$ such that $\Quotient{\Omega_P}{\Gamma_P}$ has only generalized cusps of type $0$, by Theorems \ref{theo_action1} and \ref{theo_action2}. 
\end{proof}

\begin{rem}\label{rem:other_quasi-divisible}
There exists an irreducible, loxodromic, quasi-perfect Coxeter $d$-polytope $P_d$ whose reflection group $\Gamma$ is not relatively hyperbolic with respect to $\left\{ \G_v \right\}_{v \in \mathcal{V}}$ in dimension $d = 3$, by \cite[Th.~F]{Marquis:2014aa}, and in dimension $d = 5$ and $7$, by \cite{CLM2} (see the left and the middle diagrams of Table \ref{fig:example_567}). A computation similar to the one in the proof of \cite[Prop.~7.1]{CLM2} can show that such a Coxeter polytope $P_d$ exists also in dimension $d=6$ (see the right diagram of Table \ref{fig:example_567}).
\end{rem}

\begin{table}[ht!]
\begin{minipage}{0.2\textwidth}
\centering
\begin{tikzpicture}[thick,scale=0.6, every node/.style={transform shape}] 
\node[draw, inner sep=5pt] (1) at (-1.3,2.3) {\Large $\Ss_{4\tau}^\dagger$};
\node[draw,circle] (A) at (0:1){};
\node[draw,circle] (B) at (90:1){};
\node[draw,circle] (C) at (180:1){};
\node[draw,circle] (D) at (-90:1){};
\node[draw,circle] (E) at (-1-1.414,0){};
\node[draw,circle] (F) at (-1-1.414-1.414,0){};
\draw (A) -- (B) node[above,midway] {};
\draw (B) -- (C) node[above,midway] {};
\draw (C) -- (D) node[above,midway] {};
\draw (D) -- (A) node[above,midway] {};
\draw (E) -- (C) node[above,midway] {$4$};
\draw (E) -- (F) node[above,midway] {$\infty$};
\end{tikzpicture}
\end{minipage}
\quad
\begin{minipage}{0.2\textwidth}
\centering
\begin{tikzpicture}[thick,scale=0.6, every node/.style={transform shape}] 
\node[draw, inner sep=5pt] (1) at (-0.8,2.3) {\Large $\Ss_{4\tau}^\dagger \sharp_{\phi_1} \Ss_{4\tau}^\dagger$};
\node[draw,circle] (A) at (0:1){};
\node[draw,circle] (B) at (90:1){};
\node[draw,circle] (C) at (180:1){};
\node[draw,circle] (D) at (-90:1){};
\node[draw,circle] (E) at (-0.7-1.4,1){};
\node[draw,circle] (F) at (-0.7-1.4,-1){};
\draw (A) -- (B) node[above,midway] {};
\draw (B) -- (C) node[above,midway] {};
\draw (C) -- (D) node[above,midway] {};
\draw (D) -- (A) node[above,midway] {};
\draw (C) -- (E) node[above,midway] {4};
\draw (C) -- (F) node[below,midway] {4};
\draw (E) -- (F) node[left,midway] {$\infty$};
\end{tikzpicture}
\end{minipage}
\begin{minipage}{0.2\textwidth}
\centering
\begin{tikzpicture}[thick,scale=0.6, every node/.style={transform shape}] 
\node[draw, inner sep=5pt] (1) at (-0.8,2.3) {\Large $\Ss_{4\tau}^\dagger \sharp_{\phi_2} \Ss_{4\tau}^\dagger$};
\node[draw,circle] (A) at (45:1){};
\node[draw,circle] (B) at (135:1){};
\node[draw,circle] (C) at (225:1){};
\node[draw,circle] (D) at (-45:1){};
\node[draw,circle] (E) at (-0.7-1.4,0.7){};
\node[draw,circle] (F) at (-0.7-1.4,-0.7){};
\draw (A) -- (B) node[above,midway] {};
\draw (B) -- (C) node[above,midway] {};
\draw (C) -- (D) node[above,midway] {};
\draw (D) -- (A) node[above,midway] {};
\draw (B) -- (E) node[above,midway] {4};
\draw (C) -- (F) node[below,midway] {4};
\draw (E) -- (F) node[left,midway] {$\infty$};
\end{tikzpicture}
\end{minipage}
\begin{minipage}{0.2\textwidth}
\centering
\begin{tikzpicture}[thick,scale=0.6, every node/.style={transform shape}] 
\node[draw, inner sep=5pt] (1) at (-0.8,2.3) {\Large $\Ss_{4\tau}^\dagger \sharp_{\phi_3} \Ss_{4\tau}^\dagger$};
\node[draw,circle] (A) at (0:1){};
\node[draw,circle] (B) at (90:1){};
\node[draw,circle] (C) at (180:1){};
\node[draw,circle] (D) at (-90:1){};
\node[draw,circle] (E) at (-0.7-1.4,1){};
\node[draw,circle] (F) at (-0.7-1.4,-1){};
\draw (A) -- (B) node[above,midway] {};
\draw (B) -- (C) node[above,midway] {};
\draw (C) -- (D) node[above,midway] {};
\draw (D) -- (A) node[above,midway] {};
\draw (B) -- (E) node[above,midway] {4};
\draw (D) -- (F) node[below,midway] {4};
\draw (E) -- (F) node[left,midway] {$\infty$};
\end{tikzpicture}
\end{minipage}
\caption{The Coxeter diagrams of the labeled 4-prism $\Ss_{4\tau}^{\dagger}$ or $\Ss_{4\tau}^{\dagger} \sharp_{\phi_j} \Ss_{4\tau}^{\dagger}$}
\label{fig:example_dim4}
\end{table}

\begin{table}[ht!]
\begin{minipage}{0.2\textwidth}
\centering
\begin{tikzpicture}[thick,scale=0.6, every node/.style={transform shape}]
\node[draw,circle] (3) at (0,0) {};
\node[draw,circle,below left=0.8cm of 3] (1) {};
\node[draw,circle,above left=0.8cm of 3] (2) {};
\node[draw,circle,above left=0.8cm of 1] (-1) {};

\node[draw,circle,right=0.8cm of 3] (4) {};
\node[draw,circle,above right=0.8cm of 4] (5) {};
\node[draw,circle,below right=0.8cm of 4] (6) {};

\draw (1) -- (-1);
\draw (2) -- (-1);

\draw (2) -- (3);
\draw (3)--(1);
\draw (4) -- (5) node[above,near start] {$4$};
\draw (5) -- (6) node[right,midway] {$m \geqslant 5$};
\draw (3) -- (4) node[above,midway] {};

\end{tikzpicture}
\end{minipage}
\qquad
\begin{minipage}{0.2\textwidth}
\centering
\begin{tikzpicture}[thick,scale=0.6, every node/.style={transform shape}]
\node[draw,circle] (3) at (1,0) {};

\node[draw,circle] (2) at (0.5,0.866) {};
\node[draw,circle] (1) at (0.5,-0.866) {};
\node[draw,circle] (-3) at (-1,0) {};
\node[draw,circle] (-2) at (-0.5,0.866) {};
\node[draw,circle] (-1) at (-0.5,-0.866) {};

\node[draw,circle,right=0.8cm of 3] (4) {};
\node[draw,circle,above right=0.8cm of 4] (5) {};
\node[draw,circle,below right=0.8cm of 4] (6) {};

\draw (3)--(1);
\draw (1) -- (-1);
\draw (-1) -- (-3);
\draw (-3) -- (-2);
\draw (-2) -- (2);
\draw (2) -- (3);

\draw (4) -- (5);
\draw (5) -- (6) node[right,midway] {$m \geqslant 7$};
\draw (3) -- (4) node[above,midway] {};
\end{tikzpicture}
\end{minipage}
\quad\quad
\begin{minipage}{0.2\textwidth}
\centering
\begin{tikzpicture}[thick,scale=0.6, every node/.style={transform shape}] 
\node[draw,circle] (A) at (0:1){};
\node[draw,circle] (B) at (90:1){};
\node[draw,circle] (C) at (180:1){};
\node[draw,circle] (D) at (-90:1){};
\node[draw,circle] (E) at (-1-1.414,0){};
\node[draw,circle] (F) at (-2-1.414,1){};
\node[draw,circle] (G) at (-3-1.414,0){};
\node[draw,circle] (H) at (-2-1.414,-1){};

\draw (A) -- (B) node[above,midway] {};
\draw (B) -- (C) node[above,midway] {};
\draw (C) -- (D) node[above,midway] {};
\draw (D) -- (A) node[above,midway] {};
\draw (E) -- (C) node[above,midway] {};
\draw (E) -- (F) node[above,midway] {};
\draw (F) -- (G) node[above,midway] {};
\draw (G) -- (H) node[above,midway] {};
\draw (H) -- (E) node[above,midway] {};
\end{tikzpicture}
\end{minipage}
\caption{The Coxeter diagrams of $P_5$, $P_7$ and $P_6$ from left to right}
\label{fig:example_567}
\end{table}

\appendix

\clearpage

\section{Spherical and affine Coxeter groups}\label{classi_diagram}

For the reader's convenience, we reproduce below the list of all irreducible spherical Coxeter diagrams and irreducible affine Coxeter diagrams in Table \ref{table:spherical_and_affine_diagram}. As usual we omit the label $3$ of edges from Coxeter diagrams.

\medskip

For spherical Coxeter groups, the index (in particular the $n$ for $A_n$, $B_n$ or $D_n$) is the number of nodes of the diagram. But, for affine Coxeter groups, the index (in particular the $n$ for $\tilde{A}_n$, $\tilde{B}_n$, $\tilde{C}_n$ or $\tilde{D}_n$) is \emph{one less than} the number of nodes.

\medskip
\medskip

\newcommand{\scalevalue}{0.50}

\begin{table}[h]
\begin{tabular}{m{2.5cm} m{3cm} m{1.5cm} m{2.5cm} m{3cm}}
$I_2(p)$ $(p\geqslant 5)$
&
\begin{tikzpicture}[thick,scale=\scalevalue, every node/.style={transform shape}]
\node[draw,circle] (I1) at (0,0) {};
\node[draw,circle,right=.8cm of I1] (I2) {};
\draw (I1) -- (I2)  node[above,midway] {$p$};
\end{tikzpicture}
&&
$\tilde{A}_1$
&
\begin{tikzpicture}[thick,scale=\scalevalue, every node/.style={transform shape}]
\node[draw,circle] (I1) at (0,0) {};
\node[draw,circle,right=.8cm of I1] (I2) {};
\draw (I1) -- (I2)  node[above,midway] {$\infty$};
\end{tikzpicture}
\\
$A_n$ $(n \geqslant 1)$
&
\begin{tikzpicture}[thick,scale=\scalevalue, every node/.style={transform shape}]

\node[draw,circle] (A1) at (0,0) {};
\node[draw,circle,right=.8cm of A1] (A2) {};
\node[draw,circle,right=.8cm of A2] (A3) {};
\node[draw,circle,right=1cm of A3] (A4) {};
\node[draw,circle,right=.8cm of A4] (A5) {};

\draw (A1) -- (A2)  node[above,midway] {};
\draw (A2) -- (A3)  node[above,midway] {};
\draw[ dotted,thick] (A3) -- (A4) node[] {};
\draw (A4) -- (A5) node[above,midway] {};
\end{tikzpicture}
&&
$\tilde{A}_n$ $(n \geqslant 2)$
&
\begin{tikzpicture}[thick,scale=\scalevalue, every node/.style={transform shape}]
\node[draw,circle] (A1) at (0,0) {};
\node[draw,circle,above=.8cm of A1] (A2) {};
\node[draw,circle,right=.8cm of A2] (A3) {};
\node[draw,circle,right=.8cm of A3] (A4) {};
\node[draw,circle,right=.8cm of A4] (A5) {};
\node[draw,circle,below=.8cm of A5] (A6) {};
\node[draw,circle,below=.8cm of A6] (A7) {};
\node[draw,circle,left=.8cm of A7] (A8) {};
\node[draw,circle,left=.8cm of A8] (A9) {};
\node[draw,circle,left=.8cm of A9] (A10) {};

\draw (A1) -- (A2)  node[above,midway] {};
\draw (A2) -- (A3)  node[above,midway] {};
\draw (A3) -- (A4) node[] {};
\draw (A4) -- (A5) node[above,midway] {};
\draw (A5) -- (A6) node[] {};
\draw (A6) -- (A7) node[] {};
\draw (A7) -- (A8) node[] {};
\draw[ dotted,thick] (A8) -- (A9) node[] {};
\draw (A9) -- (A10) node[] {};
\draw (A10) -- (A1) node[] {};
\end{tikzpicture}
\\
\\
$B_n$ $(n \geqslant 2)$

&
\begin{tikzpicture}[thick,scale=\scalevalue, every node/.style={transform shape}]
\node[draw,circle] (B1) at (0,0) {};
\node[draw,circle,right=.8cm of B1] (B2) {};
\node[draw,circle,right=.8cm of B2] (B3) {};
\node[draw,circle,right=1cm of B3] (B4) {};
\node[draw,circle,right=.8cm of B4] (B5) {};

\draw (B1) -- (B2)  node[above,midway] {$4$};
\draw (B2) -- (B3)  node[above,midway] {};
\draw[ dotted,thick] (B3) -- (B4) node[] {};
\draw (B4) -- (B5) node[above,midway] {};
\end{tikzpicture}

&&
$\tilde{B}_2=\tilde{C}_2$

&
\begin{tikzpicture}[thick,scale=\scalevalue, every node/.style={transform shape}]
\node[draw,circle] (H1) at (0,0) {};
\node[draw,circle,right=.8cm of H1] (H2) {};
\node[draw,circle,right=.8cm of H2] (H3) {};

\draw (H1) -- (H2)  node[above,midway] {$4$};
\draw (H2) -- (H3)  node[above,midway] {$4$};
\end{tikzpicture}

\\
$H_3$
&
\begin{tikzpicture}[thick,scale=\scalevalue, every node/.style={transform shape}]
\node[draw,circle] (H1) at (0,0) {};
\node[draw,circle,right=.8cm of H1] (H2) {};
\node[draw,circle,right=.8cm of H2] (H3) {};

\draw (H1) -- (H2)  node[above,midway] {$5$};
\draw (H2) -- (H3)  node[above,midway] {};
\end{tikzpicture}
&&
$\tilde{B}_n$ $(n \geqslant 3)$
&
\begin{tikzpicture}[thick,scale=\scalevalue, every node/.style={transform shape}]
\node[draw,circle] (B1) at (0,0) {};
\node[draw,circle,right=.8cm of B1] (B2) {};
\node[draw,circle,right=.8cm of B2] (B3) {};
\node[draw,circle,right=1cm of B3] (B4) {};
\node[draw,circle,right=.8cm of B4] (B5) {};
\node[draw,circle,above right=.8cm of B5] (B6) {};
\node[draw,circle,below right=.8cm of B5] (B7) {};

\draw (B1) -- (B2)  node[above,midway] {$4$};
\draw (B2) -- (B3)  node[above,midway] {};
\draw[ dotted,thick] (B3) -- (B4) node[] {};
\draw (B4) -- (B5) node[above,midway] {};
\draw (B5) -- (B6) node[above,midway] {};
\draw (B5) -- (B7) node[above,midway] {};
\end{tikzpicture}
\\
$H_4$
&
\begin{tikzpicture}[thick,scale=\scalevalue, every node/.style={transform shape}]
\node[draw,circle] (HH1) at (0,0) {};
\node[draw,circle,right=.8cm of HH1] (HH2) {};
\node[draw,circle,right=.8cm of HH2] (HH3) {};
\node[draw,circle,right=.8cm of HH3] (HH4) {};

\draw (HH1) -- (HH2)  node[above,midway] {$5$};
\draw (HH2) -- (HH3)  node[above,midway] {};
\draw (HH3) -- (HH4)  node[above,midway] {};
\end{tikzpicture}
&&
$\tilde{C}_n$ $(n\geqslant 3)$
&
\begin{tikzpicture}[thick,scale=\scalevalue, every node/.style={transform shape}]
\node[draw,circle] (C1) at (0,0) {};
\node[draw,circle,right=.8cm of C1] (C2) {};
\node[draw,circle,right=.8cm of C2] (C3) {};
\node[draw,circle,right=1cm of C3] (C4) {};
\node[draw,circle,right=.8cm of C4] (C5) {};

\draw (C1) -- (C2)  node[above,midway] {$4$};
\draw (C2) -- (C3)  node[above,midway] {};
\draw[ dotted,thick] (C3) -- (C4) node[] {};
\draw (C4) -- (C5) node[above,midway] {$4$};
\end{tikzpicture}
\\
\\
$D_n$ $(n\geqslant 4)$
&
\begin{tikzpicture}[thick,scale=\scalevalue, every node/.style={transform shape}]

\node[draw,circle] (D1) at (0,0) {};
\node[draw,circle,right=.8cm of D1] (D2) {};
\node[draw,circle,right=1cm of D2] (D3) {};
\node[draw,circle,right=.8cm of D3] (D4) {};
\node[draw,circle, above right=.8cm of D4] (D5) {};
\node[draw,circle,below right=.8cm of D4] (D6) {};

\draw (D1) -- (D2)  node[above,midway] {};
\draw[ dotted] (D2) -- (D3);
\draw (D3) -- (D4) node[above,midway] {};
\draw (D4) -- (D5) node[above,midway] {};
\draw (D4) -- (D6) node[below,midway] {};
\end{tikzpicture}
&&
$\tilde{D}_n$ $(n\geqslant 4)$
&
\begin{tikzpicture}[thick,scale=\scalevalue, every node/.style={transform shape}]
\node[draw,circle] (D1) at (0,0) {};
\node[draw,circle,below right=0.8cm of D1] (D3) {};
\node[draw,circle,below left=0.8cm of D3] (D2) {};
\node[draw,circle,right=.8cm of D3] (DA) {};
\node[draw,circle,right=1cm of DA] (DB) {};
\node[draw,circle,right=.8cm of DB] (D4) {};
\node[draw,circle, above right=.8cm of D4] (D5) {};
\node[draw,circle,below right=.8cm of D4] (D6) {};

\draw (D1) -- (D3)  node[above,midway] {};
\draw (D2) -- (D3)  node[above,midway] {};
\draw (D3) -- (DA) node[above,midway] {};
\draw[ dotted] (DA) -- (DB);
\draw (D4) -- (DB) node[above,midway] {};
\draw (D4) -- (D5) node[above,midway] {};
\draw (D4) -- (D6) node[below,midway] {};
\end{tikzpicture}
\\
\\
$F_4$
&
\begin{tikzpicture}[thick,scale=\scalevalue, every node/.style={transform shape}]
\node[draw,circle] (F1) at (0,0) {};
\node[draw,circle,right=.8cm of F1] (F2) {};
\node[draw,circle,right=.8cm of F2] (F3) {};
\node[draw,circle,right=.8cm of F3] (F4) {};

\draw (F1) -- (F2)  node[above,midway] {};
\draw (F2) -- (F3)  node[above,midway] {$4$};
\draw (F3) -- (F4)  node[above,midway] {};
\end{tikzpicture}
&&
$\tilde{F}_4$
&
\begin{tikzpicture}[thick,scale=\scalevalue, every node/.style={transform shape}]
\node[draw,circle] (F1) at (0,0) {};
\node[draw,circle,right=.8cm of F1] (F2) {};
\node[draw,circle,right=.8cm of F2] (F3) {};
\node[draw,circle,right=.8cm of F3] (F4) {};
\node[draw,circle,right=.8cm of F4] (F5) {};

\draw (F1) -- (F2)  node[above,midway] {};
\draw (F2) -- (F3)  node[above,midway] {$4$};
\draw (F3) -- (F4)  node[above,midway] {};
\draw (F4) -- (F5)  node[above,midway] {};
\end{tikzpicture}
\\
\\
&
&&
$\tilde{G}_2$
&
\begin{tikzpicture}[thick,scale=\scalevalue, every node/.style={transform shape}]
\node[draw,circle] (HH1) at (0,0) {};
\node[draw,circle,right=.8cm of HH1] (HH2) {};
\node[draw,circle,right=.8cm of HH2] (HH3) {};

\draw (HH1) -- (HH2)  node[above,midway] {$6$};
\draw (HH2) -- (HH3)  node[above,midway] {};
\end{tikzpicture}
\\
\\
$E_6$

&
\begin{tikzpicture}[thick,scale=\scalevalue, every node/.style={transform shape}]
\node[draw,circle] (E1) at (0,0) {};
\node[draw,circle,right=.8cm of E1] (E2) {};
\node[draw,circle,right=.8cm of E2] (E3) {};
\node[draw,circle,right=.8cm of E3] (E4) {};
\node[draw,circle,right=.8cm of E4] (E5) {};
\node[draw,circle,below=.8cm of E3] (EA) {};

\draw (E1) -- (E2)  node[above,midway] {};
\draw (E2) -- (E3)  node[above,midway] {};
\draw (E3) -- (E4)  node[above,midway] {};
\draw (E4) -- (E5)  node[above,midway] {};
\draw (E3) -- (EA)  node[left,midway] {};
\end{tikzpicture}
&&
$\tilde{E}_6$

&
\begin{tikzpicture}[thick,scale=\scalevalue, every node/.style={transform shape}]
\node[draw,circle] (E1) at (0,0) {};
\node[draw,circle,right=.8cm of E1] (E2) {};
\node[draw,circle,right=.8cm of E2] (E3) {};
\node[draw,circle,right=.8cm of E3] (E4) {};
\node[draw,circle,right=.8cm of E4] (E5) {};
\node[draw,circle,below=.8cm of E3] (EA) {};
\node[draw,circle,below=.8cm of EA] (EB) {};

\draw (E1) -- (E2)  node[above,midway] {};
\draw (E2) -- (E3)  node[above,midway] {};
\draw (E3) -- (E4)  node[above,midway] {};
\draw (E4) -- (E5)  node[above,midway] {};
\draw (E3) -- (EA)  node[left,midway] {};
\draw (EB) -- (EA)  node[left,midway] {};
\end{tikzpicture}
\\
$E_7$

&
\begin{tikzpicture}[thick,scale=\scalevalue, every node/.style={transform shape}]
\node[draw,circle] (EE1) at (0,0) {};
\node[draw,circle,right=.8cm of EE1] (EE2) {};
\node[draw,circle,right=.8cm of EE2] (EE3) {};
\node[draw,circle,right=.8cm of EE3] (EE4) {};
\node[draw,circle,right=.8cm of EE4] (EE5) {};
\node[draw,circle,right=.8cm of EE5] (EE6) {};
\node[draw,circle,below=.8cm of EE3] (EEA) {};

\draw (EE1) -- (EE2)  node[above,midway] {};
\draw (EE2) -- (EE3)  node[above,midway] {};
\draw (EE3) -- (EE4)  node[above,midway] {};
\draw (EE4) -- (EE5)  node[above,midway] {};
\draw (EE5) -- (EE6)  node[above,midway] {};
\draw (EE3) -- (EEA)  node[left,midway] {};
\end{tikzpicture}
&&
$\tilde{E}_7$

&
\begin{tikzpicture}[thick,scale=\scalevalue, every node/.style={transform shape}]
\node[draw,circle] (EE1) at (0,0) {};
\node[draw,circle,right=.8cm of EE1] (EEB) {};
\node[draw,circle,right=.8cm of EEB] (EE2) {};
\node[draw,circle,right=.8cm of EE2] (EE3) {};
\node[draw,circle,right=.8cm of EE3] (EE4) {};
\node[draw,circle,right=.8cm of EE4] (EE5) {};
\node[draw,circle,right=.8cm of EE5] (EE6) {};
\node[draw,circle,below=.8cm of EE3] (EEA) {};

\draw (EE1) -- (EEB)  node[] {};
\draw (EE2) -- (EEB)  node[] {};
\draw (EE2) -- (EE3)  node[] {};
\draw (EE3) -- (EE4)  node[] {};
\draw (EE4) -- (EE5)  node[] {};
\draw (EE5) -- (EE6)  node[] {};
\draw (EE3) -- (EEA)  node[] {};
\end{tikzpicture}
\\
$E_8$

&
\begin{tikzpicture}[thick,scale=\scalevalue, every node/.style={transform shape}]
\node[draw,circle] (EEE1) at (0,0) {};
\node[draw,circle,right=.8cm of EEE1] (EEE2) {};
\node[draw,circle,right=.8cm of EEE2] (EEE3) {};
\node[draw,circle,right=.8cm of EEE3] (EEE4) {};
\node[draw,circle,right=.8cm of EEE4] (EEE5) {};
\node[draw,circle,right=.8cm of EEE5] (EEE6) {};
\node[draw,circle,right=.8cm of EEE6] (EEE7) {};
\node[draw,circle,below=.8cm of EEE3] (EEEA) {};

\draw (EEE1) -- (EEE2)  node[] {};
\draw (EEE2) -- (EEE3)  node[] {};
\draw (EEE3) -- (EEE4)  node[] {};
\draw (EEE4) -- (EEE5)  node[] {};
\draw (EEE5) -- (EEE6)  node[] {};
\draw (EEE6) -- (EEE7)  node[] {};
\draw (EEE3) -- (EEEA)  node[] {};
\end{tikzpicture}
&&
$\tilde{E}_8$

&
\begin{tikzpicture}[thick,scale=\scalevalue, every node/.style={transform shape}]
\node[draw,circle] (EEE1) at (0,0) {};
\node[draw,circle,right=.8cm of EEE1] (EEE2) {};
\node[draw,circle,right=.8cm of EEE2] (EEE3) {};
\node[draw,circle,right=.8cm of EEE3] (EEE4) {};
\node[draw,circle,right=.8cm of EEE4] (EEE5) {};
\node[draw,circle,right=.8cm of EEE5] (EEE6) {};
\node[draw,circle,right=.8cm of EEE6] (EEE7) {};
\node[draw,circle,right=.8cm of EEE7] (EEE8) {};
\node[draw,circle,below=.8cm of EEE3] (EEEA) {};

\draw (EEE1) -- (EEE2)  node[above,midway] {};
\draw (EEE2) -- (EEE3)  node[above,midway] {};
\draw (EEE3) -- (EEE4)  node[above,midway] {};
\draw (EEE4) -- (EEE5)  node[above,midway] {};
\draw (EEE5) -- (EEE6)  node[above,midway] {};
\draw (EEE6) -- (EEE7)  node[above,midway] {};
\draw (EEE8) -- (EEE7)  node[above,midway] {};
\draw (EEE3) -- (EEEA)  node[left,midway] {};
\end{tikzpicture}
\\
&&&&\\
\end{tabular}
\caption{The irreducible spherical Coxeter diagrams on the left and the irreducible affine Coxeter diagrams on the right.}
\label{table:spherical_and_affine_diagram}
\end{table}

\newpage

\section{Lannér Coxeter groups of rank $4$}\label{appendix:lanner}

We reproduce the list of all Lannér Coxeter diagrams of rank $4$ in Table \ref{Table_dim3}.

\begin{table}[h]
\centering
\begin{tabular}{cc cc cc}
\multicolumn{6}{c}{
\begin{tabular}{>{\centering\arraybackslash}m{.18\linewidth}>{\centering\arraybackslash}m{.18\linewidth}>{\centering\arraybackslash}m{.18\linewidth}>{\centering\arraybackslash}m{.18\linewidth}>{\centering\arraybackslash}m{.18\linewidth}}
\begin{tikzpicture}[thick,scale=0.6, every node/.style={transform shape}]
\node[draw,circle] (C1) at (0,0) {};
\node[draw,circle,right=.8cm of C1] (C2) {};
\node[draw,circle,below=.8cm of C2] (C3) {}; 
\node[draw,circle,left=.8cm of C3] (C4) {};  
\draw (C1) -- (C2)  node[above,midway] {$4$};
\draw (C2) -- (C3)  node[above,midway] {};
\draw (C3) -- (C4) node[above,midway] {};
\draw (C4) -- (C1) node[above,midway] {};
\end{tikzpicture}
&
%%%%%%%%%%%%%%%%%%%%%%%%%%%%%%%%%%%%%
 \begin{tikzpicture}[thick,scale=0.6, every node/.style={transform shape}]
\node[draw,circle] (D1) at (0,0) {};
\node[draw,circle,right=.8cm of D1] (D2) {};
\node[draw,circle,below=.8cm of D2] (D3) {}; 
\node[draw,circle,left=.8cm of D3] (D4) {};  
\draw (D1) -- (D2)  node[above,midway] {$5$};
\draw (D2) -- (D3)  node[above,midway] {};
\draw (D3) -- (D4) node[above,midway] {};
\draw (D4) -- (D1) node[above,midway] {};
\end{tikzpicture}
&
%%%%%%%%%%%%%%%%%%%%%%%%%%%%%%%%%%%%% 
 \begin{tikzpicture}[thick,scale=0.6, every node/.style={transform shape}]
\node[draw,circle] (E1) at (0,0) {};
\node[draw,circle,right=.8cm of E1] (E2) {};
\node[draw,circle,below=.8cm of E2] (E3) {}; 
\node[draw,circle,left=.8cm of E3] (E4) {};
\draw (E1) -- (E2)  node[above,midway] {$4$};
\draw (E2) -- (E3)  node[above,midway] {};
\draw (E3) -- (E4) node[below,midway] {$4$};
\draw (E4) -- (E1) node[above,midway] {};
\end{tikzpicture}
 &
 %%%%%%%%%%%%%%%%%%%%%%%%%%%%%%%%%%%%%
\begin{tikzpicture}[thick,scale=0.6, every node/.style={transform shape}]
\node[draw,circle] (F1) at (0,0) {};
\node[draw,circle,right=.8cm of F1] (F2) {};
\node[draw,circle,below=.8cm of F2] (F3) {}; 
\node[draw,circle,left=.8cm of F3] (F4) {};
\draw (F1) -- (F2)  node[above,midway] {$5$};
\draw (F2) -- (F3)  node[above,midway] {};
\draw (F3) -- (F4) node[below,midway] {$4$};
\draw (F4) -- (F1) node[above,midway] {};
\end{tikzpicture}
 %%%%%%%%%%%%%%%%%%%%%%%%%%%%%%%%%%%%%
 &
 \begin{tikzpicture}[thick,scale=0.6, every node/.style={transform shape}]
\node[draw,circle] (G1) at (0,0) {};
\node[draw,circle,right=.8cm of G1] (G2) {};
\node[draw,circle,below=.8cm of G2] (G3) {}; 
\node[draw,circle,left=.8cm of G3] (G4) {};
  \draw (G1) -- (G2)  node[above,midway] {$5$};
\draw (G2) -- (G3)  node[above,midway] {};
\draw (G3) -- (G4) node[below,midway] {$5$};
\draw (G4) -- (G1) node[above,midway] {};
\end{tikzpicture}
\\
\end{tabular}
}
\\
\multicolumn{6}{c}{
\begin{tabular}{>{\centering\arraybackslash}m{.22\linewidth}>{\centering\arraybackslash}m{.22\linewidth}>{\centering\arraybackslash}m{.22\linewidth}>{\centering\arraybackslash}m{.22\linewidth}}
\begin{tikzpicture}[thick,scale=0.6, every node/.style={transform shape}]
\node[draw,circle] (C1) at (0,0) {};
 \node[draw,circle,right=.8cm of C1] (C2) {};
\node[draw,circle,right=.8cm of C2] (C3) {}; 
\node[draw,circle,right=.8cm of C3] (C4) {};
\draw (C1) -- (C2)  node[above,midway] {};
\draw (C2) -- (C3)  node[above,midway] {$5$};
\draw (C3) -- (C4) node[] {};
\end{tikzpicture}
&
%%%%%%%%%%%%%%%%%%%%%%%%%%%%%%%%%%%%%
\begin{tikzpicture}[thick,scale=0.6, every node/.style={transform shape}]
\node[draw,circle] (CC1) at (0,0) {};
\node[draw,circle,right=.8cm of CC1] (CC2) {};
\node[draw,circle,right=.8cm of CC2] (CC3) {}; 
\node[draw,circle,right=1cm of CC3] (CC4) {};
\draw (CC1) -- (CC2)  node[above,midway] {$5$};
\draw (CC2) -- (CC3)  node[above,midway] {};
\draw (CC3) -- (CC4) node[above,midway] {$4$};
\end{tikzpicture}
%%%%%%%%%%%%%%%%%%%%%%%%%%%%%%%%%%%%%
 &
\begin{tikzpicture}[thick,scale=0.6, every node/.style={transform shape}]
\node[draw,circle] (D1) at (0,0) {};
\node[draw,circle,right=.8cm of D1] (D2) {};
\node[draw,circle,right=.8cm of D2] (D3) {}; 
\node[draw,circle,right=1cm of D3] (D4) {};
\draw (D1) -- (D2)  node[above,midway] {$5$};
\draw (D2) -- (D3)  node[above,midway] {};
\draw (D3) -- (D4) node[above,midway] {$5$};
\end{tikzpicture}
 &
\begin{tikzpicture}[thick,scale=0.6, every node/.style={transform shape}]
\node[draw,circle] (B4) at (0,0) {};
\node[draw,circle,right=.8cm of B4] (B5) {};
\node[draw,circle,above right=.8cm of B5] (B6) {};
\node[draw,circle,below right=.8cm of B5] (B7) {};
\draw (B4) -- (B5) node[above,midway] {$5$};
\draw (B5) -- (B6) node[above,midway] {};
\draw (B5) -- (B7) node[above,midway] {};
\end{tikzpicture}
\end{tabular}
}
\end{tabular}
\vspace{1mm}
\caption{Lannér Coxeter groups of rank $4$}
\label{Table_dim3}
\end{table}

\bigskip

\section{2-Lannér Coxeter groups of rank $\geqslant 5$}\label{classi_lorent}

We reproduce below the complete list of 2-Lannér Coxeter groups of rank $10$, $9$, $8$, $7$, $6$, $5$ respectively in Tables \ref{Table_dim9}, \ref{Table_dim8}, \ref{Table_dim7}, \ref{Table_dim6}, \ref{Table_dim5}, \ref{Table_dim4}. This list is extracted from \cite{MR679972,Chen:2013fk}.

\medskip

There is a bijection between the set of 2-Lannér Coxeter groups $W_S$ of rank $d+1$ and the set of irreducible, large, 2-perfect labeled simplex $\Ss_W$ of dimension $d$. Each facet of $\Ss_W$ corresponds to $s\in S$, i.e. a node of the Coxeter diagram $\mathcal{D}_W$. Since the polytope $\Ss_W$ is a simplex, each vertex $v$ has a unique opposite facet, hence corresponds to an element $s_v \in S$. The link of $\Ss_W$ at $v$ is isomorphic to $\Ss_{W_{S \smallsetminus \{ s_v\}} }$. Each node $s_v$ is colored in black, orange, blue, green, depending on the property of the link of $\Ss_W$ at $v$.

\medskip
 
\begin{center}
\begin{tabular}{|c| c | c | c | c|}
\hline
\begin{tikzpicture}[thick,scale=0.7, every node/.style={transform shape}] 
\node[draw,circle] (A) at (0,0){};
\end{tikzpicture}
&
\begin{tikzpicture}[thick,scale=0.7, every node/.style={transform shape}] 
\node[draw,circle,fill=black] (A) at (0,0){};
\end{tikzpicture}
&
\begin{tikzpicture}[thick,scale=0.7, every node/.style={transform shape}] 
\node[draw=orange, circle, fill=orange] (A) at (0,0){};
\end{tikzpicture}
&
\begin{tikzpicture}[thick,scale=0.7, every node/.style={transform shape}] 
\node[draw=RoyalBlue,circle, fill=RoyalBlue] (A) at (0,0){};
\end{tikzpicture}
&
\begin{tikzpicture}[thick,scale=0.7, every node/.style={transform shape}] 
\node[draw=green_okabe_ito,circle,fill=green_okabe_ito] (A) at (1,0){};
\end{tikzpicture}
\\
\hline
spherical & rigid affine & flexible affine &  rigid Lannér &  flexible Lannér \\
\hline
\end{tabular}
\end{center}

\bigskip
\bigskip
\bigskip

%%%%%%% Dim 9 %%%%%%%%%%%%%%
\begin{table}[ht!]
\centering
\begin{tabular}{|>{\centering\arraybackslash}m{.45\linewidth}>{\centering\arraybackslash}m{.45\linewidth}|}
\hline
&\\[-1em]
\multicolumn{2}{|c|}{
{\small rigid quasi-Lannér}
}
\\
&
\\ 
\multicolumn{2}{|c|}{
\begin{tikzpicture}[thick,scale=0.6, every node/.style={transform shape}] 
\node[draw, inner sep=5pt] (1) at (-1.3,-0.1) {{\Large ${W}_{91}$}};
\node[draw,circle,fill=black] (A) at (0,0){};
\node[draw,circle] (B) at (1,0){};
\node[draw,circle] (C) at (2,0){};
\node[draw,circle] (D) at (3,0){};
\node[draw,circle] (DD) at (4,0){};
\node[draw,circle] (EE) at (5,0){};
\node[draw,circle] (EEE) at (6,0){};
\node[draw,circle,fill=black] (E) at (7,0){};
\node[draw,circle] (F) at (2,-1){};
\node[draw,circle,fill=black] (G) at (6,-1){};
\draw (A) -- (B) node[above,midway] {};
\draw (B) -- (C) node[above,midway] {};
\draw (C) -- (D) node[above,midway] {};
\draw (DD) -- (D) node[above,midway] {};
\draw (DD) -- (EE) node[above,midway] {};
\draw (EE) -- (EEE) node[above,midway] {};
\draw (EEE) -- (E) node[above,midway] {};
\draw (C) -- (F) node[above,midway] {};
\draw (EEE) -- (G) node[above,midway] {};
\end{tikzpicture}
}
\\
&
\\
\hline
\end{tabular}
\vspace{1mm}
\caption{The diagrams of the 2-Lannér Coxeter groups of rank $10$}
\label{Table_dim9}
\end{table}

%%%%%%% Dim 8 %%%%%%%%%%%%%%
\begin{table}[ht!]
\centering
\begin{tabular}{|>{\centering\arraybackslash}m{.25\linewidth} |>{\centering\arraybackslash}m{.65\linewidth}|}

\hline
&\\[-1em]
{\small flexible quasi-Lannér} & {\small rigid quasi-Lannér}
\\
&
\\
&
\begin{tikzpicture}[thick,scale=0.6, every node/.style={transform shape}] 

\node[draw, inner sep=5pt] (1) at (-2.3,-0.1) {{\Large ${W}_{81}$}};
\node[draw,circle,fill=black] (A) at (0,0){};
\node[draw,circle] (B) at (1,0){};
\node[draw,circle] (C) at (2,0){};
\node[draw,circle] (D) at (3,0){};
\node[draw,circle] (DD) at (4,0){};
\node[draw,circle] (EE) at (5,0){};
\node[draw,circle] (E) at (6,0){};
\node[draw,circle] (F) at (2,-1){};
\node[draw,circle] (G) at (5,-1){};
\draw (A) -- (B) node[above,midway] {};
\draw (B) -- (C) node[above,midway] {};
\draw (C) -- (D) node[above,midway] {};
\draw (DD) -- (D) node[above,midway] {};
\draw (DD) -- (EE) node[above,midway] {};
\draw (EE) -- (E) node[above,midway] {};
\draw (C) -- (F) node[above,midway] {};
\draw (EE) -- (G) node[above,midway] {};
\end{tikzpicture}
\\
\hline
\end{tabular}
\vspace{1mm}
\caption{The diagrams of the 2-Lannér Coxeter groups of rank $9$}
\label{Table_dim8}
\end{table}

%%%%%%% Dim 7 %%%%%%%%%%%%%%
\begin{table}[ht!]
\centering
%
\vspace{1mm}
\caption{The diagrams of the 2-Lannér Coxeter groups of rank $8$}
\label{Table_dim7}
\end{table}

%%%%%%% Dim 6 %%%%%%%%%%%%%%
\begin{table}[ht!]
\centering
%
\vspace{1mm}
\caption{The diagrams of the 2-Lannér Coxeter groups of rank $5$}
\label{Table_dim4}
\end{table}

\clearpage
\section{Prisms in dimension 6, 7 and 8}\label{appendix:678prism}

Tables \ref{tab:excep_8prism}, \ref{tab:excep_7prism} and \ref{tab:excep_6prism} provide the Coxeter diagrams of convex-projectivizable, irreducible, large, $2$-perfect labeled prisms of dimension $d = 6,7,8$ (see Theorems \ref{t:final8}, \ref{t:final7} and \ref{t:final6}).

\bigskip

\begin{table}[h]
\begin{tabular}{cc cc cc}
\begin{tikzpicture}[thick,scale=0.6, every node/.style={transform shape}] 
\node[draw, inner sep=5pt] (1) at (-4.3,0.0) {{\Large $\Ss_{8\tau}^\dagger$}};
\node[draw,circle] (A) at (180:1){};
\node[draw,circle] (B) at (135:1){};
\node[draw,circle] (C) at (90:1){};
\node[draw,circle] (D) at (45:1){};
\node[draw,circle] (E) at (0:1){};
\node[draw,circle] (F) at (-45:1){};
\node[draw,circle] (G) at (-90:1){};
\node[draw,circle] (H) at (-135:1){};
\node[draw,circle] (I) at (180:2){};
\draw (A) -- (B) node[above,midway] {};
\draw (B) -- (C) node[above,midway] {};
\draw (C) -- (D) node[above,midway] {};
\draw (D) -- (E) node[above,midway] {};
\draw (E) -- (F) node[above,midway] {};
\draw (F) -- (G) node[above,midway] {};
\draw (H) -- (G) node[above,midway] {};
\draw (A) -- (H) node[above,midway] {};
\draw (I) -- (A) node[above,midway] {};

\node[draw,circle] (J) at (180:3){};
\draw (I) -- (J) node[above,midway] {$\infty$};
\end{tikzpicture}
&&
\begin{tikzpicture}[thick,scale=0.6, every node/.style={transform shape}]\node[draw, inner sep=5pt] (1) at (-4.3,0.0) {{\Large $\Ss_{8\tau}^\dagger \sharp_{\phi_1} \Ss_{8\tau}^\dagger$}};
\node[draw,circle] (A) at (180:1){};
\node[draw,circle] (B) at (135:1){};
\node[draw,circle] (C) at (90:1){};
\node[draw,circle] (D) at (45:1){};
\node[draw,circle] (E) at (0:1){};
\node[draw,circle] (F) at (-45:1){};
\node[draw,circle] (G) at (-90:1){};
\node[draw,circle] (H) at (-135:1){};
\node[draw,circle] (I) at (165:2.2){};
\draw (A) -- (B) node[above,midway] {};
\draw (B) -- (C) node[above,midway] {};
\draw (C) -- (D) node[above,midway] {};
\draw (D) -- (E) node[above,midway] {};
\draw (E) -- (F) node[above,midway] {};
\draw (F) -- (G) node[above,midway] {};
\draw (H) -- (G) node[above,midway] {};
\draw (A) -- (H) node[above,midway] {};
\draw (I) -- (A) node[above,midway] {};

\node[draw,circle] (J) at (195:2.2){};
\draw (I) -- (J) node[left,midway] {$\infty$};
\draw (J) -- (A) node[above,midway] {};
\end{tikzpicture}
&&
\begin{tikzpicture}[thick,scale=0.6, every node/.style={transform shape}] 
\node[draw, inner sep=5pt] (1) at (-4.3,0.0) {{\Large $\Ss_{8\tau}^\dagger \sharp_{\phi_2} \Ss_{8\tau}^\dagger$}};
\node[draw,circle] (A) at (180+22.5:1){};
\node[draw,circle] (B) at (135+22.5:1){};
\node[draw,circle] (C) at (90+22.5:1){};
\node[draw,circle] (D) at (45+22.5:1){};
\node[draw,circle] (E) at (0+22.5:1){};
\node[draw,circle] (F) at (-45+22.5:1){};
\node[draw,circle] (G) at (-90+22.5:1){};
\node[draw,circle] (H) at (-135+22.5:1){};
\node[draw,circle,left=.6cm of A] (I) {};
\draw (A) -- (B) node[above,midway] {};
\draw (B) -- (C) node[above,midway] {};
\draw (C) -- (D) node[above,midway] {};
\draw (D) -- (E) node[above,midway] {};
\draw (E) -- (F) node[above,midway] {};
\draw (F) -- (G) node[above,midway] {};
\draw (H) -- (G) node[above,midway] {};
\draw (A) -- (H) node[above,midway] {};
\draw (I) -- (A) node[above,midway] {};

\node[draw,circle,left=.6cm of B] (J) {};
\draw (I) -- (J) node[left,midway] {$\infty$};
\draw (J) -- (B) node[above,midway] {};
\end{tikzpicture}
%\\
%$\Ss_{8\tau}^\dagger$
%&&
%$\Ss_{8\tau} \sharp_{\phi_1} \Ss_{8\tau}$
%&&
%$\Ss_{8\tau} \sharp_{\phi_2} \Ss_{8\tau}$
\\
\end{tabular}
\vspace{1em}
\caption{The Coxeter diagrams of the 8-prims}
\label{tab:excep_8prism}
\end{table}

\begin{table}[h]
\begin{tabular}{cc cc cc}
\begin{tikzpicture}[thick,scale=0.6, every node/.style={transform shape}] 
\node[draw, inner sep=5pt] (1) at (-4.3,0.0) {{\Large $\Ss_{7\tau}^\dagger$}};
\node[draw,circle] (A) at (180:1){};
\node[draw,circle] (B) at (180-51:1){};
\node[draw,circle] (C) at (180-51-51:1){};
\node[draw,circle] (D) at (180-51-51-51:1){};
\node[draw,circle] (E) at (180-4*51:1){};
\node[draw,circle] (F) at (180-5*51:1){};
\node[draw,circle] (G) at (180-6*51:1){};
\node[draw,circle] (H) at (180:2){};
\draw (A) -- (B) node[above,midway] {};
\draw (B) -- (C) node[above,midway] {};
\draw (C) -- (D) node[above,midway] {};
\draw (D) -- (E) node[above,midway] {};
\draw (E) -- (F) node[above,midway] {};
\draw (F) -- (G) node[above,midway] {};
\draw (A) -- (G) node[above,midway] {};
\draw (H) -- (A) node[above,midway] {};

\node[draw,circle] (Z) at (180:3){};
\draw (H) -- (Z) node[above,midway] {$\infty$};
\end{tikzpicture}
& \phantom{A}  &
\begin{tikzpicture}[thick,scale=0.6, every node/.style={transform shape}] 
\node[draw, inner sep=5pt] (1) at (-4.3,0.0) {{\Large $\Ss_{7\tau}^\dagger \sharp_{\phi_1} \Ss_{7\tau}^\dagger$}};
\node[draw,circle] (A) at (180:1){};
\node[draw,circle] (B) at (180-51:1){};
\node[draw,circle] (C) at (180-51-51:1){};
\node[draw,circle] (D) at (180-51-51-51:1){};
\node[draw,circle] (E) at (180-4*51:1){};
\node[draw,circle] (F) at (180-5*51:1){};
\node[draw,circle] (G) at (180-6*51:1){};
\node[draw,circle] (H) at (195:2){};
\draw (A) -- (B) node[above,midway] {};
\draw (B) -- (C) node[above,midway] {};
\draw (C) -- (D) node[above,midway] {};
\draw (D) -- (E) node[above,midway] {};
\draw (E) -- (F) node[above,midway] {};
\draw (F) -- (G) node[above,midway] {};
\draw (A) -- (G) node[above,midway] {};
\draw (H) -- (A) node[above,midway] {};

\node[draw,circle] (Z) at (165:2){};
\draw (H) -- (Z) node[left,midway] {$\infty$};
\draw (A) -- (Z) node[left,midway] {};
\end{tikzpicture}
& \phantom{A}  &
\begin{tikzpicture}[thick,scale=0.6, every node/.style={transform shape}] 
\node[draw, inner sep=5pt] (1) at (-4.3,0.0) {{\Large $\Ss_{7\tau}^\dagger \sharp_{\phi_2} \Ss_{7\tau}^\dagger$}};
\node[draw,circle] (A) at (0:1){};
\node[draw,circle] (B) at (51:1){};
\node[draw,circle] (C) at (2*51:1){};
\node[draw,circle] (D) at (3*51:1){};
\node[draw,circle] (E) at (4*51:1){};
\node[draw,circle] (F) at (5*51:1){};
\node[draw,circle] (G) at (6*51:1){};
\node[draw,circle,left=0.6cm of D] (H) {};
\draw (A) -- (B) node[above,midway] {};
\draw (B) -- (C) node[above,midway] {};
\draw (C) -- (D) node[above,midway] {};
\draw (D) -- (E) node[above,midway] {};
\draw (E) -- (F) node[above,midway] {};
\draw (F) -- (G) node[above,midway] {};
\draw (A) -- (G) node[above,midway] {};
\draw (H) -- (D) node[above,midway] {};

\node[draw,circle,left=0.6cm of E] (Z) {};
\draw (H) -- (Z) node[left,midway] {$\infty$};
\draw (E) -- (Z) node[left,midway] {};
\end{tikzpicture}
%\\
%$\Ss_{7\tau}^\dagger$
%&&
%$\Ss_{7\tau} \sharp_{\phi_1} \Ss_{7\tau}$
%&&
%$\Ss_{7\tau} \sharp_{\phi_2} \Ss_{7\tau}$
\\
\end{tabular}
\vspace{1em}
\caption{The Coxeter diagrams of the 7-prisms}
\label{tab:excep_7prism}
\end{table}

\begin{table}[h]
\begin{tabular}{cc cc cc}
\begin{tikzpicture}[thick,scale=0.6, every node/.style={transform shape}] 
\node[draw, inner sep=5pt] (1) at (-4.3,0.0) {{\Large $\Ss_{6\tau}^\dagger$}};
\node[draw,circle] (A) at (180:1){};
\node[draw,circle] (B) at (120:1){};
\node[draw,circle] (C) at (60:1){};
\node[draw,circle] (D) at (0:1){};
\node[draw,circle] (E) at (-60:1){};
\node[draw,circle] (F) at (-120:1){};
\node[draw,circle] (G) at (180:2){};
\draw (A) -- (B) node[above,midway] {};
\draw (B) -- (C) node[above,midway] {};
\draw (C) -- (D) node[above,midway] {};
\draw (D) -- (E) node[above,midway] {};
\draw (E) -- (F) node[above,midway] {};
\draw (F) -- (A) node[above,midway] {};
\draw (G) -- (A) node[above,midway] {};

\node[draw,circle] (Z) at (180:3){};
\draw (G) -- (Z) node[above,midway] {$\infty$};
\end{tikzpicture}
&&
\begin{tikzpicture}[thick,scale=0.6, every node/.style={transform shape}] 
\node[draw, inner sep=5pt] (1) at (-4.3,0.0) {{\Large $\Ss_{6\tau}^\dagger \sharp_{\phi_1} \Ss_{6\tau}^\dagger$}};
\node[draw,circle] (A) at (180:1){};
\node[draw,circle] (B) at (120:1){};
\node[draw,circle] (C) at (60:1){};
\node[draw,circle] (D) at (0:1){};
\node[draw,circle] (E) at (-60:1){};
\node[draw,circle] (F) at (-120:1){};
\node[draw,circle] (G) at (195:2){};
\draw (A) -- (B) node[above,midway] {};
\draw (B) -- (C) node[above,midway] {};
\draw (C) -- (D) node[above,midway] {};
\draw (D) -- (E) node[above,midway] {};
\draw (E) -- (F) node[above,midway] {};
\draw (F) -- (A) node[above,midway] {};
\draw (G) -- (A) node[above,midway] {};

\node[draw,circle] (Z) at (165:2){};
\draw (G) -- (Z) node[left,midway] {$\infty$};
\draw (A) -- (Z) node[left,midway] {};
\end{tikzpicture}
&&
\begin{tikzpicture}[thick,scale=0.6, every node/.style={transform shape}] 
\node[draw, inner sep=5pt] (1) at (-4.3,0.0) {{\Large $\Ss_{6\tau}^\dagger \sharp_{\phi_2} \Ss_{6\tau}^\dagger$}};
\node[draw,circle] (A) at (150:1){};
\node[draw,circle] (B) at (90:1){};
\node[draw,circle] (C) at (30:1){};
\node[draw,circle] (D) at (-30:1){};
\node[draw,circle] (E) at (-90:1){};
\node[draw,circle] (F) at (-150:1){};
\node[draw,circle,left=0.6cm of A] (G){};
\draw (A) -- (B) node[above,midway] {};
\draw (B) -- (C) node[above,midway] {};
\draw (C) -- (D) node[above,midway] {};
\draw (D) -- (E) node[above,midway] {};
\draw (E) -- (F) node[above,midway] {};
\draw (F) -- (A) node[above,midway] {};
\draw (G) -- (A) node[above,midway] {};

\node[draw,circle,left=0.6cm of F] (Z) {};
\draw (G) -- (Z) node[left,midway] {$\infty$};
\draw (F) -- (Z) node[left,midway] {};
\end{tikzpicture}
%\\
%$\Ss_{6\tau}^\dagger$
%&&
%$\Ss_{6\tau} \sharp_{\phi_1} \Ss_{6\tau}$
%&&
%$\Ss_{6\tau} \sharp_{\phi_2} \Ss_{6\tau}$
\\
\end{tabular}
\vspace{1em}
\caption{The Coxeter diagrams of the 6-prisms}
\label{tab:excep_6prism}
\end{table}

\clearpage

\section{Exceptional prisms in dimension $5$}\label{appendix:5prism}

We collect below the Coxeter diagrams of convex-projectivizable, irreducible, large, 2-perfect labeled 5-prisms whose deformation space is disconnected.

\begin{table}[ht!]
\begin{tabular}{cc cc cc cc}
\begin{tikzpicture}[thick,scale=0.6, every node/.style={transform shape}] 
\node[draw, inner sep=5pt] (1) at (-1.0,2.0) {\Large $\Ss_{5\tau}^\dagger$};
\node[draw,circle] (G) at (180:3){};
\node[draw,circle] (A) at (180:1){};
\node[draw,circle] (B) at (108:1){};
\node[draw,circle] (C) at (36:1){};
\node[draw,circle] (D) at (-36:1){};
\node[draw,circle] (E) at (-108:1){};
\node[draw,circle] (F) at (180:2){};
\draw (A) -- (B) node[above,midway] {};
\draw (B) -- (C) node[above,midway] {};
\draw (C) -- (D) node[above,midway] {};
\draw (D) -- (E) node[above,midway] {};
\draw (E) -- (A) node[above,midway] {};
\draw (F) -- (A) node[above,midway] {};
\draw (F) -- (G) node[above,midway] {$\infty$};
\end{tikzpicture}
&&\quad
\begin{tikzpicture}[thick,scale=0.6, every node/.style={transform shape}] 
\node[draw, inner sep=5pt] (1) at (-0.6,2.0) {\Large $\Ss_{5\tau}^\dagger \sharp_{\phi_1} \Ss_{5\tau}^\dagger$};
\node[draw,circle] (A) at (180:1){};
\node[draw,circle] (B) at (108:1){};
\node[draw,circle] (C) at (36:1){};
\node[draw,circle] (D) at (-36:1){};
\node[draw,circle] (E) at (-108:1){};
\node[draw,circle] (F) at (-1.866,0.5){};
\node[draw,circle] (G) at (-1.866,-0.5){};
\draw (A) -- (B) node[above,midway] {};
\draw (B) -- (C) node[above,midway] {};
\draw (C) -- (D) node[above,midway] {};
\draw (D) -- (E) node[above,midway] {};
\draw (E) -- (A) node[above,midway] {};
\draw (F) -- (A) node[above,midway] {};
\draw (G) -- (A) node[above,midway] {};
\draw (F) -- (G) node[left,midway] {$\infty$};
\end{tikzpicture}
&&\quad
\begin{tikzpicture}[thick,scale=0.6, every node/.style={transform shape}] 
\node[draw, inner sep=5pt] (1) at (-0.5,2.0) {\Large $\Ss_{5\tau}^\dagger \sharp_{\phi_2} \Ss_{5\tau}^\dagger$};
\node[draw,circle] (A) at (144:1){};
\node[draw,circle] (B) at (72:1){};
\node[draw,circle] (C) at (0:1){};
\node[draw,circle] (D) at (-72:1){};
\node[draw,circle] (E) at (-144:1){};
\node[draw,circle] (F) at (-1.809,0.587){};
\node[draw,circle] (G) at (-1.809,-0.587){};
\draw (A) -- (B) node[above,midway] {};
\draw (B) -- (C) node[near end,above] {};
\draw (C) -- (D) node[right,midway] {};
\draw (D) -- (E) node[above,midway] {};
\draw (E) -- (A) node[above,midway] {};
\draw (F) -- (A) node[above,midway] {};
\draw (G) -- (E) node[above,midway] {};
\draw (F) -- (G) node[left,midway] {$\infty$};
\end{tikzpicture}
&&\quad
\begin{tikzpicture}[thick,scale=0.6, every node/.style={transform shape}] 
\node[draw, inner sep=5pt] (1) at (-0.6,2.0) {\Large $\Ss_{5\tau}^\dagger \sharp_{\phi_3} \Ss_{5\tau}^\dagger$};
\node[draw,circle] (A) at (180:1){};
\node[draw,circle] (B) at (108:1){};
\node[draw,circle] (C) at (36:1){};
\node[draw,circle] (D) at (-36:1){};
\node[draw,circle] (E) at (-108:1){};
\node[draw,circle] (F) at (-1.866,0.5){};
\node[draw,circle] (G) at (-1.866,-0.5){};
\draw (A) -- (B) node[above,midway] {};
\draw (B) -- (C) node[above,midway] {};
\draw (C) -- (D) node[above,midway] {};
\draw (D) -- (E) node[above,midway] {};
\draw (E) -- (A) node[above,midway] {};
\draw (F) -- (B) node[above,midway] {};
\draw (G) -- (E) node[above,midway] {};
\draw (F) -- (G) node[left,midway] {$\infty$};
\end{tikzpicture}
\\
\end{tabular}
\vspace{1em}
\caption{The Coxeter diagrams of the exceptional 5-prisms}
\label{tab:excep_prism}
\end{table}

\bibliographystyle{alpha}
%\bibliography{bibliography}

\begin{thebibliography}{BCL20b}

\bibitem[And70a]{MR0259734}
Evgeny.~M. Andreev.
\newblock Convex polyhedra in {L}oba\v cevski\u\i\ spaces.
\newblock {\em Mat. Sb. (N.S.)}, 81 (123):445--478, 1970.

\bibitem[And70b]{MR0273510}
Evgeny.~M. Andreev.
\newblock Convex polyhedra of finite volume in {L}oba\v cevski\u\i\ space.
\newblock {\em Mat. Sb. (N.S.)}, 83 (125):256--260, 1970.

\bibitem[Bal21]{Ballas}
Samuel~A. Ballas.
\newblock Constructing convex projective 3-manifolds with generalized cusps.
\newblock {\em J. Lond. Math. Soc. (2)}, 103(4):1276--1313, 2021.

\bibitem[Bar73]{Barnette}
David Barnette.
\newblock A proof of the lower bound conjecture for convex polytopes.
\newblock {\em Pacific J. Math.}, 46:349--354, 1973.

\bibitem[BCL20a]{BallasCooperLeitner}
Samuel~A. Ballas, Daryl Cooper, and Arielle Leitner.
\newblock Generalized cusps in real projective manifolds: classification.
\newblock {\em J. Topol.}, 13(4):1455--1496, 2020.

\bibitem[BCL20b]{BCL}
Samuel~A. Ballas, Daryl Cooper, and Arielle Leitner.
\newblock The moduli space of marked generalized cusps in real projective
  manifolds.
\newblock preprint, arXiv:2008.09553.

\bibitem[BDL18]{BDL_3d_geometrization}
Samuel Ballas, Jeffrey Danciger, and Gye-Seon Lee.
\newblock Convex projective structures on nonhyperbolic three-manifolds.
\newblock {\em Geometry {\&} Topology}, 22(3):1593--1646, 2018.

\bibitem[Bea83]{beardon_bible}
Alan~F. Beardon.
\newblock {\em The geometry of discrete groups}, volume~91 of {\em Graduate
  Texts in Mathematics}.
\newblock Springer-Verlag, New York, 1983.

\bibitem[Ben60]{benzecri}
Jean-Paul Benz{{\'e}}cri.
\newblock {Sur les vari{\'e}t{\'e}s localement affines et localement
  projectives}.
\newblock {\em Bull. Soc. Math. France}, 88:229--332, 1960.

\bibitem[Ben04]{convexe_div_1}
Yves Benoist.
\newblock {Convexes divisibles. {I}}.
\newblock In {\em {Algebraic groups and arithmetic}}, pages 339--374. Tata
  Inst. Fund. Res., Mumbai, 2004.

\bibitem[Ben06a]{cd4}
Yves Benoist.
\newblock {Convexes divisibles. {IV}. {S}tructure du bord en dimension 3}.
\newblock {\em Invent. Math.}, 164(2):249--278, 2006.

\bibitem[Ben06b]{benoist_qi}
Yves Benoist.
\newblock {Convexes hyperboliques et quasiisom{\'e}tries}.
\newblock {\em Geom. Dedicata}, 122:109--134, 2006.

\bibitem[Ben09]{MR2655311}
Yves Benoist.
\newblock Five lectures on lattices in semisimple {L}ie groups.
\newblock In {\em G\'eom\'etries \`a courbure n\'egative ou nulle, groupes
  discrets et rigidit\'es}, volume~18 of {\em S\'emin. Congr.}, pages 117--176.
  Soc. Math. France, Paris, 2009.

\bibitem[BM20]{BallasMarquis}
Samuel~A. Ballas and Ludovic Marquis.
\newblock Properly convex bending of hyperbolic manifolds.
\newblock {\em Groups Geom. Dyn.}, 14(2):653--688, 2020.

\bibitem[Bob19]{bobb_gcusps}
Martin~D. Bobb.
\newblock Convex projective manifolds with a cusp of any non-diagonalizable
  type.
\newblock {\em J. Lond. Math. Soc. (2)}, 100(1):183--202, 2019.

\bibitem[Bou68]{MR0240238}
Nicolas Bourbaki.
\newblock {\em \'{E}l\'ements de math\'ematique. {F}asc. {XXXIV}. {G}roupes et
  alg\`ebres de {L}ie. {C}hapitre {IV}: {G}roupes de {C}oxeter et syst\`emes de
  {T}its. {C}hapitre {V}: {G}roupes engendr\'es par des r\'eflexions.
  {C}hapitre {VI}: syst\`emes de racines}.
\newblock Actualit\'es Scientifiques et Industrielles, No. 1337. Hermann,
  Paris, 1968.

\bibitem[Br{\o}83]{MR683612}
Arne Br{\o}ndsted.
\newblock {\em An introduction to convex polytopes}, volume~90 of {\em Graduate
  Texts in Mathematics}.
\newblock Springer-Verlag, New York-Berlin, 1983.

\bibitem[Che69]{MR0294181}
Michel Chein.
\newblock Recherche des graphes des matrices de {C}oxeter hyperboliques d'ordre
  {$\leqslant \,10$}.
\newblock {\em Rev. Fran\c caise Informat. Recherche Op\'erationnelle}, 3(Ser.
  R-3):3--16, 1969.

\bibitem[CL15a]{Chen:2013fk}
Hao Chen and Jean-Philippe Labb{\'e}.
\newblock Lorentzian {C}oxeter systems and {B}oyd-{M}axwell ball packings.
\newblock {\em Geom. Dedicata}, 174:43--73, 2015.

\bibitem[CL15b]{MR3375519}
Suhyoung Choi and Gye-Seon Lee.
\newblock Projective deformations of weakly orderable hyperbolic {C}oxeter
  orbifolds.
\newblock {\em Geom. Topol.}, 19(4):1777--1828, 2015.

\bibitem[CLM18]{CLM_survey}
Suhyoung Choi, Gye-Seon Lee, and Ludovic Marquis.
\newblock Deformations of convex real projective manifolds and orbifolds.
\newblock In {\em Handbook of group actions. {V}ol. {III}}, volume~40 of {\em
  Adv. Lect. Math. (ALM)}, pages 263--310. Int. Press, Somerville, MA, 2018.

\bibitem[CLM20]{CLM2}
Suhyoung Choi, Gye-Seon Lee, and Ludovic Marquis.
\newblock Convex projective generalized {D}ehn filling.
\newblock {\em Ann. Sci. \'{E}c. Norm. Sup\'{e}r. (4)}, 53(1):217--266, 2020.

\bibitem[CLT15]{clt_cusps}
Daryl Cooper, Darren Long, and Stephan Tillmann.
\newblock {On convex projective manifolds and cusps}.
\newblock {\em Adv. Math.}, 277:181--251, 2015.

\bibitem[CLT18]{clt_koszul}
Daryl Cooper, Darren Long, and Stephan Tillmann.
\newblock Deforming convex projective manifolds.
\newblock {\em Geom. Topol.}, 22(3):1349--1404, 2018.

\bibitem[Dav08]{MR2360474}
Michael~W. Davis.
\newblock {\em The geometry and topology of {C}oxeter groups}, volume~32 of
  {\em London Mathematical Society Monographs Series}.
\newblock Princeton University Press, Princeton, NJ, 2008.

\bibitem[Gol77]{GoldmanThesis}
William~M. Goldman.
\newblock {\em Affine manifolds and projective geometry on surfaces}.
\newblock 1977.
\newblock Thesis (Bachelor)--The Princeton University.

\bibitem[Gol88]{bill_bible}
William~M. Goldman.
\newblock Geometric structures on manifolds and varieties of representations.
\newblock In {\em Geometry of group representations ({B}oulder, {CO}, 1987)},
  volume~74 of {\em Contemp. Math.}, pages 169--198. Amer. Math. Soc.,
  Providence, RI, 1988.

\bibitem[Gol90]{surf_gold}
William Goldman.
\newblock {Convex real projective structures on compact surfaces}.
\newblock {\em J. Differential Geom.}, 31(3):791--845, 1990.

\bibitem[Gol13]{bulging}
William~M. Goldman.
\newblock {Bulging deformations of convex $\mathbb{RP}^2$-manifolds}.
\newblock unpublished notes, arXiv:1302.0777.

\bibitem[Gre13]{MR3322033}
Ryan Greene.
\newblock {\em The deformation theory of discrete reflection groups and
  projective structures}.
\newblock ProQuest LLC, Ann Arbor, MI, 2013.
\newblock Thesis (Ph.D.)--The Ohio State University.

\bibitem[JM87]{johnson_millson}
Dennis Johnson and John~J. Millson.
\newblock {Deformation spaces associated to compact hyperbolic manifolds}.
\newblock In {\em {Discrete groups in geometry and analysis ({N}ew {H}aven,
  {C}onn., 1984)}}, volume~67 of {\em {Progr. Math.}}, page 48–106.
  Birkhäuser Boston, Boston, MA, 1987.

\bibitem[Kap07]{kapo_qi}
Michael Kapovich.
\newblock {Convex projective structures on {G}romov-{T}hurston manifolds}.
\newblock {\em Geom. Topol.}, 11:1777--1830, 2007.

\bibitem[Kle76]{MR0474043}
Peter Kleinschmidt.
\newblock Eine graphentheoretische {K}ennzeichnung der {S}tapelpolytope.
\newblock {\em Arch. Math. (Basel)}, 27(6):663--667, 1976.

\bibitem[Koe99]{koecher}
Max Koecher.
\newblock {\em {The {M}innesota notes on {J}ordan algebras and their
  applications}}, volume 1710 of {\em {Lecture Notes in Mathematics}}.
\newblock Springer-Verlag, 1999.

\bibitem[Kos67]{LectHypCoxGrKoszul}
Jean-Louis Koszul.
\newblock {\em Lectures on hyperbolic {C}oxeter groups}.
\newblock {U}niversity of {N}otre {D}ame, 1967.

\bibitem[Kos68]{kos_open}
Jean-Louis Koszul.
\newblock {Déformations de connexions localement plates}.
\newblock {\em Ann. Inst. Fourier (Grenoble)}, 18:103–114, 1968.

\bibitem[Lan50]{MR0042129}
Folke Lann{\'e}r.
\newblock On complexes with transitive groups of automorphisms.
\newblock {\em Comm. S\'em., Math. Univ. Lund [Medd. Lunds Univ. Mat. Sem.]},
  11:71, 1950.

\bibitem[Mak68]{MR0259735}
Vitaly~S. Makarov.
\newblock The {F}edorov groups of four-dimensional and five-dimensional
  {L}oba\v cevski\u\i\ space.
\newblock In {\em Studies in {G}eneral {A}lgebra, {N}o. 1 ({R}ussian)}, pages
  120--129. Ki\v sinev. Gos. Univ., Kishinev, 1968.

\bibitem[Mar10a]{MR2660566}
Ludovic Marquis.
\newblock Espace des modules de certains poly\`edres projectifs miroirs.
\newblock {\em Geom. Dedicata}, 147:47--86, 2010.

\bibitem[Mar10b]{marquis_moduli_surf}
Ludovic Marquis.
\newblock {Espace des modules marqu{\'e}s des surfaces projectives convexes de
  volume fini}.
\newblock {\em Geom. Topol.}, 14(4):2103--2149, 2010.

\bibitem[Mar17]{Marquis:2014aa}
Ludovic Marquis.
\newblock Coxeter group in {H}ilbert geometry.
\newblock {\em Groups Geom. Dyn.}, 11(3):819--877, 2017.

\bibitem[Max82]{MR679972}
George Maxwell.
\newblock Sphere packings and hyperbolic reflection groups.
\newblock {\em J. Algebra}, 79(1):78--97, 1982.

\bibitem[MV00]{MR1748082}
Grigori~A. Margulis and {\`E}rnest~B. Vinberg.
\newblock Some linear groups virtually having a free quotient.
\newblock {\em J. Lie Theory}, 10(1):171--180, 2000.

\bibitem[Nie15]{MR3449173}
Xin Nie.
\newblock On the {H}ilbert geometry of simplicial {T}its sets.
\newblock {\em Ann. Inst. Fourier (Grenoble)}, 65(3):1005--1030, 2015.

\bibitem[Thu97]{thurston_bible}
William~P. Thurston.
\newblock {\em Three-dimensional geometry and topology. {V}ol. 1}, volume~35 of
  {\em Princeton Mathematical Series}.
\newblock Princeton University Press, Princeton, NJ, 1997.
\newblock Edited by Silvio Levy.

\bibitem[Vin63]{homogeneous_vin_1}
Érnest Vinberg.
\newblock {The theory of homogeneous convex cones}.
\newblock {\em Trudy Moskov. Mat. Obv \v c.}, 12:303–358, 1963.

\bibitem[Vin71]{MR0302779}
{\`E}rnest~B. Vinberg.
\newblock Discrete linear groups that are generated by reflections.
\newblock {\em Izv. Akad. Nauk SSSR Ser. Mat.}, 35:1072--1112, 1971.

\bibitem[Vin85]{MR783604}
{\`E}rnest~B. Vinberg.
\newblock Hyperbolic groups of reflections.
\newblock {\em Uspekhi Mat. Nauk}, 40(1(241)):29--66, 255, 1985.

\end{thebibliography}

\end{document}